\pdfoutput=1
\documentclass[11pt]{article}

\usepackage[a4paper, margin=1in]{geometry}
\usepackage[utf8]{inputenc}
\usepackage{amsmath}
\usepackage{amssymb}
\usepackage{amsthm}
\usepackage{graphicx}
\usepackage{float}
\usepackage{authblk}
\usepackage{hyperref}
\usepackage{cleveref}
\hypersetup{
    colorlinks=true,
    linkcolor=black,
    urlcolor=blue,
    citecolor=black
}

\newtheorem{theorem}{Theorem}[section]
\newtheorem{proposition}{Proposition}[section]
\newtheorem{observation}{Observation}[section]

\theoremstyle{remark}
\newtheorem*{remarks}{Remarks}

\title{A Constructive Framework for Generalized Fourier Analysis via Truncate-and-Generalized Limits}

\author{Yoshihiko Akaiwa}
\affil{Emeritus professor of Kyushu University \\
1-21-2 Munakata-city, Fukuoka, Japan 811-3425 \\
\texttt{yakaiwa@gmail.com}}

\date{}

\begin{document}

\maketitle

\begin{abstract}
A new truncation-based definition of Fourier inversion is presented, extending the class of functions treated by Fourier analysis beyond the classical $L^1(\mathbb{R})$ setting and beyond the setting of tempered distributions. Generalized-function frameworks treat functions with rapid growth at infinity, such as $e^{t^2}$, by test-function spaces adapted to the growth, and a function with an essential local singularity, such as $e^{1/t}$, which is not locally integrable at $0^+$, admits no distributional extension. For both, the truncate-and-generalized-limit (t.g.l.) approach yields pointwise reconstruction, away from the singular points and under a local Dirichlet-type condition, by ordinary finite-stage Fourier integrals and ordered limiting operations, without analytic functionals or other generalized-function machinery; it assigns neither a canonical forward spectrum nor a value at the singular points.

A central feature of the formulation is that it abandons an assumption common, as far as we understand, to existing generalized Fourier theories: that the forward transform must exist, as a function or as a generalized function. Convergence of the forward Fourier transform is not required in any sense. Instead, the forward transform is retained as a family of ordinary finite-stage transforms, and inversion is performed through prescribed ordered limits, with frequency truncation removed before time truncation. For the practical purpose of the Fourier method---converting a time-domain problem into a frequency-domain problem and returning its result to the time domain by the inverse transform---this pointwise reconstruction is sufficient, even when no convergent forward spectrum is available. When a spectral interpretation is desired, a generalized spectral meaning may be obtained case by case through an appropriate pairing, independently of the reconstruction procedure. Such pairings also have a natural interpretation in applications, including the response of linear time-invariant systems.

The approach is motivated by the classical transition from Fourier series on finite intervals to Fourier integral expansions on the real line. Examples illustrate its application to non-decaying, rapidly growing, oscillatory, and locally singular functions. The presentation emphasizes elementary analysis and conceptual comparisons with existing generalized Fourier theories, particularly for readers in engineering and physics. The formulation is complementary to, rather than a replacement for, distribution theory.
\end{abstract}

\noindent \textbf{Keywords:} Fourier transform; generalized Fourier analysis; improper Riemann integrals; distribution theory; signal truncation; chirp signals; asymptotic reconstruction.

\noindent \textbf{MSC 2020:} 42A38, 42A20, 42A16, 44A35

\section{Introduction}

The Fourier transform is one of the fundamental tools in analysis, with wide applications in mathematics, physics, and engineering. In its classical form, it is defined for functions $f \in L^1(\mathbb{R})$ by
\begin{equation}
F(\omega) = \int_{-\infty}^{\infty} f(t)\, e^{-i\omega t}\,dt, \label{eq:1}
\end{equation}
\begin{equation}
f(t) = \frac{1}{2\pi}\int_{-\infty}^{\infty} F(\omega)\, e^{i\omega t}\,d\omega. \label{eq:2}
\end{equation}
These relations are called the Fourier transform and the inverse Fourier transform, respectively, and are often written symbolically as $f(t) \leftrightarrow F(\omega)$. To emphasize the inverse Fourier transform, in this paper we sometimes write $\overleftarrow{f}(t)$ instead of $f(t)$.

Since the two definitions depend on each other, an essential task is to check the consistency of the inversion formula: applying the inverse transform to $F$ must recover the original function $f$, as
\[
\overleftarrow{f}(t) = \frac{1}{2\pi}\int_{-\infty}^{\infty}\left[\int_{-\infty}^{\infty} f(x)\, e^{-i\omega x}\,dx\right] e^{i\omega t}\,d\omega = f(t).
\]

Because the integrations extend over infinite intervals in both the time and frequency domains, they must be interpreted through limiting procedures. Modern treatments usually formulate this interpretation within the Lebesgue integral framework. Earlier classical treatments instead relied on improper Riemann integrals, defined as limits of integrals over finite intervals.

If the Fourier transform pair is interpreted in the Lebesgue sense, then for functions $f \in L^1(\mathbb{R})$, the Fourier transform exists, since
\[
|F(\omega)| = \left|\int_{-\infty}^{\infty} f(t)\, e^{-i\omega t}\,dt\right| \leq \int_{-\infty}^{\infty} |f(t)|\,dt < \infty.
\]
Similarly, if $F \in L^1(\mathbb{R})$, the inverse transform exists. These conditions are sufficient in the Lebesgue framework, but they are rarely satisfied. They are not necessary, however, once the integrals are interpreted in the improper Riemann sense, through localization limits.

Many functions of practical and theoretical importance do not belong to $L^1(\mathbb{R})$ --- for example, constants, polynomials, periodic functions, and singular functions such as $1/t$ (Examples~1, 3, 4, and~7, respectively) --- for which the classical formulation is insufficient, though such functions nonetheless admit meaningful Fourier representations.

Historically, the interpretation of the integrals in Eqs.~(\ref{eq:1}) and (\ref{eq:2}) has changed together with the development of integration theory. The earliest spectral representations were Fourier series, in which functions on finite intervals were expanded in trigonometric bases with coefficients given by ordinary integrals over finite domains, and convergence was achieved through a limiting process applied outside the summation. The transition to Fourier integrals replaced discrete spectral variables by continuous frequency variables (see Appendix~A); in the nineteenth century the resulting expressions were interpreted as improper integrals, defined as limits of truncated integrals over finite intervals. Dirichlet and Riemann [1,2] proved convergence results showing that Fourier inversion could be justified through the localizing effect of truncated integration domains, and Titchmarsh [3] later gave a systematic analysis of Fourier integrals. Another classical approach [4] introduces a damping factor such as $\exp(-\sigma|t|)$, which restores absolute integrability for non-decaying functions; the generalized transform is recovered in the limit $\sigma\to0$, but locally singular functions such as $1/t$ are not regularized by such a factor.

With Lebesgue integration, the transform (\ref{eq:1}) was defined on $L^1(\mathbb{R})$, with inversion under integrability assumptions on both the function and its transform. Plancherel and Riesz [5,6] then extended it to $L^2(\mathbb{R})$, where the transform is a unitary operator and inversion holds in the mean-square sense even when absolute integrability fails.

A further extension came from the theory of generalized functions developed by Schwartz [7]. Schwartz distribution theory placed generalized Fourier transformation in a systematic functional-analytic framework by defining generalized functions through their action on suitable test-function spaces. This framework rigorously incorporates objects such as the Dirac delta and principal-value distributions and remains the standard general setting for many nonclassical Fourier transforms.

Other extensions developed along different lines. Titchmarsh [3] studied generalized Fourier integrals using complex-variable and contour methods, while Gel'fand and Shilov [8] developed a generalized-function theory, including an analytic-functional framework, that further enlarged the setting in which Fourier transformation could be interpreted; the two approaches are mathematically distinct, the former a complex-frequency extension of the ordinary integral, the latter a broader theory of generalized functions. Still broader theories, including ultradistributions and Sato's hyperfunctions, permit classes of objects beyond the tempered-distribution framework --- a point of particular relevance below, since this manuscript later considers functions whose growth may exceed that setting.

Constructive alternatives have also been developed. Lighthill [9] emphasized approximation by suitable sequences of ``good functions'' within the Schwartz class, while Feichtinger [10, 17, 19, 20] developed function-space and mild-distribution frameworks that retain close contact with classical Fourier operations. These approaches illustrate that generalized Fourier transformation can be organized in ways other than the standard Schwartz-distribution formulation.

Further difficulties arise for functions of the types $t^{n}$ and $e^{t^{n}}$ ($n$ an integer). For $n>0$ the difficulty is growth at infinity: $t^{n}$ is tempered, whereas $e^{t^{n}}$ grows exponentially for $n=1$ and super-exponentially for $n>1$, beyond the tempered-distribution setting, where, as far as we understand, the more advanced theories proceed by enlarging the test space. For a locally integrable function of any growth, a transform does exist as a functional on test functions whose Fourier images have compact support (\Cref{prop:weaklimit}), but this is the truncation of the t.g.l.\ formulation absorbed into the test space, and it is an abstract functional rather than an ordinary function; the essential singularity of $e^{1/t}$ at $0^+$ is not covered even there, since $e^{1/t}$ is not locally integrable. For $n<0$ the difficulty is a local singularity at the origin. For $n=-1$ the Cauchy principal value, and for $n<-1$ the Hadamard finite part~[11], assign a value, and both are incorporated into distribution theory through their action on test functions; but $e^{t^{n}}=e^{1/t^{|n|}}$ has an exponentially divergent singularity to which neither applies as it stands. Still more general theories --- ultradistributions, Sato's hyperfunctions, and Fourier hyperfunctions --- extend the distributional framework through more abstract functional-analytic and complex-analytic constructions; brief remarks on these appear where relevant below, alongside the t.g.l.\ formulation introduced next.

\medskip
\noindent
In all of these conventional approaches, the aim is to obtain a meaningful forward Fourier transform for functions that are not absolutely integrable: summability and damping-factor methods, good-function sequences, distributional pairings, contour-shift and analytic-continuation methods, and the finite-part prescriptions all assign a spectrum $F(\omega)$ to $f$, as a function, a distribution, or a boundary value, and the inverse transform is then defined from that spectrum. As far as we understand, the existence of such a forward transform is the common starting point and the criterion of success.

\medskip
\noindent\textit{Motivation.} The present formulation was motivated by a wish to construct generalized Fourier transforms directly from ordinary Fourier integrals and elementary limiting procedures, remaining within classical analysis at every finite stage rather than adopting the functional-analytic machinery of distribution theory from the outset. A preliminary version of this idea was reported earlier in~[12]: convergence of the forward Fourier transform was abandoned, and the inverse transform was taken as two successive limits, with a restriction imposed on how these limits are taken (a conditional limiting procedure), so that the inverse transform is interpreted beyond the classical absolute-integrability framework; the present paper develops this idea into a systematic, truncation-based generalized-limit formulation, and clarifies its relation to inversion kernels and generalized delta-function representations. The objective is not to replace distribution theory, which remains the more general framework, but to offer an alternative starting point: one that reaches generalized Fourier meaning through explicit, ordered limits alone.

\medskip
\noindent\textit{The t.g.l.\ Fourier transform, and its relation to other methods.} The present approach begins from the elementary observation that many functions for which the ordinary Fourier integral fails globally nevertheless possess ordinary Fourier transforms after finite truncation. We begin, accordingly, with truncated functions
\[
f_{T,\varepsilon}(t) = f(t)\cdot \mathbf{1}_{[-T,-\varepsilon]\cup[\varepsilon,T]}(t),
\]
where $\varepsilon$ denotes a truncation parameter around a singular point at $t=0$, and $T$ denotes truncation within a finite time domain $[-T,T]$; neither the principal-value nor the finite-part concept is introduced at this stage, only truncation itself. These truncated functions belong to $L^1(\mathbb{R})$, so the classical Fourier transform applies to them directly, and the original function is recovered afterward through a limiting process. We refer to this construction as the truncate-and-generalized-limit (t.g.l.) Fourier transform. The forward transform is defined for the truncated function, and a generalized limit is taken; the inverse transform is defined as a second-stage t.g.l.\ construction in the Fourier-dual domains $t$ and $\omega$:
\[
\begin{aligned}
f(t) &= \lim_{T\to\infty,\,\varepsilon\to0}\ \lim_{\Omega\to\infty}\ \frac{1}{2\pi}\int_{-\Omega}^{\Omega} F_{T,\varepsilon}(\omega)\, e^{i\omega t}\,d\omega \\
&= \lim_{T\to\infty,\,\varepsilon\to0}\ \lim_{\Omega\to\infty}\ \frac{1}{2\pi}\int_{-\Omega}^{\Omega}\left[\int_{-\infty}^{\infty} f_{T,\varepsilon}(x)\, e^{-i\omega x}\,dx\right] e^{i\omega t}\,d\omega \\
&= \lim_{T\to\infty,\,\varepsilon\to0}\ \lim_{\Omega\to\infty}\ \frac{1}{\pi}\int_{-\infty}^{\infty} f_{T,\varepsilon}(x)\, \frac{\sin\Omega(t-x)}{t-x}\,dx.
\end{aligned}
\]
The order of these limits is essential: at each fixed $T$, only ordinary, classically convergent Fourier operations are used, before the limit $T\to\infty$ is taken. The truncation affects only the divergent parts of $f$, namely its behavior at infinity and near the singular points, and as $T\to\infty$ and $\varepsilon\to0$ the truncated function approaches $f$ arbitrarily closely at every point of $\mathbb{R}\setminus S$, that is, up to the singular points themselves, while the operations on the frequency axis remain ordinary and rigorous at every finite stage.

This structure can also be viewed as a continuous extension of the passage from Fourier series to Fourier integrals: extending a finite trigonometric sum to the full series requires a limit in the frequency domain, while the passage from a periodic to a non-periodic setting requires a further limit in the time domain. The t.g.l.\ formulation adopts the same two-step structure as a constructive organizing principle; the resulting orthogonality and reconstruction structure, and the sense in which this limiting relation alone does not establish that the exponential system spans $L^2(\mathbb{R})$ in the full Hilbert-space sense, are developed in Section~3.3.

A central feature of the construction is the asymmetry between the forward and inverse operations. For fixed $T$, the truncated function $f_T$ has the ordinary Fourier transform $F_T$, and frequency truncation gives
\[
\frac{1}{2\pi}\int_{-\Omega}^{\Omega}F_T(\omega)e^{i\omega t}\,d\omega \longrightarrow f_T(t) \qquad (\Omega\to\infty)
\]
at points satisfying the usual local Dirichlet conditions. Only after this reconstruction is completed is the time truncation removed, yielding $f_T(t)\to f(t)$.

Consequently, pointwise convergence of $F_T(\omega)$ as $T\to\infty$ is not required for pointwise inversion. For the practical purpose of the Fourier method---converting a time-domain problem into a frequency-domain problem and returning its result to the time domain by the inverse transform---this pointwise reconstruction is sufficient, even when no convergent forward spectrum is available. When an independent generalized spectral interpretation is needed, it is introduced separately through the second-stage generalized limit described in Section~3. This separation between reconstruction and spectral interpretation is a basic structural feature of the t.g.l.\ formulation.

A natural instance of this separation arises in linear time-invariant systems: even when the forward transform of the input does not converge, pairing the finite-stage transform with the system's transfer function may converge, giving the generalized spectral interpretation a direct physical meaning as the system response (Sections~4.2 and~6.2). The same pairing underlies the classical distributional Fourier transform (Section~3.9): pairing a function with the transform of a test function is, in the LTI case, the same operation as applying the system to an input and evaluating its output.

The comparison with existing frameworks is developed in Sections~3.9 and~6. Distributional and sequential approaches enlarge the class of objects on which the Fourier transform acts and, where necessary, regularize singular functions; the t.g.l.\ formulation instead retains ordinary finite-stage Fourier integrals and treats non-decay and isolated singularities directly through truncation. The distinction claimed here concerns this constructive route, not the validity or generality of the established frameworks. Correspondingly, direct t.g.l.\ admissibility requires only Riemann integrability on compact intervals avoiding a locally finite set of isolated singularities, with no growth condition at infinity or boundedness condition at the excluded points; additional local regularity is required only when pointwise inversion is asserted (see Sections~2.2, 3.2, 3.9, and~6).

\medskip
\noindent\textit{Novelty, scope, and intended readership.}\footnote{A structural comparison table and a companion document on anticipated questions of novelty, rigor, and significance accompany this arXiv submission as ancillary files.} The novelty claimed here does not lie in truncation, limiting procedures, or finite-stage Fourier inversion individually, all of which are classical, but in the definition into which they are organized. The t.g.l.\ Fourier transform is not a limiting function $F(\omega)=\lim_{T\to\infty}F_T(\omega)$, which need not exist and is not part of the definition; it is the family $\{F_{T,\varepsilon}\}$ of ordinary finite-stage transforms, together with the prescribed ordered reconstruction. A conventional definition presupposes that the forward transform $F(\omega)=\int f(t)e^{-i\omega t}\,dt$ exists, as a function or as a distribution, and builds the inverse transform from it. Here no convergence of the forward family is required: the inverse transform is an ordinary finite integral at each stage, and the limits are taken outside it. In this way the inverse transform is realized within elementary calculus for functions such as $e^{\alpha t^2}$ and $e^{\alpha/t}$, which define no tempered distribution. Both kinds of divergence, growth at infinity and non-integrable local singularities, are treated by one operation: the function is restricted to a compact set on which it is bounded, $K=[-T_1,T_2]\setminus(\text{neighborhoods of }S)$, and set to zero elsewhere. The precise scope is stated in \Cref{prop:recon}: locally Riemann-integrable functions with locally finite singular sets and no growth condition. The construction agrees with the classical or distributional transform where one exists, and no canonical spectral object is claimed otherwise. Generalized spectral values arise only as derived, conditional pairings of the family with auxiliary functions (\Cref{thm:pairing}), separate both from the definition and from pointwise reconstruction. Thus the construction retains the finite-stage cycle
\[
f \longrightarrow f_T \longrightarrow F_T \longrightarrow f_T \longrightarrow f,
\]
without requiring the untruncated forward transform to exist as an intermediate object. Finite-stage spectra can also be used in calculations (Section~4); the contribution is an explicit organization of reconstruction and selected Fourier calculations, not a universal standalone limiting spectrum.

To the author's knowledge, this particular organization of time- and frequency-domain truncations as a defining procedure for generalized Fourier transformation has not previously been formalized in this way; the author has not found it in the literature, but does not claim that it has not appeared in some form. The claim is not priority for the individual operations or for transform pairs known by other methods, but for their systematic use in this ordered constructive framework. A possible reason why it seems not to have been formulated before is that, in the standard notion, a Fourier transform is a map whose forward value, a function or a distribution, is the object of study, so that convergence of the forward integral is part of the definition; once this requirement is dropped, the definition is elementary and may appear obvious in hindsight. For functions such as $e^{\alpha t}$ and $e^{\alpha/t}$, the point is likewise not that they were previously inaccessible, but that some cases conventionally associated with more advanced generalized-function frameworks can be treated by the same elementary real-variable construction.

The presentation is intentionally constructive and elementary, aimed particularly at engineers and physicists who routinely use Fourier methods but may not require the full machinery of generalized-function theory; it is not intended to replace distribution theory or other established generalized Fourier frameworks. The development therefore uses ordinary finite-stage integrals, explicit limiting procedures, and numerous examples. The ordered-limit structure and its relation to established theories may also be of mathematical interest, as an alternative, more elementary derivation of known results.

Sections~2 and~3 give the definitions and principal mathematical results. Section~4 applies the formulation to differential equations and linear systems, and Section~5 illustrates its scope through representative non-decaying, locally singular, rapidly growing, and oscillatory examples. Section~6 discusses the structural interpretation, scope, rigor, and relation to established generalized Fourier theories. Section~7 summarizes the conclusions.

\section{Truncation and Generalized Limits}

This section introduces the mathematical foundations of the truncate-and-generalized-limit formulation: the class of directly admissible ordinary functions, truncation of the target function, generalized limits defined through pairings, their relation to distributional convergence, and the integration-by-parts mechanism associated with truncation boundaries. These constructions provide the finite-stage framework used in Section~3 to define the t.g.l.\ forward and inverse Fourier transforms.

\subsection{Function Spaces and Growth Conditions}

Let $\mathcal{S}(\mathbb{R})$ denote the Schwartz space of rapidly decreasing smooth functions defined by the condition $|t^{k}\, d^{n}f(t)/dt^{n}| \le M_{k,n}$ $(k,n \ge 0)$. Its dual space $\mathcal{S}'(\mathbb{R})$ consists of tempered distributions. A locally integrable function $f$ defines a tempered distribution, for example, if it grows at most polynomially, i.e., there exist constants $C$ and $m$ such that $|f(t)| \le C(1+|t|)^{m}$. Although distributions based on compactly supported smooth test functions, $\mathcal{D}(\mathbb{R})=C_c^\infty(\mathbb{R})$, allow more general growth at infinity, $\mathcal{D}(\mathbb{R})$ is not invariant under Fourier transformation, so $\mathcal{D}'(\mathbb{R})$ does not by itself provide a closed forward-and-inverse Fourier-transform framework; the relation of this point to the t.g.l.\ construction is discussed in Section~3.8.

Many practically important functions, such as constants, polynomials, periodic functions, and chirp signals, belong to $\mathcal{S}'(\mathbb{R})$ and are covered by classical distribution theory. However, the t.g.l.\ formulation is not inherently restricted to this class. The finite-stage t.g.l.\ construction imposes no growth condition at infinity. Suppose that $f$ is Riemann integrable on every compact interval avoiding a locally finite set $S$ of isolated singular points. After finite outer truncation and exclusion of sufficiently small neighborhoods of the points of $S$ lying in the truncation interval, the retained function is Riemann integrable on a finite union of compact intervals. Consequently its finite-stage Fourier transform is an ordinary finite-domain integral, irrespective of the growth of $f$ as $|t|\to\infty$ or of the behavior at the excluded singular points. The precise condition is given in Section~2.2. This includes rapidly increasing functions such as $e^{\alpha t^2}$ $(\alpha > 0)$, which does not define a tempered distribution through the usual Schwartz-space pairing. It is directly t.g.l.-admissible and, being smooth at every finite point, also satisfies the local regularity required for pointwise inversion, and therefore admits pointwise inverse reconstruction via the t.g.l.\ Inversion Theorem (\Cref{thm:inversion}), as illustrated in Example~6.

Thus the finite-stage t.g.l.\ construction is not restricted by growth at infinity. Whether the forward-transform family possesses a pointwise or generalized spectral limit is a separate question, as is pointwise inverse reconstruction. These issues are treated in Sections~2.4 and~3.

Throughout the present formulation, finite-stage transforms are evaluated by ordinary integrals over the retained finite intervals. The generalized character enters through the subsequent limiting operations, not through replacement of the finite-stage integral by a different integration theory.

\subsection{The t.g.l.-Admissible Class}

\noindent\textbf{Definition 2.1 (t.g.l.-admissible function).} Let $S\subset\mathbb{R}$ be a locally finite set of isolated points. A function $f:\mathbb{R}\setminus S\to\mathbb{C}$ is called \textbf{t.g.l.-admissible} if, on every compact interval $I\subset\mathbb{R}\setminus S$, the restriction $f|_I$ is Riemann integrable (and hence bounded). No growth or boundedness condition is imposed as $|t|\to\infty$, and no boundedness condition is imposed at points of $S$.

For every finite outer truncation, only finitely many points of $S$ occur in the retained interval. After excluding sufficiently small neighborhoods of those points, the retained part of $f$ is Riemann integrable on a finite union of compact intervals. Hence its finite-stage Fourier transform is an ordinary finite-domain integral.

\emph{Remarks.} T.g.l.-admissibility is a finite-stage condition. It does not by itself assert (i) pointwise existence of the untruncated forward Fourier transform, (ii) existence of a generalized-limit pairing with an auxiliary function, or (iii) pointwise Fourier inversion. These are separate questions. Pairing-based generalized spectral meaning is treated in Section~2.4 and \Cref{thm:pairing}, while pointwise reconstruction requires an additional local Fourier-inversion condition at the reconstruction point, as stated in the t.g.l.\ Inversion Theorem.

Examples of t.g.l.-admissible ordinary functions include constants, polynomials, periodic functions, chirps, and rapidly increasing smooth functions such as $e^{\alpha t^2}$. Functions such as $1/t^n$ are t.g.l.-admissible on the complement of their isolated singular points. A Dirac delta distribution is not itself an ordinary t.g.l.-admissible function under this definition; delta-convergent families and smoothing constructions are extensions discussed separately later in the paper.

\subsection{Truncation of Functions}

Let $f$ be t.g.l.-admissible with singular set $S$. For finite $T_1,T_2>0$, local finiteness implies that $S\cap[-T_1,T_2]$ is finite. For sufficiently small $\varepsilon>0$, define
\[
D_{T,\varepsilon} = [-T_1,T_2]\setminus \bigcup_{a\in S\cap[-T_1,T_2]}(a-\varepsilon,a+\varepsilon),
\]
and
\[
f_{T,\varepsilon}(t)=f(t)\mathbf{1}_{D_{T,\varepsilon}}(t).
\]
Since $D_{T,\varepsilon}$ is a finite union of compact intervals separated from $S$, $f_{T,\varepsilon}$ is Riemann integrable and compactly supported. Throughout this paper, truncation means truncation of the target function, not truncation of the integration interval: $f_{T,\varepsilon}$ remains a function defined on $\mathbb{R}$, with its values set equal to zero outside the retained portions of the original target function. Its Fourier transform is consequently written $F_{T,\varepsilon}(\omega) = \int_{-\infty}^{\infty} f_{T,\varepsilon}(t)e^{-i\omega t}\,dt$. Although this integral is numerically identical to an integral of $f(t)e^{-i\omega t}$ over the retained portions, the former notation is used throughout because the object being truncated in the t.g.l.\ formulation is the target function itself. For computational convenience, however, this integral is frequently evaluated and displayed, throughout the examples of this paper, as an equivalent integral restricted to the retained intervals; such a split-interval expression is understood as a computational representation of $F_{T,\varepsilon}(\omega)$, not as a redefinition of the truncation itself. Its finite-stage Fourier transform therefore exists as an ordinary integral. For a single singularity at $t_0=0$, this reduces to the following: for functions singular at $t=t_0$, define the doubly truncated function for $T_1,T_2,\varepsilon_1$ and $\varepsilon_2>0$,
\begin{equation}
f_{T,\varepsilon}(t) =
\begin{cases}
f(t) & -T_1 \le t-t_0 \le -\varepsilon_1,\ \ \varepsilon_2 \le t-t_0 \le T_2 \\[4pt]
0 & \text{otherwise.}
\end{cases}
\label{eq:3}
\end{equation}
Then $f_{T,\varepsilon} \in L^1(\mathbb{R})$ for each finite $T_1, T_2, \varepsilon_1$ and $\varepsilon_2$. For simplicity we assume $t_0=0$ hereafter. If $f$ has no singular points in the retained interval, the parameters $\varepsilon_1,\varepsilon_2$ are omitted. The parameters may also be chosen asymmetrically. For example, $T_1=\varepsilon_1=\varepsilon_2=0$ (that is, a cut at $t=0$ and no exclusion) gives the one-sided finite-stage function $f(t)\mathbf{1}_{[0,T_2]}(t)$, to which the limit $T_2\to\infty$ is applied afterwards. The resulting causal function $f(t)\mathbf{1}_{[0,\infty)}(t)$ is the limit, not a finite stage; $T_2=\infty$ is never used inside a finite-stage integral.

Because the truncated function vanishes outside a finite interval, boundary discontinuities are introduced at truncation points. These discontinuities produce boundary contributions under ordinary integration by parts and, when interpreted through weak differentiation or delta-convergent families, correspond to Dirac-delta-type contributions. As a result, generalized spectral objects arise naturally as limits of truncated classical integrals rather than being introduced through distributional axioms.

The truncation employed throughout this paper is not restricted to bounding the domain variable of a fixed function. For a family of functions $f_x(t)$ depending on a parameter $x$ --- such as the delta-convergent sequences of Section~5.4 --- the same finite-stage sharp truncation applies to that parameter itself, restricting $x$ to a bounded range, or excluding it from a limiting value such as $0$, before the prescribed generalized limit is taken. Both the truncation of a fixed function's domain and the truncation of a parametrized family's parameter are instances of the same underlying t.g.l.\ mechanism, formalized generally in Section~2.4.

Two functions recur throughout this paper as special cases in which the truncation limit admits an additional, independent classical evaluation. The Gaussian function and the chirp signal $e^{i\alpha t^2}$ $(\alpha>0)$ are two such special cases, discussed in detail in Sections~5.4 and~6.1 and Appendix~C.

In physical systems, restricting a Fourier basis function such as $e^{i\omega t}$ to a finite interval $T$ corresponds, in practice, to a Fourier series expansion with discrete frequencies $n\omega_0=2\pi n/T$; since restating $T$ for every application is inconvenient, Fourier integrals are instead conventionally defined over the infinite domain. This restriction may be expressed equivalently through a limited integration interval or through a function set to zero outside the interval; the latter is the convention adopted above, since it offers the mathematical advantage of automatically cancelling otherwise troublesome boundary terms in the integration-by-parts technique of Section~2.6.

\subsection{Generalized Limit}

Papoulis~[13] uses the term \emph{generalized limit} for the limit of a sequence of generalized functions, indexed by a natural number. We extend this notion to a limiting parameter $x$ that need not be a natural number, applied not to $f_x$ itself but to a pairing integral formed from it.

Let $f_x(t)$ be a parameterized family of ordinary, locally integrable functions, and let $g(t)$ be an ordinary auxiliary function such that the product $f_x(t)g(t)$ is Riemann-integrable for each $x$. At each fixed $x$, the pairing
\[
I_x(g) = \int_{-\infty}^{\infty} f_x(t)\,g(t)\,dt
\]
is an ordinary integral. The parameter limit is taken only after this integration is completed:
\[
\operatorname{g.lim}_{x\to c}[f_x;g] := \lim_{x\to c} I_x(g).
\]
The family $f_x$ need not converge as a function. It is required only that the pairing $I_x(g)$ exist at each fixed $x$ and that its limit as $x\to c$ exist. Convergence is thus tested through the pairing, in the manner of weak convergence, and not for $f_x$ itself; this is the sense in which the limit is generalized. When the limit exists, its value depends in general on the choice of $g$ (Section~3.9). A single pairing therefore yields a number, not a generalized function. A generalized spectral map in the usual sense would require a common class of auxiliary functions $g$ on which all such pairings exist and are compatible, as for tempered distributions with Schwartz probes (Section~3.8); no such class is assumed here. A pairing value depending on $g$ is expected and is not by itself non-uniqueness. What would be needed is a fixed admissible class with a uniquely determined linear continuous action; the scalar generalized limit asserts neither, and its auxiliary function and the existence of the limit must be specified and proved.

When $f_x(t)$ and its transform pair $F_x(\omega)$ appear on two Fourier-dual domains, a second-stage generalized limit may likewise be formed on the dual side, by pairing $F_x$ against an auxiliary function $G$:
\[
\operatorname{g.lim}_{x\to c}[F_x;G] = \lim_{x\to c}\int_{-\infty}^{\infty} F_x(\omega)\,G^{*}(\omega)\,d\omega.
\]
Convergence of the first-stage limit is not required for the second-stage limit to exist: generalized meaning may be obtained through this second pairing even when $f_x(t)$ itself has no first-stage limit.

The generalized limit introduced here is a general notion, stated independently of the Fourier transform; it is used later for comparison with distributional convergence (Section~2.5) and with generalized spectral interpretations (Sections~3.8, 3.9). The t.g.l.\ Fourier transform itself, defined and proved in Section~3, is constructed directly from ordinary truncated integrals and ordered limits and does not invoke the pairing $\operatorname{g.lim}[f_x;g]$.

\subsection{Distributions as the Generalized Limit}

A tempered distribution $f \in \mathcal{S}'(\mathbb{R})$ acts on each test function $g \in \mathcal{S}(\mathbb{R})$ through a pairing
\[
\langle f, g \rangle = k_{f,g},
\]
where $k_{f,g}$ is a finite scalar depending on both $f$ and $g$. Thus distributions are defined through their action on test functions rather than through pointwise evaluation.

If a sequence of functions $f_n(t)$ satisfies
\[
\langle f, g \rangle = \lim_{n\to\infty} \int_{-\infty}^{\infty} f_n(t)\, g(t)\,dt = k_{f,g},
\]
then we write $f_n(t) \to f(t)$ in the sense of distribution.

Distributional convergence may be represented within the generalized-limit notation when the auxiliary functions belong to Schwartz space, and the relationship between the two deserves a precise statement. A formal similarity exists: for a locally integrable family $f_x$, the pairing $I_x(g) = \int_{\mathbb{R}} f_x(t)g(t)\,dt$ used in the definition of $\operatorname{g.lim}_{x\to c}[f_x;g]$ has the same form as the pairing $\langle T_x, \phi\rangle = \int_{\mathbb{R}} f_x(t)\phi(t)\,dt$ that defines the regular distribution $T_x$ associated with $f_x$, and $T_x \to T$ distributionally precisely when $\langle T_x,\phi\rangle \to \langle T,\phi\rangle$ for every $\phi \in \mathcal{S}(\mathbb{R})$.

The distinction lies in what is required of the auxiliary function. The t.g.l.\ generalized limit is defined by the ordinary scalar limit of $I_x(g)$ for one specified admissible $g$; neither a prescribed test-function topology nor continuity of a functional on such a space enters the definition.

Thus a single t.g.l.\ pairing is formally an instance of the same pairing mechanism that underlies distributional convergence, but it does not by itself constitute a tempered distribution. Only when the family of t.g.l.\ pairings is defined for every Schwartz function and is linear and continuous on $\mathcal{S}(\mathbb{R})$ does the resulting object coincide with an element of $\mathcal{S}'(\mathbb{R})$.

\subsection{The t.g.l. Integration by Parts}

\begin{theorem}\label{thm:s2one}
In the truncate-and-generalized-limit formulation, boundary terms arising in integration by parts are compensated by delta-function contributions generated by truncation discontinuities.
\end{theorem}

Assume a truncated function $f_{\varepsilon,T} \in L^1_{\mathrm{loc}}(\mathbb{R})$, which is piecewise $C^1$ with compact support. Integration by parts is therefore justified in the classical sense. This cancellation mechanism provides the boundary compensation used in the t.g.l.\ treatment of singular functions, without introducing principal-value or finite-part prescriptions at the finite stage.

\begin{proof}
The derivative of the truncated function $f_{T,\varepsilon}(t)$ (Eq.~(\ref{eq:3})), understood in the distributional sense at the points of discontinuity, with a singular point at $t=0$ is given as:
\[
f'_{T,\varepsilon}(t) =
\begin{cases}
f'(t) & -T_1 \le t \le -\varepsilon_1,\ \ \varepsilon_2 \le t \le T_2 \\[4pt]
0 & t<-T_1,\ \ -\varepsilon_1<t<\varepsilon_2,\ \ T_2<t
\end{cases}
\ +\ \delta_{\mathrm{trc}}(t), \qquad -\infty<t<\infty,
\]
where $T_1,T_2,\varepsilon_1$, and $\varepsilon_2>0$, and
\[
\delta_{\mathrm{trc}}(x) = f(-T_1)\,\delta(x+T_1) - f(-\varepsilon_1)\,\delta(x+\varepsilon_1) + f(\varepsilon_2)\,\delta(x-\varepsilon_2) - f(T_2)\,\delta(x-T_2).
\]
The infinite derivatives at the discontinuities caused by truncation are represented by the Dirac delta function.

The integration-by-parts relation used above may also be interpreted in terms of generalized-limit approximations of weak derivatives. For example, the derivative of the unit step function can be realized as the limit of derivatives of truncated linear-slope approximations
\[
H_\varepsilon(t) =
\begin{cases}
0 & t<-\varepsilon \\[2pt]
(t+\varepsilon)/2\varepsilon & -\varepsilon \le t \le \varepsilon \\[2pt]
1 & t>\varepsilon,
\end{cases}
\]
whose derivatives form a delta-convergent sequence. In this way the generalized-limit integration-by-parts formula provides a constructive realization of weak differentiation within the truncate-and-generalized-limit framework.

For simplicity of notation, we assume here symmetric truncation; $\varepsilon_1=\varepsilon_2=\varepsilon$ and $T_1=T_2=T$. Let $g(t)$ be a smooth function. Integration by parts gives
\begin{align*}
\int_{-\infty}^{\infty} f'_{\varepsilon,T}(t)\,g(t)\,dt
&= \left(\int_{-T}^{-\varepsilon} + \int_{\varepsilon}^{T}\right) f'(t)\,g(t)\,dt + \int_{-\infty}^{\infty} \delta_{\mathrm{trc}}(t)\,g(t)\,dt \\[4pt]
&= \big[f(t)g(t)\big]_{-T}^{-\varepsilon} - \int_{-T}^{-\varepsilon} f(t)\,g'(t)\,dt + \big[f(t)g(t)\big]_{\varepsilon}^{T} - \int_{\varepsilon}^{T} f(t)\,g'(t)\,dt \\[4pt]
&\quad + f(-T)g(-T) - f(-\varepsilon)g(-\varepsilon) + f(\varepsilon)g(\varepsilon) - f(T)g(T) \\[4pt]
&= -\int_{-\infty}^{\infty} f_{\varepsilon,T}(t)\,g'(t)\,dt,
\end{align*}
where the boundary terms cancel exactly with the delta-function contributions.
\end{proof}

\begin{remarks}
Thus the truncation-generated boundary contributions provide the compensation required for integration by parts and are consistent with the distributional identity
\[
\langle f',g\rangle=-\langle f,g'\rangle.
\]
They may be represented either by Dirac-delta notation or, within the elementary t.g.l.\ construction, by delta-convergent families of ordinary derivatives. For locally singular functions, conventional distributional treatments may additionally employ principal-value or finite-part regularization. In the t.g.l.\ formulation, the same sharp-truncation operation used for globally nonintegrable functions also excludes the local singularity, so no separate regularization mechanism is introduced at the finite stage.

For $1/t$, the symmetry requirement depends on the limiting operation. The inner exclusion about $t=0$ is taken symmetrically. For the forward transform, symmetry of the outer truncation is required at $\omega=0$, where oscillatory cancellation is absent, but not generally for $\omega\neq0$; see Example~7. In inverse reconstruction, asymmetric time truncation is allowed, while frequency truncation is taken symmetrically over $[-\Omega,\Omega]$ to obtain the Dirichlet-kernel localization used in \Cref{thm:inversion}.

Sharp truncation also preserves cancellation among boundary contributions at every finite stage. Individual terms may diverge as the truncation is removed even when their combined contribution has a well-defined generalized limit; they must therefore be combined before the truncation limit is taken.
\end{remarks}

\subsection{Distinction between Improper Integrals and Truncation-based Generalized Limits}

The truncate-and-generalized-limit formulation should be distinguished from the classical notion of improper integrals. In the conventional definition,
\[
\int_{-\infty}^{\infty} f(t)\,dt = \lim_{T\to\infty} \int_{-T}^{T} f(t)\,dt,
\]
provided the limit exists.

An improper integral assigns a value only when the corresponding family of truncated ordinary integrals converges as the truncation is removed. The t.g.l.\ construction retains the same finite-stage family $F_T(\omega) = \int_{-\infty}^{\infty} f_T(t)\, e^{-i\omega t}\,dt$, with $f_T(t)=f(t)\cdot\mathbf{1}_{[-T,T]}(t)\in L^1(\mathbb{R})$ (Section~2.3), even when this first limit does not exist; generalized meaning may then arise through a subsequent pairing and prescribed limiting operation (Section~2.4). Thus t.g.l.\ does not redefine a divergent improper integral as convergent; it uses the finite-stage quantities as inputs to a different limiting construction. Whenever the ordinary improper integral exists, however, no generalized interpretation is required, and the t.g.l.\ formulation reduces to the ordinary value.

This distinction becomes relevant for oscillatory, non-absolutely-integrable, and locally singular functions considered in Sections~4 and~5, where the first truncation limit may fail while a later generalized pairing remains meaningful; the associated boundary-cancellation and symmetry considerations are discussed in Section~2.6. The existence of a generalized limit after a particular pairing should not be confused with convergence of the original improper integral, and the truncation parameters are accordingly retained as finite until the operation for which they were introduced has been completed.

\section{Truncate-and-Generalized-Limit Fourier Transform}

In this section we define the t.g.l.\ Fourier transform and establish its inversion through explicitly ordered time- and frequency-domain truncations. At each finite time truncation, the Fourier transform is an ordinary finite-stage integral; in inversion, the frequency truncation is removed first, and the time truncation afterward. Finite-stage applicability is determined by t.g.l.-admissibility (Section~2.2), whereas pointwise reconstruction requires the additional local regularity of \Cref{thm:inversion}; these remain logically distinct throughout. Broader comparison with distributional and functional-analytic approaches is developed in Sections~3.9 and~6.

\subsection{Definitions}

The variables $t$ and $\omega$ form a Fourier-dual pair. If $t$ denotes time with unit second, then $\omega$ denotes angular frequency with unit rad/s, so that the phase quantity $\omega t$ is dimensionless (physically interpreted in radians), and the kernel $\exp(i\omega t)$ is well defined.

\medskip
\noindent \textbf{Definition 3.1} (The Forward Fourier Transform).
In the present formulation truncation is expressed by defining the truncated function $f_T(t)=f(t)\cdot \mathbf{1}_{[-T,T]}(t)$, where $f$ is t.g.l.-admissible (Definition~2.1).

For truncated $f_T \in L^1(\mathbb{R})$, we define the truncated Fourier transform:
\[
F_T(\omega) = \int_{-\infty}^{\infty} f_T(t)\, e^{-i\omega t}\,dt.
\]
Since $f_T \in L^1(\mathbb{R})$, $F_T$ is continuous and bounded. The family $F_T(\omega)$ is called the first-stage truncate-and-generalized-limit (t.g.l.) Fourier-transform family.

\begin{remarks}
See Section~2.6 for the boundary-delta mechanism this truncated-function form makes explicit.
\end{remarks}

Pointwise convergence of $F_T(\omega)$ as $T\to\infty$ is not required. Instead, generalized meaning is obtained through the second-stage generalized limit
\[
\lim_{T\to\infty} \frac{1}{2\pi}\int_{-\infty}^{\infty} F_T(\omega)\, G^{*}(\omega)\,d\omega,
\]
where $G(\omega)$ is an auxiliary function on the Fourier-dual domain.

The t.g.l.\ inverse Fourier transform defined below is a second-stage generalized-limit construction, obtained by choosing
\[
G_{\Omega,t}(\omega) = G_t(\omega)\,\mathbf{1}_{[-\Omega,\Omega]}(\omega) = e^{-i\omega t}\, \mathbf{1}_{[-\Omega,\Omega]}(\omega),
\]
together with an additional limiting operation $\Omega\to\infty$ associated with truncation of the Fourier-dual domain: at each fixed time-domain truncation $T$, the inverse integral is first truncated to $[-\Omega,\Omega]$ and the limit $\Omega\to\infty$ is taken to obtain pointwise reconstruction of the finite-stage function $f_T$, after which the generalized truncation limit $T\to\infty$ is taken to recover $f$,
\[
\overleftarrow{f}(t) = \lim_{T\to\infty} \overleftarrow{f}_T(t) = \lim_{T\to\infty}\lim_{\Omega\to\infty} \frac{1}{2\pi}\int_{-\Omega}^{\Omega} F_T(\omega)\, e^{i\omega t}\,d\omega.
\]
Thus the inverse t.g.l.\ construction necessarily involves two limiting operations, with the prescribed order $\Omega\to\infty$ first and $T\to\infty$ afterward; this is precisely the second-stage generalized limit in the Fourier-dual domains $t$ and $\omega$. There is no contradiction between the two descriptions: ``second-stage'' describes the structural level of the construction (Section~2.4), while the two limiting operations describe what is required to complete it.

For $f \notin L^1(\mathbb{R})$ in general, the family $F_{T,\varepsilon}(\omega) = \int_{-\infty}^{\infty} f_{T,\varepsilon}(t)\,e^{-i\omega t}\,dt$ is called the first-stage t.g.l.\ Fourier-transform family (Section~2.4). When its pointwise limit exists, the truncate-and-generalized-limit (t.g.l.) Fourier transform $F$ is defined by
\begin{equation}
F(\omega) = \lim_{T_1,T_2\to\infty}\ \lim_{\varepsilon_1,\varepsilon_2\to0} F_{T,\varepsilon}(\omega),
\label{eq:4}
\end{equation}
within the t.g.l.\ framework; when this pointwise limit does not exist, generalized spectral meaning is instead assigned through the second-stage generalized limit of \Cref{thm:pairing}.

In every case the t.g.l.\ Fourier transform of $f$, as an object of the theory, is the family $\{F_{T,\varepsilon}\}$ itself. The function $F(\omega)$ of Eq.~(\ref{eq:4}), when the pointwise limit exists, and the pairing limit of \Cref{thm:pairing}, otherwise, are quantities read off from this family; neither is needed for the inverse transform of Definition~3.2, which is formed from $\{F_{T,\varepsilon}\}$ alone.

Hereafter we denote $\lim_{T_1,T_2\to\infty}\lim_{\varepsilon_1,\varepsilon_2\to0}$ as $\lim_{T\to\infty,\,\varepsilon\to0}$ for simplicity of expression.

\medskip
\noindent \textbf{Definition 3.2} (Inverse Fourier Transform).
For fixed $T_1, T_2, \varepsilon_1$ and $\varepsilon_2$, define
\[
\overleftarrow{f}_{T,\varepsilon}(t) = \lim_{\Omega\to\infty} \frac{1}{2\pi}\int_{-\Omega}^{\Omega} F_{T,\varepsilon}(\omega)\, e^{i\omega t}\,d\omega.
\]
The t.g.l.\ inverse transform is defined by
\begin{equation}
\begin{aligned}
f(t) &= \lim_{T\to\infty,\,\varepsilon\to0} f_{T,\varepsilon}(t) = \lim_{T\to\infty,\,\varepsilon\to0} \overleftarrow{f}_{T,\varepsilon}(t) \\
&= \lim_{T\to\infty,\,\varepsilon\to0}\ \lim_{\Omega\to\infty}\ \frac{1}{2\pi}\int_{-\Omega}^{\Omega} F_{T,\varepsilon}(\omega)\, e^{i\omega t}\,d\omega \\
&= \lim_{T\to\infty,\,\varepsilon\to0}\ \lim_{\Omega\to\infty}\ \frac{1}{2\pi}\int_{-\Omega}^{\Omega}\left[\int_{-\infty}^{\infty} f_{T,\varepsilon}(x)\, e^{-i\omega x}\,dx\right] e^{i\omega t}\,d\omega.
\end{aligned}
\label{eq:5}
\end{equation}
At a point of continuity the left-hand side is $f(t)$; at a jump it is to be read as $\tfrac12[f(t^-)+f(t^+)]$ (\Cref{thm:inversion}).

The essential structural feature is that the limit $T_1,T_2\to\infty,\ \varepsilon_1,\varepsilon_2\to0$ is taken after $\Omega\to\infty$ for the integral. This order reflects faithfully the transition from the Fourier series to integral for non-integrable functions (refer to Appendix A). This order yields the pointwise reconstruction established in \Cref{thm:inversion} for t.g.l.-admissible functions satisfying the stated local inversion condition. The order of limit $T_1,T_2\to\infty$ and $\varepsilon_1,\varepsilon_2\to0$ is interchangeable, since the truncation intervals do not overlap. If $f$ is not singular, the notation $\varepsilon_1,\varepsilon_2$ is omitted.

Symmetric frequency truncation plays an essential role in the present Dirichlet-kernel proof of pointwise inversion. When the frequency truncation is taken over symmetric intervals $[-\Omega,\Omega]$, the resulting kernel is proportional to $\sin(\Omega x)/\pi x$, which forms an oscillatory Dirichlet localization family and converges to the Dirac delta in the distributional sense.

Unlike the ordinary integral formulation, where the inverse is applied to a forward transform already defined as an ordinary function, the t.g.l.\ construction does not require the first-stage t.g.l.\ Fourier-transform family $F_T(\omega)$ (Section~2.4) to converge pointwise to a first-stage generalized limit $\operatorname{g.lim}_{T\to\infty}[f_T;e^{-i\omega t}]$: pointwise reconstruction is instead determined within the truncation-limit structure itself, extending to an explicit dual-domain ordered limit the framework of~[12], in which convergence of the forward transform is abandoned and the inverse transform is obtained by two limits under a prescribed restriction --- itself a second-stage generalized limit paired against $e^{i\omega t}$.

\subsection{The t.g.l.\ Inversion Theorem}

The t.g.l.-admissibility condition of Definition~2.1 guarantees the existence of the ordinary finite-stage Fourier transform after the prescribed truncations. Pointwise reconstruction is a separate issue and requires an additional condition only at the point being reconstructed. We therefore state the inversion theorem directly for a t.g.l.-admissible function satisfying a classical local Fourier-inversion condition.

\begin{theorem}[t.g.l.\ Inversion Theorem]\label{thm:inversion}
Let $f$ be t.g.l.-admissible with locally finite singular set $S$, and let $t\notin S$. Suppose that $f$ satisfies a classical local Dirichlet condition at $t$; for example, $f$ is of bounded variation in some neighborhood of $t$. Then, with the prescribed ordered limits,
\[
\lim_{T\to\infty}\ \lim_{\varepsilon\to 0}\ \lim_{\Omega\to\infty}\ \frac{1}{2\pi}\int_{-\Omega}^{\Omega} F_{T,\varepsilon}(\omega)\,e^{i\omega t}\,d\omega = \frac{f(t^-)+f(t^+)}{2}.
\]
At a continuity point $t$, this value is $f(t)$. If no singular points occur in the relevant finite truncation, the $\varepsilon$-limit is absent, and $F_{T,\varepsilon}$ reduces to $F_T(\omega)=\int_{-\infty}^{\infty}f_T(x)e^{-i\omega x}\,dx$ with $f_T=f\cdot\mathbf{1}_{[-T,T]}$. The order of limits is essential and cannot in general be reversed for $f\notin L^1(\mathbb{R})$.
\end{theorem}

The order of limits is part of the definition of the construction. The frequency cutoff is removed first at fixed finite-stage truncation; local exclusions are then removed where present; the outer time truncation is removed last. Here $T\to\infty$ and $\varepsilon\to0$ are compact notation for the prescribed outer-truncation and local-exclusion limits introduced in Sections~2.3 and~3.1; the individual parameters $T_1,T_2,\varepsilon_1,\varepsilon_2$ need not be taken symmetrically except where symmetry is explicitly required.

A concrete instance of this order-dependence is furnished by the constant function $f(t)=1$ (Example~1), whose truncated forward transform is $F_T(\omega) = 2\sin\omega T/\omega$. In the prescribed order, $\Omega\to\infty$ is taken first, at each fixed $T$, giving $\overleftarrow{f}_T(t)=1$ for $|t|<T$ by the classical Dirichlet-kernel argument; the subsequent limit $T\to\infty$ then gives $f(t)=1$, as required. If instead $T\to\infty$ is attempted first, the expression $F_T(\omega)$ fails to possess even the ordinary-function meaning needed for the subsequent $\Omega\to\infty$ step: for each fixed $\omega\neq0$, $\sin\omega T$ oscillates without limit as $T\to\infty$, so $\lim_{T\to\infty}F_T(\omega)$ does not exist pointwise, while at $\omega=0$ the value $F_T(0)=2T$ diverges. The reversed order therefore breaks down already at its first step, not through some later analytic subtlety. A related but more delicate failure is exhibited by $f(t)=t$ (Example~3) in the proof of \Cref{thm:noninterchange} below: there, $\lim_{T\to\infty}F_T(\omega)$ does survive as a generalized object, so the reversed order proceeds one step further before breaking down, only failing when the corresponding boundary term is found not to converge as $T\to\infty$.

\begin{proof}
Fix $t\notin S$. Since $S$ is locally finite and $t$ is nonsingular, there is a neighborhood $U$ of $t$ containing no point of $S$. For sufficiently large outer truncation and sufficiently small exclusion parameter, $f_{T,\varepsilon}$ agrees with $f$ on $U$. At fixed $T,\varepsilon$, $f_{T,\varepsilon}$ is compactly supported and Riemann integrable, so its Fourier transform is an ordinary finite-stage transform. The symmetric frequency-truncated inverse is the corresponding Dirichlet-kernel integral. Under the assumed local Dirichlet condition at $t$, the classical pointwise Fourier inversion theorem (see e.g.\ [3]) gives the midpoint value $(f(t^-)+f(t^+))/2$. Since subsequent removal of exclusions and outer truncation does not alter $f$ in a neighborhood of $t$, the later limits preserve this value. A detailed proof for continuously differentiable $f$ using the Riemann--Lebesgue lemma explicitly is given in \Cref{thm:c1}.
\end{proof}

Thus the hypotheses have two logically distinct roles: finite-stage Riemann admissibility away from $S$, together with local Dirichlet regularity at $t$. The first guarantees existence of the finite-stage transform; the second guarantees pointwise reconstruction.

\begin{remarks}
The theorem shows that the t.g.l.\ inversion formula is a rigorous statement within classical analysis, requiring only the classical Dirichlet--Jordan theorem and ordinary finite-stage integration. No functional-analytic machinery --- no test-function duality and no Banach space theory --- is invoked in this statement. Because t.g.l.-admissibility imposes no growth condition at infinity, it includes not only constants, polynomials, and periodic functions, but also rapidly increasing functions such as $e^{\alpha t^2}$ $(\alpha > 0)$ (Example~6), which lie outside the space of tempered distributions entirely. The directly t.g.l.-admissible class treated in this paper includes constants, polynomials, periodic functions, rapidly increasing functions, singular kernels $1/t^n$, and chirp signals; delta-convergent sequences such as those of Examples~9 and~10 are treated separately through the generalized-limit extension, consistent with Definition~2.1.

As emphasized in Section~2.2, direct admissibility, pairing-based spectral meaning, and pointwise reconstruction remain logically distinct; see the Remarks following Definition~2.1.
\end{remarks}

\subsection{Orthogonality, the Dirichlet Kernel, and Localization}

The orthogonality of complex exponential functions on finite intervals forms the foundation of Fourier series expansions. On the interval $[-T,T]$, exponential functions satisfy the classical orthogonality relation
\[
\frac{1}{2T}\int_{-T}^{T} e^{i\omega_n t}\, e^{-i\omega_m t}\,dt =
\begin{cases}
1 & n=m \\
0 & n \ne m,
\end{cases}
\]
where $\omega_k = k\pi/T$ $(k=0,\pm1,\pm2,\dots)$.

When passing from Fourier series to Fourier integrals, this relation extends naturally to continuous frequencies $\omega,\nu\in\mathbb{R}$ through the classical/Bohr mean orthogonality
\[
\frac{1}{2T}\int_{-T}^{T} e^{i\omega t}\, e^{-i\nu t}\,dt =
\begin{cases}
1,&\omega=\nu,\\[2pt]
\dfrac{\sin[(\omega-\nu)T]}{(\omega-\nu)T},&\omega\ne\nu,
\end{cases}
\qquad
\lim_{T\to\infty} \frac{1}{2T}\int_{-T}^{T} e^{i\omega t}\, e^{-i\nu t}\,dt =
\begin{cases}
1 & \omega=\nu \\
0 & \omega \ne \nu.
\end{cases}
\]
It is not literally true that the complex exponentials $\{e^{i\omega t}\}_{\omega\in\mathbb{R}}$ form an orthonormal basis of $L^2(\mathbb{R})$, since no single exponential is normalizable; rather, this mean orthogonality relation --- related to the classical Bohr mean value $M\{f\} = \lim_{T\to\infty}\frac{1}{2T}\int_{-T}^{T}f(t)\,dt$ of almost-periodic analysis [18] --- supplies the generalized orthogonality structure underlying Fourier analysis on which the t.g.l.\ inversion construction is built.

Closely related is the Dirichlet kernel
\[
D_\Omega(t-x) = \frac{1}{2\pi}\int_{-\Omega}^{\Omega} e^{i\omega(t-x)}\,d\omega = \frac{\sin[\Omega(t-x)]}{\pi(t-x)},
\]
obtained by substituting $F(\omega)=\int f(t)e^{-i\omega t}\,dt$ into the frequency-truncated inverse integral $\frac{1}{2\pi}\int_{-\Omega}^{\Omega}F(\omega)e^{i\omega t}\,d\omega$ and interchanging the order of integration. As $\Omega\to\infty$, $D_\Omega(t-x)$ becomes increasingly concentrated near $x=t$, providing an oscillatory localization --- rather than a positive approximate identity --- through which Fourier inversion reconstructs the local value of $f$ at $t$. The distributional shorthand $D_\Omega\to\delta$ may be noted, but the t.g.l.\ inversion proof (\Cref{thm:inversion}) uses the finite-$\Omega$ kernel directly, without introducing the Dirac delta at this stage.

In the t.g.l.\ construction, truncation is applied to the function rather than to the exponential functions $e^{i\omega t}$. The latter remain defined on $\mathbb{R}$, while the finite-stage truncation makes the relevant integrals ordinary ones. Oscillatory cancellation and the resulting Dirichlet-kernel localization are therefore retained throughout the limiting process; their broader structural role is discussed in Section~6.

\subsection{Theorems, Proofs, Propositions and Observations}

\begin{proposition}[Generalized spectral pairing of the t.g.l.\ Fourier-transform family]\label{thm:pairing}
For a real function $f$ with locally finite singular set $S$ and its doubly truncated function $f_{T,\varepsilon}(t)$ (Section~2.3), let $f_{T,\varepsilon} \in L^1(\mathbb{R})$ for every finite $T,\varepsilon$, and denote its Fourier transform by $F_{T,\varepsilon}(\omega)$. Let $G\in L^1_{\mathrm{loc}}(\mathbb{R})$ be an auxiliary frequency-domain function, and for $\Omega>0$ define its frequency-truncated inverse transform
\[
g_\Omega(t)=\frac{1}{2\pi}\int_{-\Omega}^{\Omega}G(\omega)\,e^{i\omega t}\,d\omega .
\]
\begin{enumerate}
\item[(i)] For every finite $T$, $\varepsilon$, and $\Omega$,
\[
\frac{1}{2\pi}\int_{-\Omega}^{\Omega}F_{T,\varepsilon}(\omega)\,G^{*}(\omega)\,d\omega
=\int_{-\infty}^{\infty}f_{T,\varepsilon}(t)\,g_\Omega^{*}(t)\,dt ,
\]
where $^{*}$ denotes the complex conjugate.
\item[(ii)] The generalized spectral pairing at stage $(T,\varepsilon)$ is
\[
P_{T,\varepsilon}[G]=\operatorname{g.lim}_{\Omega\to\infty}\frac{1}{2\pi}\int_{-\Omega}^{\Omega}F_{T,\varepsilon}(\omega)\,G^{*}(\omega)\,d\omega
\]
whenever the limit exists. It exists if and only if $\lim_{\Omega\to\infty}\int f_{T,\varepsilon}\,g_\Omega^{*}\,dt$ exists, and the two are equal. If, in addition, $g_\Omega\to g$ almost everywhere on $[-T,T]$ with $\sup_\Omega|g_\Omega|$ bounded there (for instance, $G$ is the Fourier transform of a function $g\in L^1(\mathbb{R})$ that is locally Lipschitz on a neighborhood of $[-T,T]$), then $P_{T,\varepsilon}[G]=\int f_{T,\varepsilon}\,g^{*}\,dt$.
\item[(iii)] If $P_{T,\varepsilon}[G]$ has a limit as $T\to\infty$, $\varepsilon\to0$, that limit is the generalized spectral pairing of $f$ with $G$,
\[
\lim_{T\to\infty,\,\varepsilon\to0}\ \operatorname{g.lim}_{\Omega\to\infty}\frac{1}{2\pi}\int_{-\Omega}^{\Omega}F_{T,\varepsilon}\,G^{*}\,d\omega
=\lim_{T\to\infty,\,\varepsilon\to0}\ \lim_{\Omega\to\infty}\int f_{T,\varepsilon}\,g_\Omega^{*}\,dt ,
\]
and under the additional condition in (ii) it equals $\lim_{T\to\infty,\,\varepsilon\to0}\int f_{T,\varepsilon}\,g^{*}\,dt$.
\end{enumerate}
No integrability of $G$ beyond local integrability is imposed. This proposition concerns generalized spectral interpretation through pairing; pointwise inverse reconstruction is the separate matter addressed by \Cref{thm:inversion}, and pointwise convergence of $F_{T,\varepsilon}(\omega)$ itself is neither asserted nor required.
\end{proposition}

\begin{proof}
(i) Since $f_{T,\varepsilon}\in L^1(\mathbb{R})$ and $G\in L^1[-\Omega,\Omega]$, the function $f_{T,\varepsilon}(t)G^{*}(\omega)e^{-i\omega t}$ is absolutely integrable on $\mathbb{R}\times[-\Omega,\Omega]$, with integral bounded by $\|f_{T,\varepsilon}\|_{1}\int_{-\Omega}^{\Omega}|G|$. Fubini's theorem therefore gives
\[
\int_{-\Omega}^{\Omega}F_{T,\varepsilon}(\omega)G^{*}(\omega)\,d\omega
=\int_{-\infty}^{\infty}f_{T,\varepsilon}(t)\Big[\int_{-\Omega}^{\Omega}G^{*}(\omega)e^{-i\omega t}\,d\omega\Big]dt
=2\pi\int_{-\infty}^{\infty}f_{T,\varepsilon}(t)\,g_\Omega^{*}(t)\,dt .
\]
(ii) The first statement is immediate from (i). Under the stated condition, dominated convergence with the integrable majorant $M|f_{T,\varepsilon}|$, where $M$ bounds $\sup_\Omega|g_\Omega|$ on $[-T,T]$, gives $\int f_{T,\varepsilon}\,g_\Omega^{*}\,dt\to\int f_{T,\varepsilon}\,g^{*}\,dt$. For the example, split $g_\Omega(t)=\int g(x)\frac{\sin\Omega(t-x)}{\pi(t-x)}\,dx$ at $|t-x|=\delta$: the near part is bounded by the Lipschitz constant and $\sup|g|$ on $[-T,T]$, the far part by $\|g\|_1/(\pi\delta)$, uniformly in $\Omega$, and $g_\Omega\to g$ by the Riemann--Lebesgue lemma.
(iii) follows from (i) and (ii) by applying the same ordered limit to both sides.
\end{proof}

\begin{remarks}
Convergence of the forward transform is outside the definition. The t.g.l.\ transform of $f$ is the family $\{F_{T,\varepsilon}\}$ of ordinary finite-stage transforms, and the pairing of \Cref{thm:pairing} is an auxiliary construction on top of it. Whether $F_{T,\varepsilon}(\omega)$ converges as the truncation is removed (pointwise, in $L^2$, or in $\mathcal S'$), and whether the limit depends on the manner of truncation (Example~7), is a separate question that the definition does not ask. The pairing may exist for some auxiliary functions $G$ and not for others, and its existence is not required for the pointwise reconstruction of \Cref{thm:inversion}. Agreement for one auxiliary function does not imply agreement for all test functions, and pointwise reconstruction alone does not establish a canonical spectral functional.
\end{remarks}

Since the inner-exclusion and outer-truncation limits are operations of the same structural kind --- a sharp truncation followed by a limit --- the outer truncation and limit $T\to\infty$ are used alone, representatively, throughout the remainder of this paper, with $f_T$ and $F_T$ in place of $f_{T,\varepsilon}$ and $F_{T,\varepsilon}$. Where a target function has a local singularity and the inner-exclusion limit $\varepsilon\to0$ is also required, this is stated explicitly, and $f_{T,\varepsilon}$, $F_{T,\varepsilon}$ are used in place of $f_T$, $F_T$.

\begin{proposition}[Optimality of frequency truncation in $L^2$]\label{prop:l2opt}
Consider a function $f(t)$ and its truncation $f_T(t)=f(t)\mathbf{1}_{[-T,T]}(t)$. Assume that $f_T \in L^2_{\mathrm{loc}}(\mathbb{R})$ and denote its Fourier transform by $F_T(\omega)$. For a given bandwidth $\Omega>0$, define
\[
F_{T,\Omega}(\omega) = F_T(\omega)\,\mathbf{1}_{[-\Omega,\Omega]}(\omega),
\]
and
\[
f_{T,\Omega}(t) = \frac{1}{2\pi}\int_{-\Omega}^{\Omega} F_T(\omega)\, e^{i\omega t}\,d\omega = \frac{1}{2\pi}\int_{-\infty}^{\infty} F_{T,\Omega}(\omega)\, e^{i\omega t}\,d\omega.
\]
Define the approximation error:
\[
E_T(\Omega) = \int_{-\infty}^{\infty} \left|f_T(t) - f_{T,\Omega}(t)\right|^2 dt.
\]
$f_{T,\Omega}(t)$ is the orthogonal projection of $f_T$ onto the band-limited subspace $\{h\in L^2(\mathbb{R}) : \operatorname{supp}\hat h\subset[-\Omega,\Omega]\}$, and hence the unique $L^2$-optimal band-limited approximation to $f_T$; the error depends only on the spectral components outside the truncation band.
\end{proposition}

We omit the notation of $T$ for simplicity. Expand the integrand as:
\[
E(\Omega) = \int_{-\infty}^{\infty} \left[f^2(t) - 2f(t)f_\Omega(t) + f_\Omega^2(t)\right] dt.
\]
Considering properties of the Fourier transform for real functions
\[
\int_{-\infty}^{\infty} f_1(x) f_2(x)\,dx = \frac{1}{2\pi}\int_{-\infty}^{\infty} F_1^{*}(\omega) F_2(\omega)\,d\omega = \frac{1}{2\pi}\int_{-\infty}^{\infty} F_1(\omega) F_2^{*}(\omega)\,d\omega,
\]
\[
\int_{-\infty}^{\infty} f^2(x)\,dx = \frac{1}{2\pi}\int_{-\infty}^{\infty} |F(\omega)|^2\,d\omega,
\]
we obtain
\[
\begin{aligned}
2\pi E(\Omega) &= \int_{-\infty}^{\infty} |F(\omega)|^2\,d\omega - 2\int_{-\infty}^{\infty} F(\omega) F_\Omega^{*}(\omega)\,d\omega + \int_{-\Omega}^{\Omega} |F(\omega)|^2\,d\omega\\
&= \int_{-\infty}^{\infty} |F(\omega)|^2\,d\omega - \int_{-\Omega}^{\Omega} |F(\omega)|^2\,d\omega.
\end{aligned}
\]
Thus the approximation error depends only on the spectral components outside the truncation band $|\omega|\le\Omega$, and the function $f_\Omega(t)$ is the optimal $L^2$-approximation of $f(t)$ among all functions whose Fourier transforms vanish outside this band. If we take limits, $\lim_{\Omega\to\infty} E(\Omega) = 0$ whenever $f\in L^2(\mathbb{R})$, by Plancherel's theorem.

\begin{remarks}[L$^2$ convergence]
\Cref{prop:l2opt} is a classical $L^2$ convergence result and is distinct from the local pointwise convergence of \Cref{thm:inversion}. For the sinc function of Example~12, whose spectrum is compactly supported, the error vanishes for sufficiently large $\Omega$. The chirp signal of Example~13 lies outside this $L^2$ setting, illustrating that t.g.l.\ reconstruction is not restricted to the mechanism described by this proposition.
\end{remarks}

The above result implies that frequency components inside and outside the interval $|\omega|\le\Omega$ contribute independently to the total signal energy, consistent with the continuous-frequency orthogonal-expansion viewpoint of Section~3.3. This independence guarantees that the approximation error $E(\Omega)$ is monotonically non-increasing in $\Omega$: successively enlarging the frequency-truncation band can only improve, never worsen, the reconstruction.

\begin{proposition}[Localization of the error produced by time truncation]\label{prop:localerr}
Let $f(t)\in L^2[-T_2,T_2]$, and let $T_2>T_1>0$. Let $f_{T_1,\Omega}(t)$ denote the inverse transform reconstructed from the time truncated function supported on $[-T_1,T_1]$ and frequency truncation $[-\Omega,\Omega]$. Define the reconstruction error
\[
E_\Omega(T_1,T_2) = \int_{-T_2}^{T_2} \left|f(t) - f_{T_1,\Omega}(t)\right|^2 dt.
\]
Assume that
\[
\lim_{\Omega\to\infty} f_{T_1,\Omega}(t) = f_{T_1}(t),
\]
in the pointwise sense. Then
\[
\lim_{\Omega\to\infty} E_\Omega(T_1,T_2) = E(T_1,T_2) = \int_{-T_2}^{T_2} \left|f(t) - f_{T_1}(t)\right|^2 dt.
\]
Moreover, since $f_{T_1}(t)=0$ for $|t|>T_1$, we obtain
\[
E(T_1,T_2) = \int_{-T_2}^{-T_1} f^2(t)\,dt + \int_{T_1}^{T_2} f^2(t)\,dt.
\]
\end{proposition}
Thus, after removal of the frequency truncation, the reconstruction error on $[-T_2,T_2]$ is exactly the energy discarded outside the retained time interval $[-T_1,T_1]$. This complements \Cref{prop:l2opt}, which describes the corresponding effect of frequency truncation in the $L^2$ setting.

\begin{observation}[Asymmetric truncation and the choice of symmetric frequency truncation]
Let $f_T \in L^1(\mathbb{R})$ be a truncated function with support $[-T_1,T_2]$. Consider the inverse transform
\[
\overleftarrow{f}_T(t) = \lim_{\Omega_1,\Omega_2\to\infty} \frac{1}{2\pi}\int_{-\Omega_1}^{\Omega_2} F_T(\omega)\, e^{i\omega t}\,d\omega
= \lim_{\Omega_1,\Omega_2\to\infty} \frac{1}{2\pi}\int_{-\Omega_1}^{\Omega_2}\left[\int_{-\infty}^{\infty} f_T(x)\, e^{-i\omega x}\,dx\right] e^{i\omega t}\,d\omega.
\]
For finite $\Omega_1,\Omega_2$, this reduces to a convolution of $f_T$ with the kernel
\[
K_{\Omega_1,\Omega_2}(y) = \frac{\sin(\Omega_2 y)+\sin(\Omega_1 y)}{2\pi y} + i\,\frac{\cos(\Omega_2 y)-\cos(\Omega_1 y)}{2\pi y},
\]
whose real and imaginary parts are shown explicitly below. If both the time truncation and the frequency truncation are asymmetric ($T_1 \ne T_2$ and $\Omega_1 \ne \Omega_2$), the finite-stage kernel generally has a nonzero imaginary part, though this does not by itself imply a surviving imaginary contribution in the limit; simultaneous asymmetric time and frequency truncation does not guarantee convergent inversion in general, and its convergence, when it holds, requires separate justification. If either
\begin{itemize}
\item[(i)] the frequency truncation is symmetric ($\Omega_1 = \Omega_2$), or
\item[(ii)] the time truncation is symmetric ($T_1 = T_2$),
\end{itemize}
the imaginary part cancels and the inverse transform converges to $f_T(t)$. Asymmetric physical time windows are accordingly permitted throughout this paper, but symmetric frequency truncation $\Omega_1=\Omega_2=\Omega$ is prescribed as the standard t.g.l.\ inversion convention, giving the ordinary continuous-inversion Dirichlet kernel $K_{\Omega,\Omega}(y)=\sin(\Omega y)/(\pi y)$ used in \Cref{thm:inversion}.
\end{observation}

Consider the truncated inverse transform:
\[
\overleftarrow{f}_T(t) = \lim_{\Omega_1,\Omega_2\to\infty} \frac{1}{2\pi}\int_{-\Omega_1}^{\Omega_2}\left[\int_{-T_1}^{T_2} f(x)\, e^{-i\omega x}\,dx\right] e^{i\omega t}\,d\omega
\]
\begin{multline*}
= \lim_{\Omega_1,\Omega_2\to\infty} \frac{i}{2\pi}\int_{-T_1}^{T_2} f(x)\, \frac{\cos\Omega_2(x-t) - \cos\Omega_1(x-t)}{x-t}\,dx \\
+ \lim_{\Omega_1,\Omega_2\to\infty} \frac{1}{2\pi}\int_{-T_1}^{T_2} f(x)\, \frac{\sin\Omega_2(x-t) + \sin\Omega_1(x-t)}{x-t}\,dx
\end{multline*}
\begin{multline*}
= \lim_{\Omega_1,\Omega_2\to\infty} \frac{i}{2\pi}\int_{-T_1-t}^{T_2-t} f(y+t)\, \frac{\cos\Omega_2 y - \cos\Omega_1 y}{y}\,dy \\
+ \lim_{\Omega_1,\Omega_2\to\infty} \frac{1}{2\pi}\int_{-T_1-t}^{T_2-t} f(y+t)\, \frac{\sin\Omega_2 y + \sin\Omega_1 y}{y}\,dy.
\end{multline*}
With asymmetric frequency truncation, the cosine terms may produce a nonvanishing imaginary contribution. Symmetric frequency truncation eliminates these terms and leaves the standard Dirichlet sine kernel. Since asymmetric time windows are common in applications, symmetric frequency truncation $[-\Omega,\Omega]$ is adopted throughout this paper unless otherwise stated. Example~8 illustrates the asymmetric-frequency case and its relation to the Hilbert transform.

\begin{theorem}[Interchangeability of time limit and frequency limit at a point of classical inversion]\label{thm:interchange}
Let $f\in L^1(\mathbb{R})$, let $f_T(t)=f(t)\mathbf{1}_{[-T,T]}(t)$, and define
\[
F_T(\omega) = \int_{-\infty}^{\infty} f_T(t)e^{-i\omega t}\,dt.
\]
Fix a point $t$ at which the classical symmetric Fourier inversion formula holds; for example, suppose that $f$ is of bounded variation in a neighborhood of $t$. Define
\[
I_{T,\Omega}(t) = \frac{1}{2\pi}\int_{-\Omega}^{\Omega} F_T(\omega)e^{i\omega t}\,d\omega.
\]
Then
\[
\lim_{\Omega\to\infty}\lim_{T\to\infty} I_{T,\Omega}(t) = \lim_{T\to\infty}\lim_{\Omega\to\infty} I_{T,\Omega}(t) = \frac{f(t^-)+f(t^+)}{2}.
\]
At a continuity point $t$, both iterated limits are equal to $f(t)$.
\end{theorem}

\begin{proof}
Since $f\in L^1(\mathbb{R})$, $\|f_T-f\|_1\to0$ as $T\to\infty$. If $F(\omega)=\int_{-\infty}^{\infty}f(x)e^{-i\omega x}\,dx$, then $|F_T(\omega)-F(\omega)|\le\int_{-\infty}^{\infty}|f_T(x)-f(x)|\,dx=\|f_T-f\|_1$, so
\[
\sup_{\omega\in\mathbb{R}} |F_T(\omega)-F(\omega)| \le \|f_T-f\|_1 \longrightarrow 0.
\]
Therefore, for every fixed $\Omega<\infty$,
\[
\left|I_{T,\Omega}(t) - \frac{1}{2\pi}\int_{-\Omega}^{\Omega} F(\omega)e^{i\omega t}\,d\omega\right| \le \frac{1}{2\pi}\int_{-\Omega}^{\Omega} |F_T(\omega)-F(\omega)|\,d\omega \le \frac{\Omega}{\pi}\|f_T-f\|_1 \longrightarrow 0.
\]
Thus $\lim_{T\to\infty} I_{T,\Omega}(t) = \frac{1}{2\pi}\int_{-\Omega}^{\Omega} F(\omega)e^{i\omega t}\,d\omega$. Taking $\Omega\to\infty$, the assumed classical symmetric Fourier inversion condition at $t$ gives
\[
\lim_{\Omega\to\infty}\lim_{T\to\infty} I_{T,\Omega}(t) = \frac{f(t^-)+f(t^+)}{2}.
\]

For the opposite order, fix $T>|t|$ sufficiently large. Then $f_T$ agrees with $f$ in a neighborhood of $t$, and therefore satisfies the same local inversion condition there. Classical symmetric Fourier inversion gives $\lim_{\Omega\to\infty} I_{T,\Omega}(t) = (f(t^-)+f(t^+))/2$. The right-hand side is independent of $T$ once the truncation interval contains a neighborhood of $t$. Consequently, $\lim_{T\to\infty}\lim_{\Omega\to\infty} I_{T,\Omega}(t) = (f(t^-)+f(t^+))/2$. Thus the two iterated limits agree at every point satisfying the stated classical inversion condition.
\end{proof}

\begin{remarks}
The theorem separates two convergence requirements. If $f\in L^1(\mathbb{R})$, then
\[
\sup_\omega |F_T(\omega)-F(\omega)| \leq \|f_T-f\|_1\to0,
\]
so the limit $T\to\infty$ may be passed through every fixed finite frequency integral. This does not by itself guarantee pointwise Fourier inversion, which requires an additional local condition such as that used in \Cref{thm:inversion}.

The origin of the limit-order question can be seen directly from the finite-stage double integral. For finite $T$ and $\Omega$, interchange of the two ordinary integrations gives
\[
I_{T,\Omega}(t) = \frac{1}{2\pi} \int_{-\Omega}^{\Omega} \int_{-T}^{T} f(x)e^{i\omega(t-x)}\,dx\,d\omega = \int_{-T}^{T}f(x)D_\Omega(t-x)\,dx,
\]
where
\[
D_\Omega(u)=\frac{\sin(\Omega u)}{\pi u}.
\]
Thus no order problem occurs at the finite stage. It arises only when $T$ and $\Omega$ are allowed to tend to infinity: the two iterated limits then correspond to two different orders of completing the associated improper Fourier integrations. Equality of these limits therefore requires additional convergence conditions and does not follow from pointwise convergence of $F_T$ alone.

Absolute integrability is sufficient but not necessary for interchangeability. For some non-$L^1$ functions, oscillatory cancellation allows both successive improper operations to converge, whereas for others the reversed route fails; the examples below illustrate both possibilities.

Thus the essential distinction is not simply between $L^1$ and non-$L^1$ functions. Interchange requires sufficient convergence to justify the successive improper operations and their passage through the subsequent integration. When such conditions are unavailable, the prescribed t.g.l.\ order --- $\Omega\to\infty$ first and $T\to\infty$ afterward --- gives the constructive inversion established in \Cref{thm:inversion}.
\end{remarks}

\begin{theorem}[Non-interchangeability of time limit and frequency limit for inverse transform]\label{thm:noninterchange}
Let $f \notin L^1(\mathbb{R})$ and $f \in L^1_{\mathrm{loc}}(\mathbb{R})$. For the truncated function $f_T$, let
\[
F_T(\omega) = \int_{-\infty}^{\infty} f_T(t)\, e^{-i\omega t}\,dt
\]
be the classical Fourier transform of $f_T$, and define the truncated inverse integral
\[
I_{T,\Omega}(t) = \frac{1}{2\pi}\int_{-\Omega}^{\Omega} F_T(\omega)\, e^{i\omega t}\,d\omega.
\]
Then, in general,
\[
\lim_{T\to\infty}\lim_{\Omega\to\infty} I_{T,\Omega}(t) \ne \lim_{\Omega\to\infty}\lim_{T\to\infty} I_{T,\Omega}(t)
\]
for pointwise convergence of classical improper integrals. Here the right-hand side may fail to exist, as it does in the example below.
\end{theorem}

\begin{proof}
For each fixed $T$, $f_T \in L^1(\mathbb{R})$. Therefore the classical Fourier inversion theorem gives
\[
\lim_{\Omega\to\infty} I_{T,\Omega}(t) = f_T(t)
\]
at every continuity point of $f_T$. Taking the outer limit yields
\[
\lim_{T\to\infty}\lim_{\Omega\to\infty} I_{T,\Omega}(t) = \lim_{T\to\infty} f_T(t) = f(t).
\]

Now consider the interchanged order of limits,
\[
\lim_{\Omega\to\infty}\lim_{T\to\infty} I_{T,\Omega}(t).
\]
Let $f(t) = t$ (Example~3). Then we denote truncated $f_T$ as $f_T(t) = t\,\mathbf{1}_{[-T,T]}(t)$. Its Fourier transform is
\[
F_T(\omega) = \int_{-\infty}^{\infty} f_T(t)\, e^{-i\omega t}\,dt = \int_{-T}^{T} t\, e^{-i\omega t}\,dt.
\]
A direct computation shows
\[
F_T(\omega) = \int_{-T}^{T} t\, e^{-i\omega t}\,dt = \left[\frac{t\, e^{-i\omega t}}{-i\omega}\right]_{-T}^{T} - \int_{-T}^{T} \frac{e^{-i\omega t}}{-i\omega}\,dt = i\,\frac{2T\cos\omega T}{\omega} - i\,\frac{2\sin\omega T}{\omega^2}.
\]
Using $\delta_T(\omega) = \sin(\omega T)/\pi\omega$, we obtain:
\[
F_T(\omega) = i\, 2\pi\, \delta_T'(\omega).
\]
Substituting into the inverse integral yields
\[
\overleftarrow{f}_T(t) = i \int_{-\Omega}^{\Omega} \delta_T'(\omega)\, e^{i\omega t}\,d\omega = i\,\Big[\delta_T(\omega)\, e^{i\omega t}\Big]_{-\Omega}^{\Omega} - i\cdot it \int_{-\Omega}^{\Omega} \delta_T(\omega)\, e^{i\omega t}\,d\omega,
\]
where the boundary term
\[
i\,\Big[\delta_T(\omega)\, e^{i\omega t}\Big]_{-\Omega}^{\Omega} = -2\,\delta_T(\Omega)\sin\Omega t,
\]
does not converge as a classical improper integral as $T\to\infty$, and therefore does not converge as $\Omega\to\infty$ after $T\to\infty$. Hence, $\lim_{\Omega\to\infty}\lim_{T\to\infty} I_{T,\Omega}(t)$ fails to exist as a classical improper integral. Therefore $\lim_{T\to\infty}\lim_{\Omega\to\infty} I_{T,\Omega}(t) = f(t)$ but $\lim_{\Omega\to\infty}\lim_{T\to\infty} I_{T,\Omega}(t)$ does not exist in general. Thus the limits are not interchangeable in general.
\end{proof}

\begin{remarks}
The theorem shows that outside the classical $L^1$ setting, removing the time truncation before completing the frequency-domain inversion may fail. Thus the order of the two limits is analytically significant. Throughout this paper, $\lim_{T\to\infty}$ and $\lim_{\Omega\to\infty}$ denote limits of completed finite-stage integrals, not pointwise limits of integrands taken before integration. Moreover, pointwise convergence of a forward-transform family, when it happens, is not by itself sufficient to justify substitution of that limit into the inverse integral; Example~9 gives an explicit illustration.
\end{remarks}

The point of the present construction is therefore not the general fact that iterated limits may fail to commute, but the explicit use of the ordered-limit structure in defining Fourier inversion under truncation.

Noncommutativity of iterated limits is a standard phenomenon in analysis; \Cref{thm:noninterchange} identifies its specific consequence for the time- and frequency-truncation limits used in the present Fourier construction.

\begin{observation}[Spectral spreading produced by time truncation]
For a function $f(t) \in L^1(\mathbb{R})$ and its time-truncated $f_T(t) = f(t)\cdot \mathbf{1}_{[-T,T]}(t)$, sharp rectangular time truncation generally produces sinc-type spectral spreading in the forward transform $F_T(\omega)$.
\end{observation}

The forward transform of $f_T(t)$ is given by a convolutional integral:
\[
F_T(\omega) = \frac{1}{2\pi}\int_{-\infty}^{\infty} F(\omega-x)\, \frac{2\sin(Tx)}{x}\,dx,
\]
where $2\sin(Tx)/x \leftrightarrow \mathbf{1}_{[-T,T]}(t)$; distributionally, $\frac{1}{2\pi}\cdot\frac{2\sin(Tx)}{x}=\frac{\sin(Tx)}{\pi x}\to\delta(x)$ as $T\to\infty$.

Since $2\sin(Tx)/x$ has infinite support in the $x$-domain, convolution against it generally spreads $F_T(\omega)$ over the entire frequency axis. Thus $F_T(\omega)$ is the ordinary Fourier transform of the truncated function $f_T$, not necessarily a pointwise approximation to an untruncated spectrum $F(\omega)$. Its role in the t.g.l.\ construction is first of all as a finite-stage spectral quantity used in the subsequent generalized pairing and inverse reconstruction. The spectral spreading caused by time truncation is distinct from the Dirichlet-kernel localization produced by frequency truncation in the inverse transform.

\begin{remarks}
Spectral spreading at finite $T$ is an expected consequence of truncation and does not by itself indicate failure of the generalized Fourier construction; the resulting transform $F_T(\omega)$ need not remain absolutely integrable. The relevant question is whether the prescribed subsequent pairing and ordered limits exist.
\end{remarks}

\begin{theorem}[Pointwise reconstruction --- explicit $C^1$ verification of \Cref{thm:inversion}]\label{thm:c1}
This theorem gives an explicit, self-contained verification of the t.g.l.\ inversion formula (\Cref{thm:inversion}) for the illustrative case where $f$ is piecewise $C^1$ on the retained interval and $C^1$ near the reconstruction point, making the role of the Dirichlet kernel and the Riemann--Lebesgue lemma explicit; it is not an independent inversion theorem.

Let $f\in L^1_{\mathrm{loc}}(\mathbb{R})$ be continuously differentiable in an interval containing $t$, and let
\[
f_T(t) = f(t)\cdot \mathbf{1}_{[-T_1,T_2]}(t)
\]
be the truncated function with $T_1,T_2>0$, with forward transform $F_T(\omega) = \int_{-\infty}^{\infty} f_T(t)\, e^{-i\omega t}\,dt$. Define the frequency-truncated inverse transform of $f_T$ as
\[
\overleftarrow{f}_{T,\Omega}(t) = \frac{1}{2\pi}\int_{-\Omega}^{\Omega} F_T(\omega)\, e^{i\omega t}\,d\omega = \int_{-\infty}^{\infty} f_T(x)\,\frac{\sin[\Omega(t-x)]}{\pi(t-x)}\,dx.
\]
Then
\[
\lim_{\Omega\to\infty} \overleftarrow{f}_{T,\Omega}(t) = f_T(t), \qquad \lim_{T_1,T_2\to\infty} f_T(t) = f(t), \qquad \text{so that} \qquad \lim_{T_1,T_2\to\infty}\Big[\lim_{\Omega\to\infty}\overleftarrow{f}_{T,\Omega}(t)\Big] = f(t),
\]
exhibiting the pedagogical chain $\overleftarrow{f}_{T,\Omega}\xrightarrow{\Omega\to\infty}f_T\xrightarrow{T\to\infty}f$.
\end{theorem}

\begin{proof}
Since $f_T \in L^1(\mathbb{R})$ and the frequency interval is truncated within a finite interval, Fubini's theorem justifies interchanging the order of integration, giving
\[
\overleftarrow{f}_{T,\Omega}(t) = \frac{1}{\pi}\int_{-T_1}^{T_2} \frac{\sin\Omega(x-t)}{x-t}\, f(x)\,dx = \frac{1}{\pi}\int_{-T_1-t}^{T_2-t} \frac{\sin(\Omega y)}{y}\, f(y+t)\,dy
\]
after the substitution $y=x-t$. Writing $h(y) = [f(y+t) - f(t)]/y$, this splits as
\[
\frac{1}{\pi}\int_{-T_1-t}^{T_2-t} \sin(\Omega y)\, h(y)\,dy \ + \ \frac{f(t)}{\pi} \int_{-T_1-t}^{T_2-t} \frac{\sin(\Omega y)}{y}\,dy.
\]
Since $h$ is locally bounded with $\lim_{y\to0}h(y)=f'(t)$, the Riemann--Lebesgue lemma sends the first term to $0$ as $\Omega\to\infty$; the second term converges, using $\int_{-\infty}^{\infty}\sin(x)/x\,dx=\pi$, to $f_T(t)$ as defined by the indicator function on $[-T_1,T_2]$. This gives $\lim_{\Omega\to\infty}\overleftarrow{f}_{T,\Omega}(t)=f_T(t)$; taking $T_1,T_2\to\infty$ then gives $f(t)$ at the given point $t$.
\end{proof}

\begin{remarks}
The truncate-and-generalized-limit formulation may also be interpreted in terms of approximate identities. After substituting the truncated Fourier transform into the inverse integral, the reconstruction formula can be written as a convolution with the kernel $\sin(\Omega x)/(\pi x)$ over a truncated interval. Thus the inversion process is governed by a two-parameter family of kernels depending on both the time truncation parameter $T$ and the frequency truncation parameter $\Omega$.

The kernel
\[
\delta_\Omega(t) = \frac{\sin\Omega t}{\pi t}
\]
arises from symmetric frequency truncation and is the classical Dirichlet-type kernel that underlies Fourier inversion. It possesses oscillatory structure and infinite support.

The inverse transform obtained in the present truncation-based formulation admits a natural interpretation in terms of convolution with a Dirichlet-type kernel. In particular, the inversion formula may be expressed, for each fixed $T$, in the form
\begin{equation}
\overleftarrow{f}_T(t) = \lim_{\Omega\to\infty} \int_{-T}^{T} f(x)\, \frac{\sin\Omega(t-x)}{\pi(t-x)}\,dx.
\label{eq:6}
\end{equation}
which represents reconstruction of the function through convolution with a delta-convergent kernel.

If we denote $f(t)=x(t)$ and $\overleftarrow{f}(t)=y(t)$, Eq.~(\ref{eq:6}) expresses the input-output relation of a linear time-invariant signal system. From this viewpoint, the inverse transform may be interpreted as reconstruction through a finite-bandwidth observation operator acting on the signal. In this sense, the inversion formula describes reconstruction through a physically interpretable monitoring operator rather than through an abstract duality pairing. In practical measurement systems, reconstruction of signals is typically performed through such band-limited responses rather than through idealized pointwise inversion. The present formulation therefore preserves the transition structure between finite-bandwidth reconstruction and exact inversion as the truncation parameter tends to infinity.

This interpretation highlights a conceptual distinction between the truncation-based formulation and the purely distributional inversion formula. While the distributional Fourier transform provides an immediate representation of the inverse transform, the present approach explicitly exhibits the intermediate reconstruction mechanism through convolution with Dirichlet-type kernels. These kernels may be regarded as constructive realizations of localization processes arising in signal reconstruction and quantum-mechanical observation models, where impulse-type responses are obtained through limiting processes applied to band-limited approximations rather than as idealized distributions.

The Fourier inversion formula is established pointwise at every point of continuity for functions admitting truncation-based inversion limits using general asymmetric truncations in the time variable. The proof relies on oscillatory cancellation arising from truncated exponential kernels and does not require a priori principal-value regularization or explicit distributional pairing with test functions in the definition of the inverse transform.

As discussed generally in the Introduction, the forward transform is defined as a first-stage t.g.l.\ Fourier-transform family, while the inverse transform converges pointwise even for functions without $L^1(\mathbb{R})$. This distinction clarifies the relation to the classical Fourier series approach in Appendix~A, where convergence of coefficients is required; the present formulation instead separates the two roles, generalized-limit convergence for the forward transform and localized pointwise convergence for the inverse transform.

Sharp truncation introduces discontinuities at the cutoff points and may produce oscillatory artifacts on the Fourier-transform side, such as Gibbs-type effects. Consequently, $F_T(\omega)$ does not in general converge to the Fourier transform in an absolute norm such as $L^1(\mathbb{R})$, even as the truncation points tend to infinity. The t.g.l.\ formulation does not require such convergence: the truncated transform is used only within ordinary oscillatory integrals, where positive and negative contributions cancel, and the second integration of the inverse transform produces a localization kernel through which these oscillatory effects are absorbed in the ordered limiting process. The defect of sharp truncation thus remains at the level of absolute Fourier-side approximation, but does not affect the operational convergence that defines the t.g.l.\ inverse transform.
\end{remarks}

\begin{theorem}[Fourier transform of differentiated functions]\label{thm:deriv}
Let $n\ge 1$, let $f^{(n-1)}$ be absolutely continuous on every interval $[-T,T]$, and let $f_T=f\,\mathbf 1_{[-T,T]}$, $F_T(\omega)=\int_{-T}^{T}f(t)e^{-i\omega t}\,dt$. For every finite $T$:
\begin{enumerate}\setlength{\itemsep}{0pt}\setlength{\topsep}{2pt}\setlength{\parsep}{0pt}
\item[(i)] the weak derivative of the truncated function is
\[
f_T^{(n)}=(f^{(n)})_T+\beta_{n,T},\qquad
\beta_{n,T}=\sum_{k=0}^{n-1}\Big[f^{(n-1-k)}(-T)\,\delta^{(k)}(t+T)-f^{(n-1-k)}(T)\,\delta^{(k)}(t-T)\Big],
\]
where $(f^{(n)})_T=f^{(n)}\mathbf 1_{[-T,T]}$ (the deltas are those of \Cref{thm:s2one}); its transform is $(i\omega)^nF_T$, so that of $(f^{(n)})_T$ is $(i\omega)^nF_T-\hat\beta_{n,T}$, where the hat denotes the Fourier transform;
\item[(ii)] $t^n f_T(t)\leftrightarrow(-i)^{-n}F_T^{(n)}(\omega)$;
\item[(iii)] if $|t|<T$ and $f^{(n)}$ satisfies a classical local Dirichlet condition at $t$ (for example, bounded variation near $t$), then
\[
\lim_{\Omega\to\infty}\frac1{2\pi}\int_{-\Omega}^{\Omega}\Big[(i\omega)^nF_T(\omega)-\hat\beta_{n,T}(\omega)\Big]e^{i\omega t}\,d\omega=\frac{f^{(n)}(t^-)+f^{(n)}(t^+)}{2}.
\]
\end{enumerate}
Hence $f^{(n)}(t)\leftrightarrow(i\omega)^nF(\omega)$ and $t^nf(t)\leftrightarrow(-i)^{-n}F^{(n)}(\omega)$ in the sense of (i)--(iii) followed by $T\to\infty$.
\end{theorem}

\begin{proof}
(i) For $n=1$, integration by parts on $[-T,T]$ gives
\[
\int_{-T}^{T}f'(t)e^{-i\omega t}\,dt=i\omega F_T(\omega)+f(T)e^{-i\omega T}-f(-T)e^{i\omega T}.
\]
By \Cref{thm:s2one}, $f_T'=(f')_T+f(-T)\delta(t+T)-f(T)\delta(t-T)$, whose delta terms transform to $f(-T)e^{i\omega T}-f(T)e^{-i\omega T}$, exactly minus the boundary term above. The two cancel identically, so $f_T'\leftrightarrow i\omega F_T$. For compactly supported $u$, $\langle u',e^{-i\omega t}\rangle=i\omega\langle u,e^{-i\omega t}\rangle$; induction on $n$ gives $\beta_{n,T}$ and $(i\omega)^nF_T$.

(ii) The interval $[-T,T]$ is compact, so differentiation under the integral sign is elementary.

(iii) $(f^{(n)})_T$ is an ordinary truncated function with transform $(i\omega)^nF_T-\hat\beta_{n,T}$ by (i), and equals $f^{(n)}$ near $t$; \Cref{thm:inversion} at fixed $T$ gives the limit. \qedhere
\end{proof}

\begin{remarks}
At every finite $T$ the boundary terms of the integration by parts do not vanish; they are the transforms of the truncation deltas of the weak derivative, with opposite sign, and cancel the delta contribution in $(i\omega)^nF_T$ identically, requiring neither $T\to\infty$ nor $f(\pm T)\to 0$. Differentiation under the Fourier integral is thereby justified in the ordered form (iii), with no limit function $F(\omega)=\lim_{T\to\infty}F_T(\omega)$ required.

The identity for $F^{(n)}$ follows by differentiation under the integral over the compact interval at each finite $T$, and the limit $T\to\infty$ exists in whatever sense $F_T$ converges. The only distributional ingredient is the weak derivative of a compactly supported function, whose boundary deltas are realized by ordinary delta-convergent families (\Cref{thm:s2one}).
\end{remarks}
\begin{proposition}[Pointwise reconstruction without convergence of the forward transform; the t.g.l.\ construction and its scope]\label{prop:recon}
Let $f$ be t.g.l.-admissible with locally finite singular set $S$, integrable on every compact subset of $\mathbb{R}\setminus S$.
\begin{enumerate}\setlength{\itemsep}{0pt}\setlength{\topsep}{2pt}\setlength{\parsep}{0pt}
\item[(i)] (Time truncation; no condition on the forward transform) Let $f_{T,\varepsilon}$ be $f$ cut to $[-T_1,T_2]$, with the $\varepsilon$-neighborhoods of the points of $S$ excluded, and set to zero elsewhere; it is a function on the whole line. The parameters $T_1,T_2,\varepsilon_1,\varepsilon_2$ may be chosen arbitrarily, not necessarily symmetrically. Then $f_{T,\varepsilon}\in L^1(\mathbb{R})$ and $F_{T,\varepsilon}(\omega)=\int_{-\infty}^{\infty}f_{T,\varepsilon}(t)e^{-i\omega t}\,dt$ is an ordinary Fourier integral. The t.g.l.\ transform of $f$ is the family $\{F_{T,\varepsilon}\}$. Convergence of the forward transform is outside the definition (\Cref{thm:pairing}). Pointwise convergence may fail: for $f=1$, $F_T(\omega)=2\sin(T\omega)/\omega$ has no limit for $\omega\neq0$, although $F_T\to2\pi\delta$ in $\mathcal S'(\mathbb{R})$; for non-tempered $f$, even convergence in $\mathcal S'(\mathbb{R})$ fails (\Cref{prop:tempered}).
\item[(ii)] (Frequency truncation and ordered limits) At each finite stage the inverse integral is formed with symmetric frequency truncation, $I_{T,\varepsilon,\Omega}(t)=\frac1{2\pi}\int_{-\Omega}^{\Omega}F_{T,\varepsilon}(\omega)e^{i\omega t}\,d\omega$, and the limits are taken outside the integral in the prescribed order: $\Omega\to\infty$ first, then $\varepsilon\to0$, then $T\to\infty$.
\item[(iii)] (Pointwise reconstruction of $f$ in the time domain, without convergence of the forward transform) With (i) and (ii), at every $t\notin S$ at which $f$ satisfies a local Dirichlet condition,
\[
\lim_{T\to\infty}\ \lim_{\varepsilon\to0}\ \lim_{\Omega\to\infty}\ \frac1{2\pi}\int_{-\Omega}^{\Omega}F_{T,\varepsilon}(\omega)e^{i\omega t}\,d\omega=\tfrac12\bigl[f(t^-)+f(t^+)\bigr],
\]
which equals $f(t)$ at a point of continuity, for every choice of $T_1,T_2,\varepsilon_1,\varepsilon_2$ (the limits are taken with $T_1,T_2$ eventually containing $t$ and $\varepsilon$ eventually smaller than the distance from $t$ to $S$). No convergence of the family $\{F_{T,\varepsilon}\}$, pointwise or in a generalized sense, no growth condition, no integrability at infinity, and no extension of $f$ across $S$ is used.
\item[(iv)] (The prescription is needed) The order of the limits cannot in general be reversed (\Cref{thm:noninterchange}). A definition that first requires the pointwise limit $F(\omega)=\lim_{T\to\infty}F_T(\omega)$ already fails for $f=1$, as noted in (i). Symmetric frequency truncation is prescribed (Observation~3.1).
\item[(v)] (Beyond tempered distributions) The admissible class contains functions that define no tempered distribution, both with growth at infinity and with severe local singularities, for example $f(t)=e^{\alpha t^2}$ (Example~6) and the essentially singular $f(t)=e^{\alpha/t}$, $\alpha>0$ (Example~11). For these, (i)--(iv) hold, with the exclusion about $t=0$ in the second case. No extension across $t=0$ and no Fourier-compatible distribution is claimed.
\item[(vi)] (Compatibility) If $f$ is locally integrable, also near $S$, with $|f(t)|\le C(1+|t|)^N$ almost everywhere, then for every $g\in\mathcal{S}(\mathbb{R})$ the pairing limit $\lim_{T\to\infty}\int F_T g\,d\omega$ exists and equals $\langle F,g\rangle$ in $\mathcal{S}'(\mathbb{R})$ (Section~3.8).
\end{enumerate}
\end{proposition}

\begin{proof}
(i) and (ii) are the definition, with Example~1. (iii): for finite $T$ and $\varepsilon>0$, $f_{T,\varepsilon}\in L^1(\mathbb{R})$, and classical inversion with symmetric frequency truncation gives $\lim_{\Omega\to\infty}I_{T,\varepsilon,\Omega}(t)=\frac12[f_{T,\varepsilon}(t^-)+f_{T,\varepsilon}(t^+)]$ at points satisfying the local Dirichlet condition. For fixed $t\notin S$, once $t\in(-T_1,T_2)$ and $\varepsilon$ is smaller than the distance from $t$ to $S$, $f_{T,\varepsilon}=f$ in a neighborhood of $t$, so the subsequent limits give $\frac12[f(t^-)+f(t^+)]$; no convergence of $F_{T,\varepsilon}$ as $T\to\infty$ or $\varepsilon\to0$ is used (\Cref{thm:inversion}). (iv) is \Cref{thm:noninterchange} and Observation~3.1. For (v), $e^{\alpha t^2}$ defines no tempered distribution, since pairing it with the Schwartz function $e^{-\alpha t^2/2}$ diverges, and $e^{\alpha/t}$ is not locally integrable at $0^+$, since $\int_0^{\varepsilon}e^{\alpha/t}\,dt$ diverges for every $\varepsilon>0$, so it defines no distribution as it stands. Indeed, a tempered distribution agreeing with $e^{\alpha t^2}$ on $C_c^\infty(\mathbb{R})$ would be determined by continuity, but $\varphi_n(t)=e^{-\alpha t^2/2}\chi(t/n)\to e^{-\alpha t^2/2}$ in $\mathcal S(\mathbb{R})$ for a smooth cutoff $\chi$ (with $\chi=1$ on $[-1,1]$, $0\le\chi\le1$), while $\int e^{\alpha t^2}\varphi_n\,dt\ge\int_{-n}^{n}e^{\alpha t^2/2}dt\to\infty$. Examples~6 and~11 treat them. For (vi), Fubini on $[-T,T]$ gives $\int F_T g\,d\omega=\int_{-T}^{T}f(t)G(t)\,dt$, which converges absolutely to $\langle f,G\rangle=\langle F,g\rangle$, since $f$ has polynomial growth and $G\in\mathcal{S}$.
\end{proof}

\begin{remarks}
Part (iii) is classical Dirichlet--Jordan localization applied to $f_{T,\varepsilon}$, and no new local inversion theory and no new Fourier-compatible spectral extension is claimed. The contribution is the specification of the procedure in (i) and (ii): which truncations are taken, that no convergence of the forward family is required, that the frequency truncation is symmetric while the time truncation is arbitrary, and the order of the limits. With this specification, pointwise inversion holds on the whole admissible class.

The significance lies in the definition. For functions such as $e^{t}$, $e^{t^2}$, $1/t$ and $e^{1/t}$ the conventional forward Fourier integral may fail to converge, yet the ordered inverse transform recovers the function at every nonsingular point of the stated type, so that reconstruction in the time domain is separated from the existence of a convergent forward transform. Fourier analysis serves to move a time function to the frequency domain, operate there, and return; when that purpose is achieved at the finite stages, the canonical existence of the forward transform is not needed. Operations on the finite-stage spectra are justified separately: differentiation by \Cref{thm:deriv}, and multiplication by a transfer function $H$ in Section~4.2, where $y_T=h*x_T$ has transform $HX_T$ and converges to the convolution integral by dominated convergence whenever that integral converges absolutely. For $f=e^{1/t}$ the output of such a convolution with $h(t)\neq0$ diverges as $\varepsilon\to0$, since $\int_{\varepsilon}^{2\varepsilon}h(t-\tau)e^{1/\tau}\,d\tau$ grows exponentially; the divergence of the output reflects the divergence of $f$, and no finite part is assigned.
\end{remarks}

\begin{proposition}[Convergence of the forward transform cannot be required beyond tempered distributions]\label{prop:tempered}
Let $f\in L^1_{\mathrm{loc}}(\mathbb{R})$, $f_T=f\cdot\mathbf 1_{[-T_1,T_2]}$ and $F_T=\hat f_T$. If $F_T$ converges in $\mathcal S'(\mathbb{R})$ as $T_1,T_2\to\infty$, then the pairing $\int_{-T_1}^{T_2}f(t)G(t)\,dt$ converges for every Schwartz function $G$. Hence, if for some $G\in\mathcal S(\mathbb{R})$ this pairing is not determined, then $f$ is not tempered and $\{F_T\}$ does not converge in $\mathcal S'(\mathbb{R})$. This is the case for $f=e^{\alpha t^2}$ ($\alpha>0$), with $G=e^{-\alpha t^2/2}$, since $e^{\alpha t^2/2}$ is not integrable at infinity.
\end{proposition}

\begin{proof}
For $g\in\mathcal S$ put $G=\hat g\in\mathcal S$; Fubini on the compact interval gives $\langle F_T,g\rangle=\int_{-T_1}^{T_2}fG\,dt$, and $G$ runs through all of $\mathcal S$ as $g$ does. Convergence in $\mathcal S'$ means convergence of each such pairing.
\end{proof}

\begin{remarks}
Together with \Cref{prop:recon}(vi): for $f$ of polynomial growth the forward transform converges in $\mathcal S'$ and agrees with the classical transform; for a non-tempered $f$ it cannot converge in $\mathcal S'$, and a definition that requires convergence in $\mathcal S'$ cannot reach such $f$, whereas the t.g.l.\ construction does not require it; in the space $Z'$ of \Cref{prop:weaklimit}, convergence does hold for every $f\in L^1_{\mathrm{loc}}$, with the limitations noted there. The proposition concerns $\mathcal S'$ only; other theories (for example, ultradistributions or analytic functionals) may give a forward transform in other spaces, and none is claimed to be impossible.
\end{remarks}

\subsection{Operational meaning of generalized convergence}

For functions whose truncated transforms $F_T(\omega)$ do not converge pointwise as $T\to\infty$, generalized spectral meaning may still be obtained through the pairing
\[
\frac{1}{2\pi}\operatorname{g.lim}_{\Omega\to\infty}\int_{-\Omega}^{\Omega}F_T(\omega)G^*(\omega)\,d\omega
\]
with an appropriate auxiliary function $G$, as formulated in \Cref{thm:pairing}, where the frequency truncation $\Omega$ is removed first, as in \Cref{thm:inversion}. This operational convergence is distinct from pointwise convergence of $F_T$.

Inverse reconstruction uses a different mechanism. Substitution of $F_T$ into the frequency-truncated inverse produces convolution with the Dirichlet kernel, whose oscillatory localization yields pointwise reconstruction under the conditions of \Cref{thm:inversion}. Thus forward generalized spectral meaning is obtained through auxiliary pairing, whereas inverse reconstruction is obtained through Dirichlet localization. Further interpretation is given in Section~6.

\subsection{Constructive interpretation of generalized Fourier transforms}

The t.g.l.\ formulation does not begin by introducing a new generalized-function space. For every finite truncation, $f_{T,\varepsilon}\in L^1(\mathbb{R})$, and
\[
F_{T,\varepsilon}(\omega) = \int f_{T,\varepsilon}(t)e^{-i\omega t}\,dt
\]
is an ordinary Fourier integral. Generalized meaning enters only through the prescribed subsequent limits and pairings.

If $F_T(\omega)$ converges with sufficient uniformity to justify passage of the limit through the frequency integral, the ordinary Fourier procedure applies. When it does not, the t.g.l.\ inverse instead follows the prescribed ordered construction
\[
f \longrightarrow f_T \longrightarrow F_T \longrightarrow f_T \longrightarrow f.
\]
Possible reversal of the two limiting operations must be justified separately rather than assumed; see \Cref{thm:noninterchange} and its Remarks. Broader interpretation is given in Section~6.

\subsection{Integral Structures and Ordered-Limit Mechanisms}

A characteristic feature of the t.g.l.\ formulation is that ordinary Fourier integration is completed at finite truncation before the generalized truncation limit is taken. Consequently, boundary terms arising from integration by parts remain explicit at the finite stage and must be combined with their compensating contributions before removal of the truncation.

\Cref{thm:s2one} gives the corresponding delta-compensated cancellation, while Section~2.6 shows how the same mechanism may be represented by delta-convergent ordinary approximations. This finite-stage treatment makes potentially divergent boundary contributions explicit rather than suppressing them through an assumed limiting cancellation.

\subsection{Relation to the Distributional Fourier Transform}

For functions outside the classical Fourier setting, distribution theory provides a standard generalized framework for Fourier transformation. Since some readers of this paper may not be familiar with this framework, we briefly summarize the concepts needed for comparison with the t.g.l.\ formulation.

A distribution is characterized through its action on suitable test functions, denoted $\langle f,g\rangle$; tempered distributions act on the space $\mathcal{S}(\mathbb{R})$ of Schwartz test functions. Schwartz test functions are smooth and rapidly decreasing, which allows generalized objects with no pointwise meaning of their own to be handled through such finite pairings.

For a distribution $f\in\mathcal{S}'(\mathbb{R})$, the Fourier transform $F$ and its inverse $\overleftarrow{f}$ are defined as a functional [7]:
\[
\langle F,g\rangle = \langle f,G\rangle, \qquad \langle \overleftarrow{f},G\rangle = \langle F,\overleftarrow{g}\rangle,
\]
where $g\in\mathcal{S}(\mathbb{R})$ is a test function and $G\in\mathcal{S}(\mathbb{R})$ is its Fourier transform ($g\leftrightarrow G$) [9]: the transform is thus defined indirectly, through its action on test functions, rather than pointwise. When $f\in L^1(\mathbb{R})$, so that $F$ also exists as an ordinary function, interchanging the order of integration gives
\[
\langle F(\omega),g(\omega)\rangle = \int_{-\infty}^{\infty} F(\omega)\, g(\omega)\,d\omega = \int_{-\infty}^{\infty}\int_{-\infty}^{\infty} f(t)\, e^{-i\omega t}\,dt\, g(\omega)\,d\omega = \int_{-\infty}^{\infty} f(t)\, G(t)\,dt,
\]
so the distributional definition reproduces the classical Fourier transform whenever the latter exists.

Another definition by generalized function is known as [8]:
\[
\langle F,G\rangle = 2\pi\langle f,g\rangle, \qquad \langle\overleftarrow{f},\overleftarrow{g}\rangle = \frac{1}{2\pi}\langle F,G\rangle.
\]
In this case for $f\in L^1(\mathbb{R})$, we have a well known result as:
\[
\int_{-\infty}^{\infty} f(t)\, g(t)\,dt = \frac{1}{2\pi}\int_{-\infty}^{\infty} F(\omega)\, G(\omega)\,d\omega,
\]
where $G(\omega)$ not $G(-\omega)$ appears due to the definition with $\exp(i\omega t)$ instead of $\exp(-i\omega t)$. In this paper we assume the definition by distribution.

We have shown that although the t.g.l.\ forward transform $F(\omega)$ diverges for $f\notin L^1(\mathbb{R})$, the t.g.l.\ integral of $F(\omega)$ with a well-defined auxiliary $G(\omega)$ converges (\Cref{thm:pairing}). We now make explicit how this construction recovers the distributional pairing $\langle F,g\rangle=\langle f,G\rangle$ itself.

For the truncated function $f_T$ and its ordinary Fourier transform
\[
F_T(\omega) = \int_{-T}^{T} f(t)\,e^{-i\omega t}\,dt,
\]
finite-stage interchange of the ordinary integrations gives
\[
\int_{-\infty}^{\infty} F_T(\omega)\,g(\omega)\,d\omega = \int_{-T}^{T} f(t)\,G(t)\,dt = \langle f_T,G\rangle.
\]
Taking the truncation limit,
\[
\lim_{T\to\infty}\int_{-\infty}^{\infty} F_T(\omega)\,g(\omega)\,d\omega = \lim_{T\to\infty}\int_{-T}^{T} f(t)\,G(t)\,dt,
\]
and, when the right-hand side defines the distributional pairing $\langle f,G\rangle$, this gives
\[
\langle F,g\rangle = \lim_{T\to\infty}\int_{-\infty}^{\infty} F_T(\omega)\,g(\omega)\,d\omega \qquad \text{and hence} \qquad \langle F,g\rangle = \langle f,G\rangle.
\]

This relation gives an elementary interpretation of the distributional Fourier transform. Rather than regarding $F$ first as an abstract generalized function, one may begin with the ordinary Fourier transforms $F_T$ of the truncated functions and define their limiting action on a test function $g$. When this limit exists for the required test-function class, it reproduces the distributional Fourier transform. Thus the distributional pairing can be recovered from ordinary finite-stage Fourier integrals followed by a generalized limiting operation.

In this sense, the t.g.l.\ construction provides a classical finite-stage interpretation of the distributional Fourier-transform definition: generalized meaning enters through the limiting action of the ordinary transforms $F_T$, rather than through a pointwise limit $F_T(\omega)\to F(\omega)$.

The existence of this limit for one particular $g$ does not by itself establish convergence in the distributional sense. Distributional convergence requires the corresponding limiting relation for the specified test-function class. The analogous inverse-side pairing $\langle \overleftarrow{f}_T,G\rangle = \langle F_T,\overleftarrow{g}\rangle$ holds at every finite $T$, giving $\overleftarrow{f}_{\mathrm{tgl}} = \lim_{T\to\infty} f_T = f$ in $\mathcal{S}'(\mathbb{R})$.

Compactly supported pairing functions correspond to truncated functions, and the
truncation can be carried through both sides of the pairing. Let
$G_T=G\cdot\mathbf 1_{[-T,T]}$ be the truncated pairing function, which has
compact support, and let $g_T$ be the function with $g_T\leftrightarrow G_T$.
The forward pairings are then related by
\[
\langle f,G_T\rangle=\langle f_T,G\rangle=\langle F_T,g\rangle=\langle F,g_T\rangle ,
\]
and the inverse pairings by
\[
\langle \overleftarrow{f},G_T\rangle=\langle \overleftarrow{f}_T,G\rangle=\langle F_T,\overleftarrow{g}\rangle
=\langle f_T,G\rangle .
\]
Hence $\overleftarrow{f}_T=f_T$, which coincides with $f$ on $(-T,T)$ but does not recover $f$ outside $[-T,T]$. Thus, with compactly supported pairing functions, as in the test-function space $\mathcal{D}(\mathbb{R})$, the
Fourier transform and its inverse correspond to the truncated transform $F_T$
and the truncated function $f_T$, without the limit $T\to\infty$; recovering $f$
requires this limit, which in the t.g.l.\ formulation is taken outside the
pairing.

The present t.g.l.\ formulation and distributional Fourier theory share a common structural feature in that generalized Fourier meaning is ultimately determined through dual-domain pairing operations. However, the inverse reconstruction mechanism differs fundamentally. The t.g.l.\ formulation introduces explicit frequency-domain localization through finite reciprocal-domain truncation, whereby pointwise inverse reconstruction is obtained directly through ordered localization limits generated by ordinary oscillatory integrals without introducing auxiliary test-function pairings into the inverse reconstruction process itself.

The Dirac delta is defined distributionally by
\[
\langle\delta,g\rangle = g(0), \qquad \text{i.e.,} \qquad \langle \delta(t),f(t)\rangle = \int_{-\infty}^{\infty} \delta(t)\, f(t)\,dt = f(0).
\]
More precisely, the delta function $\delta(t)$ and delta-convergent sequence of functions defined by the generalized limit are defined as [13]:
\[
\lim_{\Omega\to\infty} \int_a^b \delta_\Omega(t)\, f(t)\,dt = \int_a^b \delta(t)\, f(t)\,dt =
\begin{cases}
\big[f(0^+)+f(0^-)\big]/2 & (ab<0) \\
0 & (ab>0),
\end{cases}
\]
where $a$ and $b$ are constants that take any small or infinite value.

It can also be approached by families of ordinary functions whose mass becomes increasingly concentrated near the origin. If $\int \delta_\varepsilon(t)\,g(t)\,dt \to g(0)$, the family $\delta_\varepsilon$ is said here to converge to $\delta$ in the distributional sense. Gaussian, Lorentzian, rectangular, and Dirichlet-type kernels provide familiar examples of such localization,
\[
\lim_{\sigma\to0} \frac{1}{\sqrt{2\pi}\,\sigma}\, e^{-t^2/2\sigma^2} = \lim_{\sigma\to0} \frac{\sigma}{\pi(\sigma^2+t^2)} = \lim_{\varepsilon\to0} \frac{1}{2\varepsilon}\, \mathbf{1}_{[-\varepsilon,\varepsilon]}(t) = \delta(t).
\]
The admissible class of functions depends on the particular delta-convergent sequence used; in general, global $L^1(\mathbb{R})$ integrability is not required.

It is worthy to note the expression
\[
\lim_{\sigma\to0}\frac{\sigma}{t^2+\sigma^2} = \lim_{\sigma\to0}\mathrm{Im}\left[\frac{1}{t-i\sigma}\right] \neq \frac{1}{t}.
\]
The correct result is given by
\[
\lim_{\sigma\to0}\frac{1}{t-i\sigma} = \lim_{\sigma\to0}\frac{t+i\sigma}{t^2+\sigma^2} = \frac{1}{t}+i\pi\delta(t),
\]
where
\[
\lim_{\sigma\to0}\frac{\sigma}{t^2+\sigma^2} = \pi\delta(t)
\]
is proved in Example~1.

Dirichlet-type kernels arising from frequency truncation form one such delta-convergent family,
\[
D_\Omega(t) = \frac{\sin(\Omega t)}{\pi t} = \frac{1}{2\pi}\int_{-\Omega}^{\Omega} e^{i\omega t}\,d\omega, \qquad \lim_{\Omega\to\infty}\int_{-\infty}^{\infty} D_\Omega(t)\,f(t)\,dt = f(0) \quad (f\in L^1(\mathbb{R})),
\]
the latter proved in \Cref{thm:c1}. The Dirichlet kernel differs from positive delta-convergent families because localization is produced by oscillatory cancellation. Nevertheless, under the conditions used in Fourier inversion, its pairing with a suitable function approaches the value of that function at the localization point; the full localization argument is given in Section~3.3 and \Cref{thm:inversion}.

For locally singular functions, distribution theory may assign generalized meaning through regularizations such as the Cauchy principal value and, for stronger singularities, the Hadamard finite part. These constructions extend the action of singular expressions to suitable test functions even when the corresponding ordinary improper integrals do not exist; Section~5 treats $1/t$ and $1/t^n$ as t.g.l.\ examples of this kind.

Distributional differentiation transfers the derivative from a possibly nonsmooth generalized function to the smooth test function,
\[
\langle f',g\rangle = -\langle f,g'\rangle, \qquad \langle f^{(n)},g\rangle = (-1)^n\langle f,g^{(n)}\rangle,
\]
permitting derivatives to be defined even when an ordinary pointwise derivative does not exist. With $g(t)=e^{-i\omega t}$ this gives the familiar Fourier relation
\[
f^{(n)}(t) \leftrightarrow (i\omega)^n F(\omega).
\]
In the t.g.l.\ formulation, the corresponding relation (\Cref{thm:deriv}) is obtained constructively through finite-stage integration by parts, including the truncation-generated boundary contributions described in Section~2.6.

Another constructive route is to approximate a nonsmooth or singular object by a sequence of smoother ordinary functions and then take an appropriate generalized limit. Such sequential or smoothing approaches provide an intuitive bridge between ordinary functions and generalized functions. The t.g.l.\ formulation is related in spirit but uses truncation and ordered dual-domain limits as its primary construction.

Within distribution theory, regular, principal-value, finite-part, and derivative-type singular objects are mathematically unified as tempered distributions in the common dual space $\mathcal{S}'(\mathbb{R})$. For example,
\[
1,\quad \text{p.v.}(1/t),\quad \text{f.p.}(1/t^2),
\]
and derivative-type singular distributions such as $\delta'(t)$ all belong to $\mathcal{S}'(\mathbb{R})$, though their constructive realizations require different limiting interpretations depending on the local nature of the singularity.

The two frameworks need not have identical scope: some rapidly growing ordinary functions may fall outside the tempered-distribution setting while still admitting ordinary Fourier transforms at every finite truncation, as discussed further in Section~6.3. The t.g.l.\ formulation examines such cases through the existence of the prescribed finite-stage operations and subsequent limits; this is not a claim that the t.g.l.\ construction is generally broader than distribution theory or the other generalized-function theories discussed in the Introduction.

\begin{proposition}[Weak limit of the finite-stage family]\label{prop:weaklimit}
Let $f\in L^1_{\mathrm{loc}}(\mathbb{R})$, $f_T=f\,\mathbf 1_{[-T,T]}$ and $F_T=\hat f_T$. Let $\varphi\in\mathcal S(\mathbb{R})$ and $\hat\varphi(t)=\int\varphi(\omega)e^{-i\omega t}\,d\omega$, and suppose $f\hat\varphi\in L^1(\mathbb{R})$. Then
\[
\int_{-\infty}^{\infty}F_T(\omega)\varphi(\omega)\,d\omega=\int_{-T}^{T}f(t)\hat\varphi(t)\,dt\ \longrightarrow\ \int_{-\infty}^{\infty}f(t)\hat\varphi(t)\,dt=:\langle F,\varphi\rangle\qquad(T\to\infty).
\]
In particular:
\begin{enumerate}
\item[(a)] if $|f(t)|\le C(1+|t|)^N$, every $\varphi\in\mathcal S(\mathbb{R})$ qualifies, and $\langle F,\varphi\rangle$ is the classical Fourier transform of $f\in\mathcal S'(\mathbb{R})$;
\item[(b)] for every $f\in L^1_{\mathrm{loc}}(\mathbb{R})$, every $\varphi\in\mathcal S(\mathbb{R})$ with $\hat\varphi\in C_c^\infty(\mathbb{R})$ qualifies; if $\operatorname{supp}\hat\varphi\subset[-R,R]$, the pairing is constant for $T\ge R$, so $F_T\to\hat f$ weakly in the dual $Z'$ of the space $Z$ of such $\varphi$ (the Gel'fand--Shilov space of analytic functionals), where $\hat f$ is the Fourier transform of $f\in\mathcal D'(\mathbb{R})$.
\end{enumerate}
\end{proposition}

\begin{proof}
Since $f_T\in L^1$ and $\varphi\in L^1$, Fubini's theorem gives the first equality. If $f\hat\varphi\in L^1$, dominated convergence gives the limit. For (a), $\hat\varphi\in\mathcal S$ and $f$ has polynomial growth, so $f\hat\varphi\in L^1$. For (b), $\hat\varphi$ has compact support and $f$ is locally integrable, so $f\hat\varphi\in L^1$ and the integral over $[-T,T]$ equals the full integral once $T\ge R$.
\end{proof}

\begin{remarks}
For these test classes the limit is the transform of the established theory: the tempered distribution in (a), and the transform in the dual $Z'$ of the Gel'fand--Shilov space $Z$ in (b). The proposition thus shows agreement with those theories, not a difference. The transform in (b) is not an independent construction: pairing with a test function $\varphi$ whose Fourier image $\hat\varphi$ has compact support amounts to truncating $f$ in the time domain, that is, to convolving $\hat f$ with the Fourier transform of a truncation function. The truncation that the t.g.l.\ formulation performs on $f$ itself is here absorbed into the test space. It also shows what they require: since $\hat\varphi$ has compact support, the pairing in (b) depends only on $f$ on a bounded interval, and $Z$ is a very small space, so elements of $Z'$ are analytic functionals without values at individual frequencies $\omega$, and they cannot in general be multiplied by a function $H(\omega)$ as needed in Section~4.2. For $f=e^{\alpha/t}$, which is not locally integrable at $0^+$, the proposition does not apply, and the exclusion family $F_{T,\varepsilon}$ is needed (Section~6.4).
\end{remarks}

\subsection{Duality Structures in Generalized Fourier Formulations}

In distribution theory, the pairing $\langle F,g\rangle$ specifies the action of a generalized function on a prescribed class of test functions, and convergence is understood with respect to that test-function space. In the t.g.l.\ formulation, a second-stage pairing with an auxiliary function $G(\omega)$ may instead be used to assign generalized meaning to a finite-stage transform family $F_T(\omega)$ (Section~2.4). The formal resemblance of the two pairings does not make them identical, for the reason already noted in Section~3.8.

When the t.g.l.\ limiting pairing exists for the Schwartz test-function class under the assumptions of Section~3.8, it reproduces the corresponding distributional Fourier-transform functional. Thus the two constructions agree on their common domain even though their starting points are different.

The inverse t.g.l.\ construction has a different purpose. At fixed finite time truncation, symmetric frequency truncation produces the Dirichlet kernel, and the limit $\Omega\to\infty$ reconstructs the truncated function under the conditions of \Cref{thm:inversion}. The subsequent time-truncation limit then recovers the original function. This pointwise reconstruction should therefore be distinguished from both distributional test-function pairing and the second-stage generalized spectral pairing.

Broader comparison with distributional and sequential Fourier theories, including Lighthill's and Feichtinger's constructions, ultradistributions, and hyperfunctions, is given in Sections~1 and~6.3.

The common feature of these formulations is therefore the use of pairing to extract meaningful limiting information from objects that may not possess ordinary pointwise Fourier limits. Their mathematical meanings nevertheless depend on the class of admissible pairing functions and on the limiting procedure being used.

\section{Applications to Linear Systems and Differential Equations}

The preceding sections established the t.g.l.\ construction for generalized Fourier transformation and inversion. We now examine its use as an operational tool in differential equations and linear time-invariant systems. The emphasis here is not on individual transform pairs, which are treated in Section~5, but on structural operations such as differentiation, convolution, and Green-function reconstruction.

The discussion distinguishes generalized Fourier representation from system stability. Absolute integrability of the impulse response, $h\in L^1(\mathbb{R})$, provides the usual BIBO-stability condition; the existence of a t.g.l.\ Fourier representation for $h$ or $x$ does not by itself imply this stability, and the two properties are kept logically distinct throughout this section.

In classical system theory the impulse response representation and the convolution formula are often introduced using the Dirac delta distribution. This formulation is mathematically justified within distribution theory and provides a standard framework for describing linear time-invariant systems. In the present section we first recall this conventional representation in order to establish the connection with existing theory. We then reinterpret the same constructions within the truncate-and-generalized-limit framework by replacing the idealized delta distribution with delta-convergent kernels and by introducing truncation-based limiting procedures. In this way the transition from distributional formulations to the truncate-and-generalized-limit approach is made explicit, and no ambiguity arises concerning the interpretation of the impulse-response representation used below.

\subsection{Linear differential equations and linear time-invariant systems}

Consider a linear differential equation for functions $x(t)$ and $y(t)$, not necessarily belonging to $L^1(\mathbb{R})$:
\[
\sum_{m=0}^{M} a_m\, \frac{d^m}{dt^m} y(t) = \sum_{n=0}^{N} b_n\, \frac{d^n}{dt^n} x(t).
\]
Let the input be $x(t) = \delta(t)$, where the delta symbol is first understood in the conventional distributional sense used in linear-system theory; it will later be replaced by delta-convergent kernels in the truncate-and-generalized-limit formulation, with the initial conditions taken to be null.

The corresponding output is denoted by $h(t)$, which is called the impulse response of the system. Because the coefficients are constant and the initial condition is null, the system is linear and time-invariant. Then the response to a weighted and shifted impulse $x(\tau)\delta(t-\tau)$ is $x(\tau)h(t-\tau)$.

Formally, a signal may be represented through superposition of shifted impulses as
\[
x(t) = \int_{-\infty}^{\infty} \delta(t-\tau)\, x(\tau)\,d\tau.
\]
The corresponding output is obtained by superposition:
\[
y(t) = \int_{-\infty}^{\infty} h(t-\tau)\, x(\tau)\,d\tau = \int_{-\infty}^{\infty} h(\tau)\, x(t-\tau)\,d\tau,
\]
which is the convolution representation of the system response.

Taking the Fourier transform yields
\[
Y(\omega) = H(\omega)\, X(\omega)
\]
where $x(t)\leftrightarrow X(\omega)$, $h(t)\leftrightarrow H(\omega)$ and $y(t)\leftrightarrow Y(\omega)$, provided that the transforms exist.

Thus the system can be characterized entirely by its impulse response $h(t)$. A linear time-invariant system is BIBO stable --- bounded input producing bounded output --- if and only if $h(t) \in L^1(\mathbb{R})$ [13].

Formally, the inverse Fourier representation of the system response may be written as
\[
y(t) = \frac{1}{2\pi}\int_{-\infty}^{\infty} H(\omega)\, X(\omega)\, e^{i\omega t}\,d\omega,
\]
provided that the integral exists in an appropriate sense. The precise interpretation of this representation will be discussed below.

However, if $X(\omega)$ does not converge in the classical sense, the integral representation above cannot be evaluated directly. This illustrates a limitation of the classical Fourier transform, which is defined only through convergent integrals.

Distribution theory resolves this difficulty by defining Fourier transforms through pairing with test functions. In contrast, the truncate-and-generalized-limit formulation provides a constructive alternative based on truncation of the signal followed by an ordered limiting process.

Let $x_T(t)$ denote a truncated version of the input signal, with $x_T(t) \to x(t)$ as $T\to\infty$. Assume that, at the point $t$ under consideration, the product $h(t-\tau)x_T(\tau)$ is integrable in $\tau$. Then
\[
y_T(t) = \int_{-\infty}^{\infty} h(t-\tau)\, x_T(\tau)\,d\tau,
\]
is an ordinary finite-stage integral, and we define $y(t) = \lim_{T\to\infty} y_T(t)$.
At each finite truncation stage, when the convolution theorem is applicable,
\[
Y_T(\omega) = H(\omega)\, X_T(\omega),
\]
and $y_T$ is reconstructed from this finite-stage product. The limit $T\to\infty$ is then taken only after the finite-stage operation has been completed; no pointwise existence of an untruncated $X(\omega)$ is required.
Because the truncated signal $x_T(t)$ belongs to $L^1(\mathbb{R})$, the Fourier transform $X_T(\omega)$ exists classically. The generalized limit $T\to\infty$ is then applied only after the finite-stage integral has been evaluated, providing a constructive realization of the system response even when the original signal $x(t)$ is not absolutely integrable.

The finite-stage construction is justified by the following elementary statement.

\begin{proposition}\label{prop:lti}
Let $h\in L^1(\mathbb{R})$ be of bounded variation on $\mathbb{R}$, let $x\in L^1_{\mathrm{loc}}(\mathbb{R})$, and let $x_T=x\,\mathbf 1_{[-T,T]}$. Then, for every $T>0$,
\[
y_T=h*x_T\in L^1(\mathbb{R}),\qquad Y_T(\omega)=H(\omega)X_T(\omega),
\]
and, for every $t\in\mathbb{R}$,
\[
y_T(t)=\lim_{\Omega\to\infty}\frac{1}{2\pi}\int_{-\Omega}^{\Omega}H(\omega)X_T(\omega)e^{i\omega t}\,d\omega .
\]
If, moreover, $\tau\mapsto h(t-\tau)x(\tau)$ is absolutely integrable at a given $t$, then
\[
\lim_{T\to\infty}y_T(t)=\int_{-\infty}^{\infty}h(t-\tau)x(\tau)\,d\tau=(h*x)(t).
\]
\end{proposition}

\begin{proof}
Since $h,x_T\in L^1(\mathbb{R})$, we have $y_T\in L^1(\mathbb{R})$ and the convolution theorem gives $Y_T=HX_T$. Since $h$ is bounded and $x_T\in L^1$, $y_T$ is continuous and bounded, and its total variation is at most $\mathrm{TV}(h)\,\|x_T\|_1$. The Dirichlet--Jordan theorem therefore gives the stated inversion at every $t$. Finally $|h(t-\tau)x_T(\tau)|\le|h(t-\tau)x(\tau)|$, so dominated convergence gives the last limit.
\end{proof}

\begin{remarks}
\emph{Faithfulness.} In the setting of \Cref{prop:lti}, the frequency-domain computation at every finite stage reproduces the time-domain computation on the truncated input exactly: $y_T(t)=\int_{-T}^{T}h(t-\tau)x(\tau)\,d\tau$ is both the inverse transform of $HX_T$ and the output of the system for the input $x_T$. Hence $\lim_{T\to\infty}y_T(t)$ exists if and only if the corresponding truncated time-domain integral converges, and the two values then coincide; likewise, the pairing of \Cref{thm:pairing} equals $\lim\int f_{T,\varepsilon}\,g_\Omega^{*}\,dt$, a time-domain quantity. The t.g.l.\ output therefore neither creates convergence where the time-domain problem diverges nor destroys it where the time-domain problem converges, in contrast with the formal frequency-domain routes discussed in Section~6.5.

This bears on the absence of a canonical forward transform. When the purpose of Fourier analysis is to move a time-domain problem to the frequency domain, operate there, and return, the criterion of correctness is fidelity to the time-domain problem, and a canonical forward spectrum is one means to that end, not the criterion itself. The t.g.l.\ formulation meets the criterion without it, so for this purpose the missing canonical forward transform is a limited loss. It remains a loss where the spectrum itself is the object of interest, for example for uniqueness of a spectrum, Plancherel-type identities, or continuity of the transform; the formulation makes no claim there. The method adds no information beyond the truncated time-domain problem; its role is to justify frequency-domain operations, such as multiplication by $H(\omega)$ or differentiation (\Cref{thm:deriv}), at each finite stage without requiring the forward transform to converge.
\end{remarks}

The expression $x(t)=\delta(t)$ is not itself an ordinary function. Instead, we now replace the distributional impulse by a delta-convergent kernel (cf.\ Examples~9 and~10)
\[
\delta_\Omega(t) = \frac{\sin(\Omega t)}{\pi t}
\]
which provides a truncation-based realization of the impulse within the truncate-and-generalized-limit framework.

In accordance with the definition of delta-convergent kernels used in the present framework, the limit $\Omega\to\infty$ is applied only after the integral has been evaluated.

Let the response to $\delta_\Omega(t)$ be denoted by $h_\Omega(t)$. Then for a general input represented formally as
\[
x(t) = \lim_{\Omega\to\infty} \int_{-\infty}^{\infty} \delta_\Omega(t-\tau)\, x(\tau)\,d\tau,
\]
the corresponding output is
\[
y(t) = \lim_{\Omega\to\infty} \int_{-\infty}^{\infty} h_\Omega(t-\tau)\, x(\tau)\,d\tau = \int_{-\infty}^{\infty} h(t-\tau)\, x(\tau)\,d\tau.
\]
Other delta-convergent kernels such as Gaussian approximations may also be used in place of $\delta_\Omega(t)$.

Consider a system with impulse response
\[
h(t) = \delta_\varepsilon(t) = \frac{1}{2\varepsilon}\, \mathbf{1}_{[-\varepsilon,\varepsilon]}(t),
\]
where $\delta_\varepsilon(t)$ is a delta-convergent kernel with width parameter $\varepsilon$.

For an input signal $x(t)=\delta_\Omega(t)$, the output becomes
\[
y(t) = \lim_{\Omega\to\infty} \int_{-\infty}^{\infty} \delta_\varepsilon(t-\tau)\, \delta_\Omega(\tau)\,d\tau = \delta_\varepsilon(t).
\]
Thus the output differs qualitatively from the input: whereas the input kernel spreads across an infinite time interval with envelope $1/t$, the output is confined to $t\in[-\varepsilon,\varepsilon]$, reflecting the system's finite time resolution --- $\varepsilon$ represents the minimum detectable time scale, and signals of shorter duration cannot be resolved accurately.

Let
\[
H_\varepsilon(\omega) = \frac{\sin\omega\varepsilon}{\omega\varepsilon} \leftrightarrow \delta_\varepsilon(t).
\]
If we define an effective bandwidth $\Omega$ for $H_\varepsilon(\omega)$ as the smallest $|\omega|$ at which $H_\varepsilon(\omega) = 0$, then
\[
\Omega\varepsilon = \pi.
\]
Since the pulse width $\Delta d$ is twice $\varepsilon$, $\Omega\,\Delta d = 2\pi$. This expresses a reciprocal relation between time width and effective frequency bandwidth associated with the kernel.

Thus the impulse-response representation of linear time-invariant systems may be interpreted equivalently either within the distributional framework or within the truncate-and-generalized-limit formulation, the latter providing a constructive realization based on truncation and ordered limiting processes applied directly to ordinary functions.

\subsection{Application of the t.g.l.\ Fourier transform to differential equations}

We now apply the t.g.l.\ Fourier transform to the linear constant-coefficient differential equation introduced above:
\[
\sum_{m=0}^{M} a_m\, \frac{d^m}{dt^m} y(t) = \sum_{n=0}^{N} b_n\, \frac{d^n}{dt^n} x(t).
\]
When the ordinary Fourier integrals are not absolutely convergent, we work at the finite stage with the response to the truncated input. Let $h$ be the impulse response, that is, $\sum_{m} a_m h^{(m)}=\sum_{n} b_n\delta^{(n)}$ in the weak sense (Section~4.1), let $x_T=x\,\mathbf 1_{[-T,T]}$, and set $y_T=h*x_T$. Then $y_T$ is the response to the truncated input; it is not the truncation $y\,\mathbf 1_{[-T,T]}$ of the solution. Truncation does not commute with differential operators: for independently truncated $y$ and $x$ the weak derivatives carry boundary delta terms (\Cref{thm:deriv}) that need not agree, so the truncated functions satisfy the equation only up to terms supported at $\pm T$. For $y_T=h*x_T$, however,
\[
\sum_{m=0}^{M} a_m\, y_T^{(m)}=\sum_{n=0}^{N} b_n\, x_T^{(n)},
\]
where the boundary deltas of $x_T$ are included on the right (\Cref{thm:deriv}(i)). Since $h\in L^1(\mathbb{R})$ and $x_T\in L^1(\mathbb{R})$, $Y_T(\omega)=H(\omega)X_T(\omega)$, with
\[
H(\omega)=\frac{\sum_{n=0}^{N} b_n(i\omega)^n}{\sum_{m=0}^{M} a_m(i\omega)^m}
\]
the transform of $h$.
Taking the inverse transform gives, under the assumptions $h\in L^1(\mathbb{R})$ and $x\in L^\infty(\mathbb{R})$ --- so that $|h(t-\tau)x_T(\tau)|\le\|x\|_\infty|h(t-\tau)|$ justifies passage to the limit by dominated convergence ---
\[
y(t) = \lim_{T\to\infty} y_T(t) = \lim_{T\to\infty} \int_{-\infty}^{\infty} h(t-\tau)\, x_T(\tau)\,d\tau = \int_{-\infty}^{\infty} h(t-\tau)\, x(\tau)\,d\tau.
\]
Thus the usual frequency-domain solution extends constructively to bounded inputs outside $L^1(\mathbb{R})$ whenever the stated convolution conditions hold.

\emph{Example.} For $h_a(t)=\sqrt{a/\pi}\,e^{-at^2}$, with $H_a(\omega)=e^{-\omega^2/4a}$, and $x(t)=t$, the transform $X_T(\omega)$ has no limit as $T\to\infty$, yet $y_T(t)=\int_{-T}^{T}h_a(t-\tau)\tau\,d\tau\to t$, since $|\tau|h_a(t-\tau)$ is integrable and $h_a$ has unit mass and zero first moment.

\emph{Example.} For the same $h_a$ (so $H_a\in L^1$) and $x(t)=e^{\alpha t^2}$, $0<\alpha<a$, which has no distributional transform, the convolution integral converges absolutely since $0<\alpha<a$, and $y_T(t)=\int_{-T}^{T}h_a(t-\tau)e^{\alpha\tau^2}\,d\tau\to\sqrt{a/(a-\alpha)}\,e^{a\alpha t^2/(a-\alpha)}$; this is the convolution integral itself, so the example shows consistency beyond the tempered setting, not necessity.

This construction also clarifies a conceptual point about the classical distributional Fourier transform's defining pairing, $\langle F,g\rangle=\langle f,G\rangle$ (Section~3.9): that pairing can itself be read as applying the transform to a system. If the test function's transform $G$ is identified with a system's impulse response and $f$ with the system input, the pairing computes the same convolution-type system output obtained constructively above. This correspondence is made explicit for the t.g.l.\ construction in Section~6.2, where the auxiliary pairing function specializes to the system's transfer function $H(\omega)$, and the resulting generalized limit reduces to the system output evaluated at the origin.

\subsection{Eigenfunction approach to non-$L^1$ systems}

Complex-exponential eigenfunctions provide another frequency-domain interpretation of linear systems that does not require the signals themselves to belong to $L^1(\mathbb{R})$. Consider again the linear constant-coefficient differential equation
\[
\sum_{m=0}^{M} a_m\, \frac{d^m}{dt^m} y(t) = \sum_{n=0}^{N} b_n\, \frac{d^n}{dt^n} x(t).
\]
A standard method for analyzing such equations is the eigenfunction approach used in alternating-current theory. Complex exponential functions
\[
x(t) = X\, e^{i\omega t}
\]
are eigenfunctions of linear time-invariant operators with constant coefficients. Substituting the trial solutions $x(t) = X e^{i\omega t}$, $y(t) = Y e^{i\omega t}$ into the differential equation yields
\[
Y = H(\omega)\, X
\]
where $H(\omega)$ is the eigenvalue, which is called the frequency transfer function of the system. The same relation follows from the convolution representation
\[
y(t) = \int_{-\infty}^{\infty} h(t-\tau)\, x(\tau)\,d\tau,
\]
for exponential inputs $x(t) = X e^{i\omega t}$, we obtain
\[
y(t) = X H(\omega)\, e^{i\omega t} = H(\omega)\, x(t).
\]
Thus complex exponentials act as eigenfunctions of linear time-invariant systems, and the transfer function $H(\omega)$ plays the role of the corresponding eigenvalue. For example,
\[
y(t) = \frac{d}{dt} x(t) \quad \Rightarrow \quad y(t) = i\omega\, x(t),
\]
and
\[
y(t) = \int x(t)\,dt \quad \Rightarrow \quad y(t) = \frac{1}{i\omega}\, x(t), \qquad \omega\neq0,
\]
where the homogeneous integration constant is taken to be zero.
These relations remain meaningful for the exponential eigenfunctions although such functions are not absolutely integrable. Hence the frequency-domain eigenfunction description of a linear system does not depend on the existence of the ordinary Fourier transform of the input signal.

The same eigenfunction viewpoint extends to constant-coefficient partial differential equations, e.g., Maxwell's equations, where trial functions of the form
\[
e^{-i(k_x x + k_y y + k_z z - \omega t)}
\]
reduce differential operators to algebraic multipliers. The Green-function example in Section~4.4 gives a related application.

With null initial conditions, the complex exponentials $e^{i\omega t}$ therefore have two related roles in the present formulation: they are eigenfunctions of constant-coefficient linear systems, and they satisfy the continuous-frequency mean-orthogonality relation of Section~3.3. These properties underlie their central role in Fourier-domain system analysis.

\subsection{Application to the Green function of the two-dimensional Laplacian}

We consider the Green-function equation of the two-dimensional Laplacian
\[
-\Delta u(x,y) = \delta(x)\delta(y) \qquad (x,y) \in \mathbb{R}^2,
\]
where
\[
\Delta = \frac{\partial^2}{\partial x^2} + \frac{\partial^2}{\partial y^2}
\]
denotes the Laplacian operator and $\delta(x)\delta(y)$ is the two-dimensional Dirac distribution. This equation defines the fundamental solution of the two-dimensional Laplacian and leads to the logarithmic Green kernel. Although delta functions such as $\delta(x)$ are used here for brevity, consistent with the t.g.l.\ formulation adopted throughout this paper they are to be understood as delta-convergent sequences of ordinary functions, together with the generalized limit under which integration against them is defined.

We introduce the two-dimensional Fourier transform
\[
U(\xi,\eta) = \int_{-\infty}^{\infty}\int_{-\infty}^{\infty} u(x,y)\, e^{-i(\xi x + \eta y)}\,dx\,dy
\]
with inverse transform
\[
u(x,y) = \frac{1}{(2\pi)^2}\int_{-\infty}^{\infty}\int_{-\infty}^{\infty} U(\xi,\eta)\, e^{i(\xi x + \eta y)}\,d\xi\,d\eta.
\]
Distributionally, $\mathcal{F}\{\delta(x)\delta(y)\}=1$, so
\[
(\xi^2+\eta^2)\, U(\xi,\eta) = 1, \qquad U(\xi,\eta) = \frac{1}{\xi^2+\eta^2}.
\]
The inverse integral is singular at the frequency origin and therefore requires a generalized interpretation.

In the truncate-and-generalized-limit formulation we introduce frequency truncation
\[
0 < \varepsilon < \rho < \Omega, \qquad \rho = \sqrt{\xi^2+\eta^2}
\]
and consider
\[
u_{\varepsilon,\Omega}(x,y) = \frac{1}{(2\pi)^2}\int\!\!\int_{\varepsilon<\rho<\Omega} \frac{e^{i(\xi x+\eta y)}}{\xi^2+\eta^2}\,d\xi\,d\eta.
\]
The Green function is obtained from the localization limits $\varepsilon\to0$, $\Omega\to\infty$.

Introduce frequency-domain polar coordinates
\[
\xi = \rho\cos\theta, \quad \eta = \rho\sin\theta, \quad d\xi\,d\eta = \rho\,d\rho\,d\theta.
\]
Let
\[
r = \sqrt{x^2+y^2}, \quad x = r\cos\varphi, \quad y = r\sin\varphi, \quad dx\,dy = r\,dr\,d\varphi.
\]
Then
\[
\xi x + \eta y = r\rho\cos(\theta-\varphi).
\]
Thus
\[
u_{\varepsilon,\Omega}(r) = \frac{1}{(2\pi)^2}\int_\varepsilon^\Omega\int_0^{2\pi} \frac{e^{ir\rho\cos(\theta-\varphi)}}{\rho^2}\,\rho\,d\theta\,d\rho = \frac{1}{(2\pi)^2}\int_\varepsilon^\Omega\int_0^{2\pi} \frac{e^{ir\rho\cos\theta}}{\rho^2}\,\rho\,d\theta\,d\rho.
\]
Using the identity
\[
\int_0^{2\pi} e^{ir\rho\cos\theta}\,d\theta = 2\pi\, J_0(\rho r),
\]
where $J_0$ is the Bessel function, we obtain
\[
u_{\varepsilon,\Omega}(r) = \frac{1}{2\pi}\int_\varepsilon^\Omega \frac{J_0(\rho r)}{\rho}\,d\rho.
\]
The untruncated integral of $J_0(\rho r)/\rho$ is not used as an ordinary improper integral. For every $0<\varepsilon<\Omega<\infty$, however, $u_{\varepsilon,\Omega}(r)$ is an ordinary integral over a compact interval. Differentiation with respect to $r$ may therefore be performed at this finite stage, after which the limits $\varepsilon\to0^+$ and $\Omega\to\infty$ are taken. To determine the physically relevant $r$-dependence of the Green kernel, differentiate the finite-stage expression with respect to $r$:
\[
\frac{\partial u_{\varepsilon,\Omega}}{\partial r} = \frac{1}{2\pi}\int_\varepsilon^\Omega \frac{\partial J_0(\rho r)}{\partial r}\,\frac{d\rho}{\rho}.
\]
Using $\partial J_0(\rho r)/\partial r = -\rho J_1(\rho r)$, the factor of $\rho$ cancels and
\[
\frac{\partial u_{\varepsilon,\Omega}}{\partial r} = -\frac{1}{2\pi}\int_\varepsilon^\Omega J_1(\rho r)\,d\rho,
\]
an ordinary, manifestly convergent integral at every finite stage. Since $\partial J_0(\rho r)/\partial \rho = -rJ_1(\rho r)$, this is exactly $\frac{1}{2\pi r}$ times the finite-stage change of $J_0(\rho r)$, so that
\[
\frac{\partial u_{\varepsilon,\Omega}}{\partial r} = \frac{1}{2\pi r}\big[J_0(\Omega r)-J_0(\varepsilon r)\big].
\]
For fixed $r>0$, $J_0(\Omega r)\to0$ as $\Omega\to\infty$, while $J_0(\varepsilon r)\to1$ as $\varepsilon\to0^+$. Therefore the generalized limiting derivative is
\[
\frac{du}{dr} = -\frac{1}{2\pi r}.
\]
Integrating with respect to $r$ gives
\[
u(r) = -\frac{1}{2\pi}\ln r + C.
\]
The t.g.l.\ construction determines the $r$-dependent part of the Green kernel; the additive constant $C$ is not determined by the differential equation and is not claimed to vanish. Equivalently, introducing an arbitrary reference length $r_0>0$ gives $u(r) = -\frac{1}{2\pi}\ln(r/r_0)$.

As a check of the sign and normalization, integration of $-\Delta u=\delta_0$ over a disk gives $-\int_{\partial B_r}(\partial u/\partial r)\,ds = 1$; since $\partial u/\partial r = -1/(2\pi r)$ is constant on $\partial B_r$,
\[
-\int_0^{2\pi}\left(-\frac{1}{2\pi r}\right) r\,d\theta = 1,
\]
confirming the coefficient and sign obtained above.

The distributional Fourier calculation gives the same logarithmic Green kernel, up to the arbitrary additive constant. The t.g.l.\ calculation provides a constructive finite-stage route: the singular radial inverse transform is retained as $u_{\varepsilon,\Omega}(r)$, differentiation is performed before removing the truncations, and the resulting generalized limit gives
\[
\frac{du}{dr}=-\frac{1}{2\pi r}, \qquad u(r)=-\frac{1}{2\pi}\ln r+C.
\]
The Green function itself is classical; the purpose of this example is to show how it is recovered through ordinary finite-stage integrals followed by the prescribed limiting procedure.

\section{Examples}

Section~4 illustrated structural applications of the t.g.l.\ formulation. We now consider representative individual functions outside the classical $L^1(\mathbb{R})$ setting. The examples are grouped according to the source of non-integrability: non-decay at infinity, local singularities, combined local and asymptotic obstructions, and oscillatory nonabsolute convergence.

Where useful, established distributional, convergence-factor, or other generalized interpretations are recalled for comparison. The purpose is not to rederive the general theory in each example, but to show how the same finite-stage truncation and generalized-limit construction applies across these different classes.

\subsection{Locally integrable functions with non-decay at infinity}

We first consider functions that are integrable on every finite interval but fail to belong to $L^1(\mathbb{R})$ because of their behavior at infinity.

\medskip
\noindent \textbf{Example 1. A constant function $f(t)=1$ for $-\infty<t<\infty$.}

Here our t.g.l.\ approach, conventional elementary methods and the distribution theory are compared to clarify the distinction between them.

This function $f(t)=1$ is not absolutely integrable. For its truncation $f_T(t)$, we have:
\[
F_T(\omega) = \int_{-\infty}^{\infty} f_T(t)\, e^{-i\omega t}\,dt = \int_{-T}^{T} e^{-i\omega t}\,dt = \frac{2\sin\omega T}{\omega} = 2\pi\delta_T(\omega),
\]
where $\delta_T(\omega) = \sin\omega T/\pi\omega$.

The inverse Fourier transform is:
\[
\overleftarrow{f}_T(t) = \lim_{\Omega\to\infty} \frac{1}{2\pi}\int_{-\Omega}^{\Omega} \frac{2\sin\omega T}{\omega}\, e^{i\omega t}\,d\omega = \frac{1}{\pi}\int_0^\infty \frac{2\sin\omega T}{\omega}\cos\omega t\,d\omega =
\begin{cases}
1 & |t|<T \\
1/2 & |t|=T \\
0 & |t|>T,
\end{cases}
\]
where we used
\begin{equation}
\int_0^\infty \frac{\sin ax\, \cos bx}{x}\,dx =
\begin{cases}
\pi/2 & |b|<a \\
\pi/4 & |b|=a \\
0 & |b|>a,
\end{cases}
\qquad a>0,\ b\in\mathbb{R}.
\label{eq:7}
\end{equation}
Thus, in the pointwise sense, $\lim_{T\to\infty} \overleftarrow{f}_T(t) = \lim_{T\to\infty} f_T(t) = f(t)=1$ $(-\infty<t<\infty)$. Note that $F_T(\omega)$ is an ordinary function, and the integrals are well defined for finite $T$.

For $f(t)=1$, we have $F(\omega) = \lim_{T\to\infty} F_T(\omega)$. The spectral density function $F(\omega)$ never converges in the pointwise sense, and is singular at $\omega=0$:
\[
F(0) = \lim_{T\to\infty}\lim_{\omega\to0} \frac{2\sin\omega T}{\omega} = \lim_{T\to\infty} 2T = \infty.
\]
On the other hand, the spectrum function $C(\omega)$ [Appendix A] converges pointwise as a normalized spectral coefficient density as:
\[
C(\omega) = \lim_{T\to\infty} C_T(\omega) =
\begin{cases}
1 & \omega=0 \\
0 & \omega \ne 0.
\end{cases}
\]
Physically, $C(\omega)$ carries the same unit as $f(t)$ itself (e.g., volts), while the spectral density $F(\omega)$ carries units of $f(t)$ per unit frequency (e.g., volts per hertz); see Appendix A. It is worth noting that the singularity of $F(\omega)$ disappears if it is integrated even over a small interval. More generally, when $f$ is bounded, the normalized quantity $C_T(\omega) = F_T(\omega)/2T$ remains bounded and typically converges as $T\to\infty$, even in cases such as this one where the unnormalized $F(\omega)$ does not. Boundedness of the underlying function is therefore what allows a normalized spectral quantity to acquire a pointwise limit where the spectral density function itself cannot.

The Fourier transform can be written as the generalized limit (Section~3.8):
\[
F(\omega) = \lim_{T\to\infty} \frac{2\sin\omega T}{\omega} = \lim_{T\to\infty} 2\pi\delta_T(\omega) = 2\pi\delta(\omega).
\]
For a function $f(t) = \exp(i\omega_0 t)$, we obtain:
\[
F_T(\omega) = \int_{-T}^{T} e^{i\omega_0 t}\, e^{-i\omega t}\,dt = \frac{2\sin[(\omega-\omega_0)T]}{\omega-\omega_0} = 2\pi\delta_T(\omega-\omega_0).
\]
Based on the above-mentioned result and definition of the Fourier transform, we determine that $F(\omega) = 2\pi\delta(\omega-\omega_0)$ in the sense of the generalized (or distributional) limit.

\medskip
\noindent\textit{Distributional-pairing derivation.} Using distribution theory, we have:
\[
\langle f(t),G(t)\rangle = \int_{-\infty}^{\infty} G(t)\,dt = 2\pi g(0) = \langle 2\pi\delta(\omega),g(\omega)\rangle = \langle F(\omega),g(\omega)\rangle,
\]
where $\delta(\omega)$ is the delta function. Thus we obtain $F(\omega) = 2\pi\delta(\omega)$. In the derivation above, we used the symmetry of the Fourier transform pair as follows:
\begin{equation}
2\pi g(-\omega) = \int_{-\infty}^{\infty} G(t)\, e^{-i\omega t}\,dt.
\label{eq:8}
\end{equation}
The inverse Fourier transform $\overleftarrow{f}(t)=1$ is obtained as follows:
\[
\langle \overleftarrow{f}(t),G(t)\rangle = \langle F(\omega),\overleftarrow{g}(\omega)\rangle = \langle 2\pi\delta(\omega),\overleftarrow{g}(\omega)\rangle
= 2\pi\overleftarrow{g}(0) = \int_{-\infty}^{\infty} G(t)\,dt = \langle 1,G(t)\rangle.
\]
In the argument above the delta function appears directly; no delta-convergent sequence of functions is used. Furthermore, the Fourier transform is changed to an operator from the ordinary definition. The results obtained using this approach as $1\leftrightarrow2\pi\delta(\omega)$ and $\delta(t)\leftrightarrow1$ are based on the aforementioned facts.

\medskip
\noindent\textit{Function-sequence derivation.} Another approach is based on a function sequence: the function $f(t)=1$ never requires the smudge function defined in [9]. A function sequence $f_n(t) = \exp(-t^2/n^2)$ $(n=1,2,3,\dots)$ converges to $1$ for $n\to\infty$. The Fourier transform of $f_n(t)$ is $F_n(\omega) = \sqrt{\pi}\, n\, \exp[-(n\omega)^2/4]$, which never converges in the classical sense for $n\to\infty$. However it converges to $2\pi\delta(\omega)$ in the sense of the generalized limit or distribution. Thus we denote $f(t)=1\leftrightarrow2\pi\delta(\omega)$.

\medskip
\noindent\textit{Damping-factor derivation.} Another method [4] uses a damping factor, a concept similar to the Laplace transform as
\[
F_\sigma(\omega) = \int_{-\infty}^{\infty} f_\sigma(t)\, e^{-i\omega t}\,dt.
\]
Here $f_\sigma(t) = f(t)\exp(-\sigma|t|)$ $(\sigma>0)$. Notably $\lim_{\sigma\to0} f_\sigma(t) = f(t)$. Substituting $f(t)=1$, we obtain:
\[
F_\sigma(\omega) = \int_{-\infty}^{\infty} e^{-\sigma|t|}\, e^{-i\omega t}\,dt = \frac{2\sigma}{\sigma^2+\omega^2}.
\]
The inverse transform is obtained as:
\[
\overleftarrow{f}_\sigma(t) = \frac{1}{2\pi}\int_{-\infty}^{\infty} \frac{2\sigma}{\sigma^2+\omega^2}\, e^{i\omega t}\,d\omega = \frac{1}{\pi}\int_0^\infty \frac{2\sigma\cos\omega t}{\sigma^2+\omega^2}\,d\omega = e^{-\sigma|t|} \to 1\ \ (\sigma\to0),
\]
where we used
\[
\int_0^\infty \frac{\cos ax}{b^2+x^2}\,dx = \frac{\pi}{2b}\, e^{-ab} \qquad (b>0).
\]
If we consider the limit $\sigma\to0$ before the inverse Fourier transform, the inverse Fourier transform cannot be obtained, since the integration becomes impossible due to the fact that $\lim_{\sigma\to0} F_\sigma(\omega)$ never converges. $F_\sigma(\omega)/2\pi$ is a delta-convergent sequence of functions. We prove this: consider
\[
\lim_{\sigma\to0} \int_{-\infty}^{\infty} F_\sigma(\omega)\big[g(\omega)-g(0)\big]\,d\omega = \lim_{\sigma\to0} \int_{-\infty}^{\infty} \frac{2\sigma\omega}{\sigma^2+\omega^2}\cdot\frac{g(\omega)-g(0)}{\omega}\,d\omega.
\]
Since $[g(\omega)-g(0)]/\omega$ is bounded ($g(\omega)\in\mathcal{S}(\mathbb{R})$), the integral tends to zero for $\sigma\to0$. Then
\[
\lim_{\sigma\to0} \int_{-\infty}^{\infty} \frac{2\sigma}{\sigma^2+\omega^2}\, g(\omega)\,d\omega = \int_{-\infty}^{\infty} \frac{2\sigma}{\sigma^2+\omega^2}\, g(0)\,d\omega = 2\pi g(0),
\]
where $\int_{-\infty}^{\infty} \dfrac{\sigma}{\sigma^2+x^2}\,dx = \pi$ is used.

\medskip
\noindent\textit{Transform-pair-symmetry derivation.} The final method often used to handle $f(t)=1$ involves using the symmetry of the Fourier transform pair [13], as depicted in Eq.~(\ref{eq:8}). For $f(t) = \delta(t)$, we obtain:
\[
F(\omega) = \int_{-\infty}^{\infty} \delta(t)\, e^{-i\omega t}\,dt = 1.
\]
Therefore we obtain:
\[
\int_{-\infty}^{\infty} F(t)\, e^{-i\omega t}\,dt = 2\pi f(-\omega) = 2\pi\delta(-\omega) = 2\pi\delta(\omega).
\]
Equation~(\ref{eq:8}) is derived from the fact that the Fourier transform pair holds, i.e., $f(t)$ is assumed to be absolutely integrable. The application of Eq.~(\ref{eq:8}) is inconsistent with $f(t)=1$. However, the final method gives the same results, i.e., $2\pi\delta(\omega)$ as other methods, although it is based on the ordinary definition of the Fourier transform.

Our limiting approach can avoid the required condition for $f \in L^1(\mathbb{R})$. In this aspect, our approach belongs to the same family as those with convergence factor methods. Our approach introduces a truncation function $p_T(t)=\mathbf{1}_{[-T,T]}(t)$. The generalized function sequence method (without the smudge function) adopts $f_n(t) = \exp(-t^2/n^2)$. The Laplace transform concept method assumes that $f_\sigma(t) = f(t)\exp(-\sigma|t|)$.

However, a distinction exists between our method and the convergent factor methods: the convergence factor method cannot be applied directly to non-locally integrable functions such as $f(t)=1/t$ (Example~7) and $f(t)=\lim_{\Omega\to\infty}\sin(\Omega t)/t$ (Example~10), whereas the t.g.l.\ method can.

The ordinary definition of the inverse Fourier transform (Eq.~(2)) should not be applied for non-absolutely integrable functions; instead of this, our definition (Eq.~(4)) is useful, as long as the Fourier transform pair are defined by the transition from the Fourier series. This is the reason why artificial parameters $n$ for the generalized function sequence method, and $\sigma$ for the Laplace transform concept method are introduced.

Now we consider asymmetric truncation and interchange of limit orders. We obtain the forward transform of truncated $f_T$:
\[
F_T(\omega) = \int_{-T_1}^{T_2} e^{-i\omega t}\,dt = \frac{i}{\omega}\big(e^{-i\omega T_2} - e^{i\omega T_1}\big).
\]
Limit $T_1,T_2\to\infty$, $F_T(\omega)$ does not converge pointwise.

We take a second-stage generalized limit as
\[
\lim_{T\to\infty}\int_{-\infty}^{\infty} F_T(\omega)\,G^{*}(\omega)\,d\omega
= \lim_{T_2\to\infty}\int_{-\infty}^{\infty} \frac{i}{\omega}\,e^{-i\omega T_2}\,G^{*}(\omega)\,d\omega
- \lim_{T_1\to\infty}\int_{-\infty}^{\infty} \frac{i}{\omega}\,e^{i\omega T_1}\,G^{*}(\omega)\,d\omega
\]
\[
= \lim_{T_2\to\infty}\left[i\int_{-\infty}^{\infty} e^{-i\omega T_2}\,H^{*}(\omega)\,d\omega + iG^{*}(0)\int_{-\infty}^{\infty} \frac{e^{-i\omega T_2}}{\omega}\,d\omega\right]
\]
\[
\qquad - \lim_{T_1\to\infty}\left[i\int_{-\infty}^{\infty} e^{i\omega T_1}\,H^{*}(\omega)\,d\omega + iG^{*}(0)\int_{-\infty}^{\infty} \frac{e^{i\omega T_1}}{\omega}\,d\omega\right] = 2\pi G^{*}(0),
\]
where $H(\omega) = (G(\omega)-G(0))/\omega$ is assumed absolutely integrable, and the Riemann--Lebesgue lemma together with $\lim_{T\to\infty}\displaystyle\int_{-\infty}^{\infty} \frac{e^{i\omega T}}{i\omega}\,d\omega = \pi$ are used.

Therefore $\lim_{T\to\infty} F_T(\omega) = 2\pi\delta(\omega)$ is obtained in the sense of the generalized limit or distributions.

The truncated inverse transform of $F_T(\omega)$ is:
\[
f_{T,\Omega}(t) = \frac{1}{2\pi}\int_{-\Omega}^{\Omega} \left[\frac{\sin\omega(T_2-t)}{\omega} + \frac{\sin\omega(T_1+t)}{\omega}\right] d\omega.
\]
The t.g.l.\ inversion becomes:
\[
\lim_{T\to\infty}\lim_{\Omega\to\infty} f_{T,\Omega}(t) = \lim_{T\to\infty} f(t)\,\mathbf{1}_{[-T_1,T_2]}(t) = f(t).
\]
As established in \Cref{thm:inversion}'s Remarks and \Cref{thm:noninterchange}, the forward transform $F_T(\omega)=2\sin(\omega T)/\omega$ has no pointwise limit as $T\to\infty$ for $\omega\neq0$, and diverges at $\omega=0$. Nevertheless, for this particular example, the completed finite-$\Omega$ integral $f_{T,\Omega}(t)$ itself does possess a limit as $T\to\infty$, at every fixed finite $\Omega$, and this limit agrees with $f(t)$. This numerical agreement --- obtained without ever forming a pointwise limit of $F_T(\omega)$ --- is specific to the present example and does not demonstrate general interchangeability of the two limits: whenever the forward transform fails to converge pointwise, as it does here, the order of the two limits cannot be interchanged in general (\Cref{thm:noninterchange}).

\medskip
\noindent \textbf{Example 2. Sign function $\mathrm{sgn}(t)=t/|t|$.}

Truncated forward transform is:
\[
F_T(\omega) = \int_{-T}^{T} \mathrm{sgn}(t)\,e^{-i\omega t}\,dt = \frac{2}{i\omega} - \frac{2\cos\omega T}{i\omega}.
\]
The truncated inverse transform is:
\[
f_{T,\Omega}(t) = \frac{1}{2\pi}\int_{-\Omega}^{\Omega}\left[\frac{2}{i\omega}-\frac{2\cos\omega T}{i\omega}\right]e^{i\omega t}\,d\omega = \frac{1}{\pi}\int_{-\Omega}^{\Omega}\left[\frac{\sin\omega t}{\omega} - \frac{\cos(\omega T)\sin\omega t}{\omega}\right]d\omega.
\]
Then
\[
\lim_{\Omega\to\infty} f_{T,\Omega}(t) = f_T(t) = \frac{1}{\pi}\int_{-\infty}^{\infty}\left[\frac{\sin\omega t}{\omega} - \frac{\cos(\omega T)\sin\omega t}{\omega}\right]d\omega = \mathrm{sgn}(t) +
\begin{cases}
0, & |t|<T,\\
-1, & t>T,\\
1, & t<-T.
\end{cases}
\]
We then obtain the pointwise convergent t.g.l.\ inverse transform
\[
\lim_{T\to\infty} f_T(t) = \lim_{T\to\infty}\lim_{\Omega\to\infty} f_{T,\Omega}(t) = \mathrm{sgn}(t).
\]
For each fixed finite $\Omega$, the oscillatory second term tends to zero as $T\to\infty$. Hence $\lim_{T\to\infty} f_{T,\Omega}(t) = \frac{1}{\pi}\int_{-\Omega}^{\Omega}\frac{\sin\omega t}{\omega}\,d\omega$, and consequently $\lim_{\Omega\to\infty}\lim_{T\to\infty} f_{T,\Omega}(t) = \mathrm{sgn}(t)$. Since the prescribed t.g.l.\ order also gives $\lim_{T\to\infty}\lim_{\Omega\to\infty} f_{T,\Omega}(t) = \mathrm{sgn}(t)$, the reversed limiting order is verified for this example by the reconstruction criterion: both limiting procedures recover the same original function, with midpoint value $0$ at $t=0$. This agreement is specific to this example and does not imply general interchangeability of the t.g.l.\ truncation limits.

Let
\[
F_T(\omega) = \frac{2}{i\omega} + F_{2,T}(\omega), \qquad F_{2,T}(\omega) = -\frac{2\cos\omega T}{i\omega}.
\]
For the second-stage generalized limit, the Riemann--Lebesgue lemma gives
\[
\lim_{T\to\infty} \int_{-\infty}^{\infty} F_{2,T}(\omega)\,G^{*}(\omega)\,d\omega = -\lim_{T\to\infty}\int_{-\infty}^{\infty} \frac{2\cos\omega T}{i\omega}\,G^{*}(\omega)\,d\omega = 0,
\]
where the pairing function $G^{*}(\omega)/\omega$ is assumed absolutely integrable.

\medskip
\noindent \textbf{Example 3. Power function $f(t) = t^n$ $(n=0,1,2,\dots)$.}

These functions have polynomial growth and are not in $L^1(\mathbb{R})$. The case of $n=0$ has already been discussed. We first consider the case for $n=1$.

We have for $\omega \ne 0$:
\[
F_T(\omega) = \int_{-T}^{T} t\, e^{-i\omega t}\,dt = \left[\frac{t\, e^{-i\omega t}}{-i\omega}\right]_{-T}^{T} - \int_{-T}^{T} \frac{e^{-i\omega t}}{-i\omega}\,dt = i\,\frac{2T\cos\omega T}{\omega} - i\,\frac{2\sin\omega T}{\omega^2},
\]
and $F_T(0)=0$.

Using $\delta_T(\omega) = \sin(\omega T)/\pi\omega$, we obtain:
\[
\delta_T'(\omega) = \frac{T\cos\omega T}{\pi\omega} - \frac{\sin\omega T}{\pi\omega^2}.
\]
Thus, we obtain, in the generalized/distributional sense rather than as an ordinary pointwise limit:
\[
F(\omega) = \lim_{T\to\infty} F_T(\omega) = i\,2\pi \lim_{T\to\infty} \delta_T'(\omega) = i\,2\pi\,\delta'(\omega).
\]
The truncated inverse Fourier transform is obtained as:
\[
\overleftarrow{f}_{T,\Omega}(t) = \frac{1}{2\pi}\int_{-\Omega}^{\Omega} F_T(\omega)\, e^{i\omega t}\,d\omega = i\int_{-\Omega}^{\Omega} \delta_T'(\omega)\, e^{i\omega t}\,d\omega
\]
\[
= i\,\Big[\delta_T(\omega)\, e^{i\omega t}\Big]_{-\Omega}^{\Omega} + t\int_{-\Omega}^{\Omega} \delta_T(\omega)\, e^{i\omega t}\,d\omega
= -2\,\delta_T(\Omega)\sin\Omega t + t\int_{-\Omega}^{\Omega} \delta_T(\omega)\, e^{i\omega t}\,d\omega.
\]
For $\Omega\to\infty$, the first term of the last right-hand side vanishes. Using Eq.~(\ref{eq:7}), we obtain:
\[
\overleftarrow{f}_T(t) = t \lim_{\Omega\to\infty} \int_{-\Omega}^{\Omega} \delta_T(\omega)\, e^{i\omega t}\,d\omega = t \lim_{\Omega\to\infty} \int_{-\Omega}^{\Omega} \frac{\sin T\omega\, \cos\omega t}{\pi\omega}\,d\omega =
\begin{cases}
t & |t|<T \\
0 & |t|>T.
\end{cases}
\]
Then, we obtain $f(t) = \lim_{T\to\infty} f_T(t) = \lim_{T\to\infty} \overleftarrow{f}_T(t) = t$ for $t\in(-\infty,\infty)$.

The order of the two limits cannot be interchanged in this example, as demonstrated explicitly in the proof of \Cref{thm:noninterchange}.

For $n \ge 2$, let $f_T(t)=t^n\,\mathbf{1}_{[-T,T]}(t)$. For every fixed finite $T$, $f_T$ belongs to $L^1(\mathbb{R})$, so $F_T$ is an ordinary Fourier integral. Classical inversion gives $\lim_{\Omega\to\infty}\frac{1}{2\pi}\int_{-\Omega}^{\Omega}F_T(\omega)e^{i\omega t}\,d\omega=f_T(t)$ at every continuity point, and then $\lim_{T\to\infty}f_T(t)=t^n$. Thus t.g.l.\ inversion is obtained entirely through finite-stage ordinary Fourier integrals.

As a generalized spectral consistency check, the t.g.l.\ integral with integration by parts [\Cref{thm:s2one}] may also be used. Using the truncated function $f_T(t)$ of $f(t)=1$, $f_T(t) \leftrightarrow F_T(\omega) = 2\sin(\omega T)/\omega = 2\pi\delta_T(\omega)$ and $\lim_{T\to\infty} F_T(\omega) = 2\pi\delta(\omega)$ in a distributional sense we obtain:
\[
t^n \leftrightarrow (-i)^{-n}\, 2\pi \lim_{T\to\infty} \delta_T^{(n)}(\omega) = (-i)^{-n}\, 2\pi\, \delta^{(n)}(\omega),
\]
the last equality understood in the generalized-limit/distributional sense; this is the familiar distributional relation reproduced by the construction, not required as the starting point of the reconstruction above. As a further consistency check, the distributional inverse transform is given as:
\[
\langle \overleftarrow{f}(t),G(t)\rangle = \langle F(\omega),\overleftarrow{g}(\omega)\rangle = \big\langle (-i)^{-n}\, 2\pi\, \delta^{(n)}(\omega),\overleftarrow{g}(\omega)\big\rangle
\]
\[
= (-i)^{-n} \int_{-\infty}^{\infty} \delta^{(n)}(\omega) \int_{-\infty}^{\infty} G(t)\, e^{i\omega t}\,dt\,d\omega
= (-i)^{-n} \int_{-\infty}^{\infty} G(t) \int_{-\infty}^{\infty} \delta^{(n)}(\omega)\, e^{i\omega t}\,d\omega\,dt
\]
\[
= \int_{-\infty}^{\infty} t^n\, G(t)\,dt = \langle t^n,G(t)\rangle.
\]
Thus we have $\overleftarrow{f}_T(t) = t^n$.

Since the truncated functions are in $L^1(\mathbb{R})$, we obtain the t.g.l.\ inverse transform readily:
\[
\overleftarrow{f}_T(t) = \lim_{T\to\infty} t^n\, \mathbf{1}_{[-T,T]}(t) = t^n.
\]

\medskip
\noindent \textbf{Example 4. Periodic function.}

A periodic function $f(t)$ is not absolutely integrable and is expressed as
\[
f(t) = \sum_{n=-\infty}^{\infty} f_{T_0}(t-nT_0),
\]
where $T_0$ is the duration, and $f_{T_0}(t) = f(t)$ for $-T_0/2 \le t \le T_0/2$ and $f_{T_0}(t) = 0$ otherwise. We use the Fourier series as
\[
f(t) = \sum_{n=-\infty}^{\infty} C_n\, e^{in\omega_0 t} \qquad (\omega_0 = 2\pi/T_0),
\]
where
\[
C_n = \frac{1}{T_0}\int_{-T_0/2}^{T_0/2} f(t)\, e^{-in\omega_0 t}\,dt.
\]
This example illustrates explicitly how the Fourier transform of periodic functions arises through the same continuous-frequency orthogonal-expansion viewpoint of Section~3.3, applied here via truncation-based limiting procedures.

We truncate $f(t)$ in an interval $[-T,T]$ and denote it as $f_T(t)$, and its Fourier transform as $F_T(\omega)$. Applying $e^{in\omega_0 t}\leftrightarrow 2\pi\lim_{T\to\infty}\delta_T(\omega-n\omega_0)$ termwise to the Fourier-series representation gives the generalized spectral representation
\[
F(\omega) = 2\pi\sum_{n=-\infty}^{\infty} C_n\, \delta(\omega-n\omega_0).
\]
Here the expression is understood in the generalized-limit sense associated with the periodic Fourier-series representation; it is not intended as an unrestricted interchange of the truncation limit with an arbitrary infinite series. Using $F_T(\omega)$ for the periodic function, where $\delta_T(\omega) = \sin(T\omega)/\pi\omega$, we obtain:
\[
C(\omega) = \lim_{T\to\infty} C_T(\omega) = \lim_{T\to\infty} F_T(\omega)/2T = \lim_{T\to\infty} \sum_{n=-\infty}^{\infty} 2\pi C_n\, \delta_T(\omega-n\omega_0)\, /\, 2T
= \sum_{n=-\infty}^{\infty} C_n\, \delta_{\omega,n\omega_0},
\]
where
\[
\delta_{\omega,n\omega_0} =
\begin{cases}
1 & \omega = n\omega_0 \\
0 & \omega \ne n\omega_0
\end{cases}
\qquad (n=0,\pm1,\pm2,\dots).
\]
Thus, the spectrum $C(\omega)$ [see Appendix A] converges even when the spectral density $F(\omega)$ does not.

It is worth noting that $C(\omega)=0$ for absolutely integrable functions under the normalization adopted here, whereas for unbounded functions the existence of $C(\omega)$ is not guaranteed.

\medskip
\noindent \textbf{Example 5. Exponentially increasing function $f_\alpha(t)=e^{\alpha t}$, $\alpha>0$.}

Although $e^{\alpha t}\notin L^1(\mathbb{R})$ and does not define a tempered distribution, it is Riemann integrable on every finite interval and is therefore directly admissible to the finite-truncation procedure of the present formulation.

For finite $T_1,T_2>0$, define
\[
f_{T,\alpha}(t)=
\begin{cases}
e^{\alpha t},&-T_1\leq t\leq T_2,\\
0,&\text{otherwise}.
\end{cases}
\]
Its ordinary Fourier transform is
\[
F_{T,\alpha}(\omega) =\int_{-T_1}^{T_2}e^{\alpha t} e^{-i\omega t}\,dt =\frac{e^{(\alpha-i\omega)T_2}-e^{-(\alpha-i\omega)T_1}}{\alpha-i\omega}.
\]
Although $F_{T,\alpha}(\omega)$ has no ordinary limit as $T_2\to\infty$, the first generalized-limit operation is performed through the inverse Fourier integral before the truncation limits are removed:
\[
\lim_{T_1,T_2\to\infty} \lim_{\Omega\to\infty} \frac{1}{2\pi}\int_{-\Omega}^{\Omega} F_{T,\alpha}(\omega)e^{i\omega t}\,d\omega =e^{\alpha t}.
\]

The second-stage generalized limit may be illustrated by introducing
\[
g_\sigma(t)=e^{-\sigma |t|},\qquad \sigma>0,
\]
whose Fourier transform is
\[
G_\sigma(\omega) =\int_{-\infty}^{\infty} e^{-\sigma|t|}e^{-i\omega t}\,dt =\frac{2\sigma}{\sigma^2+\omega^2}.
\]
Since $G_\sigma(\omega)$ is real and even, $G_\sigma^*(\omega)=G_\sigma(\omega)$.

For every finite $T_1,T_2$, Parseval's relation gives
\[
\frac{1}{2\pi} \int_{-\infty}^{\infty} F_{T,\alpha}(\omega)G_\sigma^*(\omega)\,d\omega = \int_{-T_1}^{T_2}e^{\alpha t} e^{-\sigma|t|}\,dt.
\]
The right-hand side can be evaluated explicitly as
\[
\int_{-T_1}^{0}e^{(\alpha+\sigma)t}\,dt + \int_0^{T_2}e^{(\alpha-\sigma)t}\,dt.
\]
Hence
\[
\frac{1}{2\pi} \int_{-\infty}^{\infty} F_{T,\alpha}(\omega)G_\sigma^*(\omega)\,d\omega = \frac{1-e^{-(\alpha+\sigma)T_1}}{\alpha+\sigma} + \frac{e^{(\alpha-\sigma)T_2}-1}{\alpha-\sigma}, \qquad \sigma\ne\alpha.
\]

The behavior of the second generalized limit is therefore determined explicitly. For $\sigma>\alpha$, both terms possess finite limits as $T_1,T_2\to\infty$, and
\[
\lim_{T_1,T_2\to\infty} \frac{1}{2\pi} \int_{-\infty}^{\infty} F_{T,\alpha}(\omega)G_\sigma^*(\omega)\,d\omega = \frac{1}{\alpha+\sigma} +\frac{1}{\sigma-\alpha} = \frac{2\sigma}{\sigma^2-\alpha^2}.
\]
This is just the ordinary integral $\int_{-\infty}^{\infty}e^{\alpha t} e^{-\sigma|t|}\,dt$. At the critical value $\sigma=\alpha$,
\[
\int_{-T_1}^{T_2}e^{\alpha t} e^{-\alpha|t|}\,dt = \frac{1-e^{-2\alpha T_1}}{2\alpha}+T_2,
\]
which diverges as $T_2\to\infty$; for $0<\sigma<\alpha$, the divergence is exponential. Thus the conventional damping-factor method yields an ordinary convergent pairing only for $\sigma>\alpha$.

The point of the second-stage t.g.l.\ operation is that the finite-$T$ quantity $\frac{1}{2\pi}\int F_{T,\alpha}(\omega)G_\sigma^*(\omega)\,d\omega$ remains perfectly well defined before the generalized truncation limit is taken. It therefore makes explicit where the ordinary damping procedure succeeds and where its conventional convergence requirement fails. In particular, decreasing $\sigma$ from the ordinary convergence region encounters the boundary $\sigma=\alpha$ before $\sigma=0$ can be reached. Consequently, the Fourier transform of $e^{\alpha t}$ cannot be obtained by simply calculating the ordinary damped transform for $\sigma>\alpha$ and then taking $\sigma\to0^+$.

The Taylor expansion of $e^{\alpha t}$ can also be used to relate this example formally to the theory of analytic functionals. Applying the pairing $t^n \leftrightarrow 2\pi(-i)^{-n}\delta^{(n)}(\omega)$ termwise to $e^{\alpha t}=\sum_{n=0}^{\infty}(\alpha t)^n/n!$ gives, formally,
\[
e^{\alpha t} \leftrightarrow 2\pi\sum_{n=0}^{\infty}\frac{\alpha^n(-i)^{-n}}{n!}\delta^{(n)}(\omega),
\]
where the termwise generalized limit is assumed to commute with the infinite sum,
\[
\lim_{T\to\infty} 2\pi\sum_{n=0}^{\infty}\frac{\alpha^n(-i)^{-n}}{n!}\delta_T^{(n)}(\omega) = 2\pi\sum_{n=0}^{\infty}\frac{\alpha^n(-i)^{-n}}{n!}\delta^{(n)}(\omega).
\]
This termwise transformation is treated here as purely formal: it would require independent justification for interchanging the infinite sum with the generalized limiting operation, and no such interchange is assumed or required by the t.g.l.\ construction itself. The direct t.g.l.\ reconstruction of $e^{\alpha t}$ was already obtained above from finite-stage ordinary Fourier integrals, without recourse to this formal series. This expression lies beyond the space of tempered distributions, since the sum extends over derivatives of $\delta$ of arbitrarily high order, whereas a tempered distribution supported at a point is necessarily a finite linear combination of such derivatives.

Gel'fand and Shilov [8] treat this function within the analytic-functional framework of their theory, introducing the complex variable $s=\omega+i\sigma$ together with contour integration; using their definition $F(s)=\int f(t)e^{ist}\,dt$, their result is
\[
F(s) = 2\pi\sum_{n=0}^{\infty}\frac{\alpha^n(-i)^{-n}}{n!}\delta^{(n)}(s) = 2\pi\delta(s-i\alpha).
\]
This result can be checked against the ordinary Fourier transform of $f(t)=e^{i\omega_0 t}$, for which $F(\omega)=2\pi\delta(\omega-\omega_0)$. Taking $\alpha=i\omega_0$ and $\omega=-s$ gives
\[
F(s) = 2\pi\delta(s-i\alpha) = 2\pi\delta(s+\omega_0) = 2\pi\delta(-\omega+\omega_0) = 2\pi\delta(\omega-\omega_0) = F(\omega),
\]
in agreement with the known result. A related complex-frequency treatment, formulated in terms of generalized Fourier integrals, is given by Titchmarsh [3]. Although the analytic-functional framework of Gel'fand and Shilov and Titchmarsh's generalized Fourier-integral approach are conceptually distinct, both illustrate how extension into the complex-frequency domain can assign generalized Fourier meaning beyond the ordinary real-frequency integral.

The present t.g.l.\ formulation takes a different and more elementary route. It starts from ordinary finite-interval integrals and assigns meaning through the prescribed order of generalized limits, without requiring $e^{\alpha t}$ itself to belong to $L^1(\mathbb{R})$ or to the space of tempered distributions.

\medskip
\noindent \textbf{Example 6. A rapidly increasing function $f(t) = \exp(\alpha t^2)$ $(\alpha>0)$.}

For a rapidly increasing function [14], the t.g.l.\ integral yields the inverse Fourier transform as:
\[
\lim_{T\to\infty} \overleftarrow{f}_T(t) = \lim_{T\to\infty} f_T(t) = f(t) = e^{\alpha t^2}.
\]
The forward transform, $\lim_{T\to\infty} F_T(\omega)$, diverges. However, with a rapidly decreasing function, $g(t) = \exp(-\beta t^2)$ $(\beta>0)$, it converges absolutely ($\beta>\alpha$) in the sense of generalized limit as:
\[
\lim_{T\to\infty} \langle F_T(\omega),g(\omega)\rangle = \lim_{T\to\infty} \langle f_T(t),G(t)\rangle = \int_{-\infty}^{\infty} e^{\alpha t^2}\, e^{-\beta t^2}\,dt = \sqrt{\frac{\pi}{\beta-\alpha}}.
\]

\begin{remarks}
The function $e^{\alpha t^2}$ ($\alpha>0$) is locally integrable and defines an element of $\mathcal{D}'(\mathbb{R})$, but it is not tempered (\Cref{prop:tempered}), so the Fourier transform on $\mathcal{S}'(\mathbb{R})$ does not apply to it. In the t.g.l.\ formulation, $f_T=e^{\alpha t^2}\mathbf{1}_{[-T,T]}\in L^1(\mathbb{R})$ and $\{F_T(\omega)\}$ is an ordinary integral at every stage; pairing with $g(t)=e^{-\beta t^2}$ ($\beta>\alpha$), as shown above, gives generalized meaning when suitable auxiliary functions exist. Since $e^{\alpha t^2}$ is continuous, it is t.g.l.-admissible, and the inversion theorem (\Cref{thm:inversion}) reconstructs it pointwise, independently of the pairing.
\end{remarks}

\medskip
\noindent\emph{Remark (a formal complex-time correspondence).} The complex-frequency continuation used for $e^{\alpha t}$ in Example~5 supplies only linear exponential damping and cannot regularize this quadratic growth. Continuation in the \emph{time} variable instead can. Since $e^{t^2}$ is entire, put $t=is$ with $s$ real:
\[
e^{t^2}=e^{-s^2}, \qquad e^{-i\omega t}=e^{\omega s}, \qquad dt = i\,ds.
\]
The divergent real-axis integral is thereby associated formally with
\[
\text{c.t.}\ F(\omega) = i\int_{-\infty}^{\infty} e^{-s^2+\omega s}\,ds = i\,e^{\omega^2/4}\int_{-\infty}^{\infty} e^{-(s-\omega/2)^2}\,ds = i\sqrt{\pi}\,e^{\omega^2/4},
\]
obtained by completing the square, $-s^2+\omega s = -(s-\omega/2)^2+\omega^2/4$, and evaluating the standard Gaussian integral $\int_{-\infty}^{\infty}e^{-u^2}du=\sqrt{\pi}$. This is a formal complex-time correspondence, not the ordinary real-axis Fourier transform, since the original integral diverges; it illustrates that continuation in time, rather than in frequency, is the natural complex direction for quadratic growth. More generally, for $t=e^{i\theta}s$, $|e^{t^n}| = e^{\cos(n\theta)s^n}$, so directions with $\cos(n\theta)<0$ convert the growth of $e^{t^n}$ into decay. Both complex-frequency and complex-time methods rely on analytic continuation off the real axis; the t.g.l.\ construction, by contrast, retains real $t$ and real $\omega$ throughout, together with the ordinary Fourier kernel, and replaces continuation by finite truncation and ordered limits. The general phenomenon by which analytic continuation of a convergent Gaussian integral across a singular value of a parameter produces a well-defined, branch-dependent phase factor is treated rigorously, in the language of generalized Fresnel integrals and the metaplectic representation, in the harmonic-analysis literature; see, e.g., [24]. The present formal substitution is of the same general character, though we make no claim that the specific transform pair considered here has been treated in that literature. A closer, contour-based examination of this correspondence --- including the extent to which it can or cannot be justified by contour deformation --- is given in Section~6.5.

Similarly to the forward transform, substituting $\omega=-is$ ($s$ real) into the inverse Fourier integral applied to $\text{c.t.}\ F(\omega)=i\sqrt{\pi}\,e^{\omega^2/4}$ gives
\[
d\omega=-i\,ds, \qquad e^{i\omega t}=e^{st},
\]
so, the complex-frequency inversion is obtained:
\[
\text{c.f.}\ \overleftarrow{f}(t)=\frac{-i}{2\pi}\int_{-\infty}^{\infty} i\sqrt{\pi}\,e^{-s^2/4}\,e^{st}\,ds = \frac{\sqrt{\pi}}{2\pi}\int_{-\infty}^{\infty}e^{-(s/2-t)^2+t^2}\,ds = e^{t^2}.
\]

\begin{sloppypar}
Since $e^{\alpha t^2}$ is not tempered, Lighthill's sequences~[9], mild distributions~[10,19] and distribution theory by Riemann integrals~[17], which as far as we understand operate within $\mathcal{S}'$, give no Fourier transform for it; only Gel'fand--Shilov spaces~[14] and Kawai's Fourier hyperfunctions~[23] (Section~6.5) address such growth. The t.g.l.\ formulation gives the family $\{F_T\}$, the pointwise reconstruction and one pairing, without claiming a Fourier-compatible test space.
\end{sloppypar}

\subsection{Examples for non-locally integrable singular functions}

We now consider functions that are not locally integrable because of singularities, and therefore require regularization even on finite intervals within the classical framework. Non-convergent parameterized functions or singularity functions cannot be integrated even in a finite interval. The t.g.l.\ integral, by truncation of the parameter limit or of functions around nearby singular points, makes it possible as well as other regularization methods. In contrast to the regularization method, divergent boundary term cancellation is built in the t.g.l.\ integration by parts.

\medskip
\noindent \textbf{Example 7. Inverse power function $f(t) = 1/t^n$ $(n=1,2,3,\dots)$.}

This function is singular at a point $t=0$. Such transforms can be interpreted rigorously within the framework of tempered distributions [7]. A case of $n=1$ is considered first. Here we introduce the conventional treatments to show the distinction between them and the t.g.l.\ integral.

The Cauchy p.v.\ is applied to the Fourier transform of a singular function $1/t$ as follows:
\[
F(\omega) = \mathrm{p.v.}\int_{-\infty}^{\infty} \frac{1}{t}\, e^{-i\omega t}\,dt
= \lim_{\varepsilon\to0} \left(\int_{-\infty}^{-\varepsilon} + \int_\varepsilon^\infty\right) \frac{1}{t}\, e^{-i\omega t}\,dt = -i\pi\, \mathrm{sgn}(\omega).
\]
In this derivation, symmetric cancellation of the odd integrand, i.e.\ $\cos(\omega t)/t$, and regularization through the factor $\sin(\omega t)$ are used.

These transforms are classically interpreted within distribution theory through principal-value or finite-part regularizations. Here we show that the same results arise naturally from truncation-based limiting procedures.

The t.g.l.\ construction under symmetric truncation coincides with this p.v.\ prescription term for term, since the excluded neighborhoods $[-\varepsilon,\varepsilon]$ agree; no separate derivation is required.

Next we consider an asymmetrical truncation case:
\[
F(\omega) = \text{t.g.l.}\int_{-\infty}^{\infty} \frac{1}{t}\, e^{-i\omega t}\,dt = \lim_{\varepsilon\to0,\,T\to\infty} F_{\varepsilon,T}(\omega),
\]
where
\[
F_{\varepsilon,T}(\omega) = \left(\int_{-T}^{-\varepsilon} + \int_\varepsilon^{\varepsilon+\varepsilon_0} + \int_{\varepsilon+\varepsilon_0}^{T}\right) \frac{1}{t}\, e^{-i\omega t}\,dt.
\]
$F(\omega)$ does not converge as a classical improper integral, but its t.g.l.\ inverse transform converges to $1/t$ pointwise. Let us consider an integration above,
\[
F_{\varepsilon,\varepsilon_0}(\omega) = \int_\varepsilon^{\varepsilon+\varepsilon_0} \frac{1}{t}\, e^{-i\omega t}\,dt.
\]
It converges point-wise. The inverse transform of it becomes:
\[
f_{\varepsilon,\varepsilon_0}(t) = \lim_{\Omega\to\infty} \frac{1}{2\pi}\int_{-\Omega}^{\Omega} F_{\varepsilon,\varepsilon_0}(\omega)\, e^{i\omega t}\,d\omega = \lim_{\Omega\to\infty} \frac{1}{2\pi}\int_{-\Omega}^{\Omega} \left[\int_\varepsilon^{\varepsilon+\varepsilon_0} \frac{1}{x}\, e^{-i\omega x}\,dx\right] e^{i\omega t}\,d\omega.
\]
We can derive that:
\[
\lim_{\varepsilon\to0,\,T\to\infty} f_{\varepsilon,\varepsilon_0,T}(t) =
\begin{cases}
1/t & \varepsilon \le t \le \varepsilon+\varepsilon_0 \\
0 & \text{otherwise}.
\end{cases}
\]
Similarly, with other integrations we obtain:
\[
\lim_{\varepsilon\to0,\,T\to\infty} f_{\varepsilon,\varepsilon_0,T}(t) = f(t) = 1/t \qquad t \ne 0.
\]

\emph{Remark (manner of truncation).} (1) Convergence of the forward transform as the truncation is removed is not part of the t.g.l.\ definition. (2) If it is considered, it depends on the manner of truncation. For independent exclusions $(-\varepsilon_-,\varepsilon_+)$ and $\omega\ne0$, after $T\to\infty$ and with $o(1)$ referring to $\varepsilon_-,\varepsilon_+\to0$,
\[
F_{\varepsilon_-,\varepsilon_+}(\omega)=-i\pi\,\mathrm{sgn}\,\omega+\ln\frac{\varepsilon_-}{\varepsilon_+}+o(1).
\]
The forward family converges if and only if $\varepsilon_-/\varepsilon_+$ tends to a finite positive number $r$ (if the ratio tends to $0$ or $\infty$, the logarithm diverges and so does the family); the limit is the principal-value transform shifted by the constant $\ln r$, that is, by the multiple $(\ln r)\,\delta$ of the Dirac delta on the time side, and the symmetric choice $r=1$ gives the principal-value transform. (3) The inverse transform does not depend on the manner of truncation: for $t\ne0$ at which the local Dirichlet condition holds, the ordered limits of \Cref{thm:inversion} give $1/t$ pointwise for every choice of exclusions and time truncations, since the added term is supported at $t=0$.
For higher-order singularities $1/t^n$, $n\ge2$, the ordinary Cauchy principal-value prescription does not provide a general regularization of the entire family. In particular, for $1/t^2$, symmetric exclusion leaves a divergent contribution of order $1/\varepsilon$. Higher-order singularities are conventionally treated by finite-part distributions or, equivalently, through distributional differentiation of the principal-value distribution $\mathrm{p.v.}\,1/t$.

The Hadamard f.p.\ method [11] proceeds by writing:
\[
F(\omega) = \mathrm{f.p.}\int_{-\infty}^{\infty} \frac{1}{t^n}\, e^{-i\omega t}\,dt
= \int_{-\infty}^{\infty} \frac{1}{t^n}\, e^{-i\omega t}\,dt - F_{\mathrm{div}}(\omega),
\]
where $F_{\mathrm{div}}(\omega)$ represents the divergent contribution. A central difficulty of this approach lies in identifying the divergent part.

For symmetric truncation, introduce
\[
F_{2,\varepsilon,T}(\omega) = \int_{-T}^{-\varepsilon}\frac{e^{-i\omega t}}{t^2}\,dt + \int_{\varepsilon}^{T}\frac{e^{-i\omega t}}{t^2}\,dt.
\]
Using $1/t^2=-(1/t)'$, integration by parts gives
\[
F_{2,\varepsilon,T}(\omega) = \frac{2\cos(\omega\varepsilon)}{\varepsilon} - \frac{2\cos(\omega T)}{T} - i\omega F_{1,\varepsilon,T}(\omega),
\]
where $F_{1,\varepsilon,T}(\omega)$ denotes the finite-stage transform of $1/t$ already considered above. As $\varepsilon\to0^+$, $2\cos(\omega\varepsilon)/\varepsilon\sim2/\varepsilon$, so symmetric exclusion alone does not give a finite forward-transform limit for $1/t^2$.

Sharp truncation of $1/t$, however, generates boundary delta terms (Section~2.6) whose Fourier transforms reproduce exactly this same boundary contribution $2\cos(\omega\varepsilon)/\varepsilon - 2\cos(\omega T)/T$; by the t.g.l.\ weak-differentiation procedure, these truncation-generated terms cancel it before the generalized limits are taken --- the divergent boundary term is not discarded, but accounted for explicitly through the truncation-generated boundary distributions. The remaining generalized-transform relation is therefore
\[
F_2(\omega) = -i\omega F_1(\omega).
\]
Using $F_1(\omega)=-i\pi\,\mathrm{sgn}(\omega)$ from above,
\[
F_2(\omega) = -i\omega\big[-i\pi\,\mathrm{sgn}(\omega)\big] = -\pi|\omega|.
\]

For $n\ge3$, generalizing via
\[
\frac{1}{t^n} = \frac{(-1)^{n-1}}{(n-1)!}\,\frac{d^{n-1}}{dt^{n-1}}\left(\frac{1}{t}\right)
\]
gives
\[
F_n(\omega) = \frac{\pi(-i)^n}{(n-1)!}\,\omega^{n-1}\,\mathrm{sgn}(\omega),
\]
in agreement with Eq.~(\ref{eq:9}) below. These transform formulas are not new; the purpose of this calculation is to demonstrate, through the $n=2$ case, how the t.g.l.\ finite-stage integration-by-parts and weak-differentiation procedure handles a non-integrable even singularity.
\begin{equation}
\mathrm{f.p.}\,\frac{1}{t^n} \leftrightarrow -i\pi\, \frac{(-i\omega)^{n-1}}{(n-1)!}\, \mathrm{sgn}(\omega).
\label{eq:9}
\end{equation}
In distribution theory, a locally non-integrable singular function such as $1/t$ cannot directly define a distribution; the p.v.\ concept must be introduced instead, and the resulting distribution is denoted $\mathrm{p.v.}\,1/t$. Furthermore, for $1/t^n$ with $n\ge2$, the Hadamard f.p.\ concept must be introduced. For $f(t)=1/t^n$, distribution theory leads to the same result as Eq.~(\ref{eq:9}), where the distribution $\mathrm{f.p.}\,1/t^n$ is understood.

Next we consider inverse transforms. The ordinary inversion becomes,
\[
\overleftarrow{f}(t) = \frac{1}{2\pi}\int_{-\infty}^{\infty} F(\omega)\, e^{i\omega t}\,d\omega = \frac{1}{t} - \lim_{\Omega\to\infty} \frac{\cos(\Omega t)}{t},
\]
which does not converge.

The distributional inversion cancels the above divergent term with a test function as:
\[
\lim_{\Omega\to\infty} \int_{-\infty}^{\infty} \frac{\cos(\Omega t)}{t}\, G(t)\,dt = \lim_{\Omega\to\infty} \int_{-\infty}^{\infty} \cos(\Omega t)\, \frac{G(t)-G(0)}{t}\,dt + \lim_{\Omega\to\infty} \mathrm{p.v.}\, G(0) \int_{-\infty}^{\infty} \frac{\cos(\Omega t)}{t}\,dt = 0,
\]
where the first integral vanishes owing to the Riemann--Lebesgue lemma and the second integral owing to symmetric cancellation of the Cauchy p.v.\ integral.

In the t.g.l.\ framework, the same result is obtained as a limit of truncated classical integrals, without explicitly invoking test functions. In the distributional treatment of the singular functions, the principal value is defined through a symmetric exclusion of the singularity, reflecting a balance between the positive and negative parts of the domain. This symmetry is built into the definition with a test function and ensures cancellation in the sense of distributions. In contrast, the t.g.l.\ definition and the pointwise reconstruction for $t\ne0$ do not require an a priori symmetric truncation; a canonical forward transform does (see the remark above). The distributional formulas are thereby reproduced, while the t.g.l.\ formulation organizes the calculation through finite-stage ordinary integrals and explicit boundary terms rather than through an a priori symmetric prescription.

\medskip
\noindent \textbf{Example 8. Asymmetric frequency-truncation with Hilbert transformation.}

The Hilbert transform is defined for $f \in L^1(\mathbb{R})$ as a convolution integral, understood in the sense of the Cauchy principal value since the kernel $1/x$ is itself singular at $x=0$:
\[
\hat{f}(t) = \frac{1}{\pi}\,\mathrm{p.v.}\!\int_{-\infty}^{\infty} f(t-x)\,\frac{1}{x}\,dx = \frac{1}{\pi}\,\mathrm{p.v.}\!\int_{-\infty}^{\infty} f(x)\,\frac{1}{t-x}\,dx.
\]
The forward transform of $\hat{f}$ becomes:
\[
\hat{F}(\omega) = F(\omega)\,(-i)\,\mathrm{sgn}(\omega).
\]
Define
\[
F^{+}_{\mathrm{ssb}}(\omega) = F(\omega) + i\hat{F}(\omega), \qquad F^{-}_{\mathrm{ssb}}(\omega) = F(\omega) - i\hat{F}(\omega).
\]
Then we obtain the asymmetrically frequency-truncated forward transform as,
\[
F^{+}_{\mathrm{ssb}}(\omega) = 2F(\omega)\cdot \mathbf{1}_{[0,\infty]}(\omega), \qquad F^{-}_{\mathrm{ssb}}(\omega) = 2F(\omega)\cdot \mathbf{1}_{[-\infty,0]}(\omega).
\]
These results give the mathematical foundation of so-called single-side-band (ssb) modulation signals, which are written as
\[
f_{\mathrm{ssb}}(t) = f(t)\cos\omega_c t \mp \hat{f}(t)\sin\omega_c t,
\]
where $\omega_c \gg 1$ is called the carrier frequency. The forward transform becomes for $\omega>0$:
\[
F^{+}_{\mathrm{ssb}}(\omega) = F(\omega-\omega_c)/2 + i\hat{F}(\omega-\omega_c)/2 = F(\omega-\omega_c)\cdot \mathbf{1}_{[\omega_c,\infty]}(\omega),
\]
\[
F^{-}_{\mathrm{ssb}}(\omega) = F(\omega-\omega_c)/2 - i\hat{F}(\omega-\omega_c)/2 = F(\omega-\omega_c)\cdot \mathbf{1}_{[-\infty,\omega_c]}(\omega).
\]
This observation provides a simple interpretation of the Hilbert transform as the correction term that appears when the Fourier inversion is performed using one-sided frequency truncation. The construction of single-sideband modulation using the Hilbert transform provides a typical example of analytic-signal-based spectral localization widely used in communication theory (see, e.g., [15]). This example illustrates that the truncation-based generalized-limit inverse transform naturally accommodates one-sided spectral representations arising in practical signal-processing applications.

\medskip
\noindent \textbf{Example 9. The Dirac delta-convergent sequence of functions.}

The Dirac delta-convergent functions do not converge in the classical sense as their parameter tends to a target value, and are therefore not locally integrable in the limit. Example~9 and Example~10 consider two representative delta-convergent functions,
\[
f(t) = \lim_{\varepsilon\to0}\delta_\varepsilon(t), \qquad \delta_\varepsilon(t) = \frac{1}{2\varepsilon}\,\mathbf{1}_{[-\varepsilon,\varepsilon]}(t),
\]
and
\[
f(t) = \lim_{\Omega\to\infty}\delta_\Omega(t), \qquad \delta_\Omega(t) = \frac{\sin(\Omega t)}{\pi t}.
\]
The relation $\lim_{\Omega\to\infty}\delta_\Omega(t)=\delta(t)$ is proved equivalently in \Cref{thm:c1}. The relation $\lim_{\varepsilon\to0}\delta_\varepsilon(t)=\delta(t)$ is proved as follows: consider the integral
\[
\lim_{\varepsilon\to0}\int_{-\infty}^{\infty}\delta_\varepsilon(t)\big[f(t)-f(0)\big]\,dt = \lim_{\varepsilon\to0}\frac{1}{2\varepsilon}\int_{-\varepsilon}^{\varepsilon}\big[f(t)-f(0)\big]\,dt.
\]
The integral vanishes as $\varepsilon\to0$, since the integrand tends to the derivative $f'(0)$, which is assumed bounded, while the length of the integration interval tends to zero. Thus
\[
\lim_{\varepsilon\to0}\int_{-\infty}^{\infty}\delta_\varepsilon(t)\,f(t)\,dt = f(0)\lim_{\varepsilon\to0}\frac{1}{2\varepsilon}\int_{-\varepsilon}^{\varepsilon}dt = f(0).
\]

For $\delta_\varepsilon(t)$, the (truncated) Fourier transform is
\[
\Delta_\varepsilon(\omega) = \int_{-\infty}^{\infty}\delta_\varepsilon(t)\,e^{-i\omega t}\,dt = \frac{\sin(\omega\varepsilon)}{\omega\varepsilon}.
\]
The inverse Fourier transform is
\[
\lim_{\Omega\to\infty}\frac{1}{2\pi}\int_{-\Omega}^{\Omega}\Delta_\varepsilon(\omega)\,e^{i\omega t}\,d\omega = \lim_{\Omega\to\infty}\frac{1}{2\pi}\int_{-\Omega}^{\Omega}\frac{\sin(\omega\varepsilon)}{\omega\varepsilon}\,e^{i\omega t}\,d\omega = \delta_\varepsilon(t).
\]
Therefore,
\[
\lim_{\varepsilon\to0}\lim_{\Omega\to\infty}\frac{1}{2\pi}\int_{-\Omega}^{\Omega}\Delta_\varepsilon(\omega)\,e^{i\omega t}\,d\omega = \lim_{\varepsilon\to0}\delta_\varepsilon(t).
\]
The limit of the Fourier transform itself also converges pointwise:
\[
\lim_{\varepsilon\to0}\Delta_\varepsilon(\omega) = \lim_{\varepsilon\to0}\frac{\sin(\omega\varepsilon)}{\omega\varepsilon} = 1.
\]
If this pointwise limit is substituted into the inverse integral before the frequency-domain limit is taken, we instead obtain
\[
\overleftarrow{f}(t) = \lim_{\Omega\to\infty}\frac{1}{2\pi}\int_{-\Omega}^{\Omega}\lim_{\varepsilon\to0}\Delta_\varepsilon(\omega)\,e^{i\omega t}\,d\omega = \lim_{\Omega\to\infty}\frac{1}{2\pi}\int_{-\Omega}^{\Omega}1\cdot e^{i\omega t}\,d\omega = \lim_{\Omega\to\infty}\delta_\Omega(t),
\]
which is significantly different, at the finite stage, from $\lim_{\varepsilon\to0}\delta_\varepsilon(t)$. Thus the interchange of the limit $\varepsilon\to0$ with the frequency-domain integration or limit $\Omega\to\infty$ is not permitted, even though $\lim_{\varepsilon\to0}\Delta_\varepsilon(\omega)$ itself converges pointwise. The reason the reversed order yields an incorrect finite-stage result is that $\lim_{\varepsilon\to0}\Delta_\varepsilon(\omega)=1$ holds only for each fixed $\omega$; the convergence is not uniform as $\omega\to\infty$, and it is precisely this unbounded range that the subsequent limit $\Omega\to\infty$ probes, more precisely,
\[
\lim_{\varepsilon\to0}\lim_{\omega\to\infty}\omega\varepsilon\;(=\infty) \neq \lim_{\omega\to\infty}\lim_{\varepsilon\to0}\omega\varepsilon\;(=0).
\]

The two results are nevertheless the same in the sense of the generalized limit, or in the distributional sense:
\[
\lim_{\Omega\to\infty}\delta_\Omega(t) = \lim_{\varepsilon\to0}\delta_\varepsilon(t) = \delta(t).
\]

\medskip
\noindent \textbf{Example 10. Sinc-type delta-convergent function $f(t)=\lim_{\Omega\to\infty}\sin(\Omega t)/\pi t$.}

This function has no finite pointwise limit at any $t$: for $t\neq0$ it oscillates without convergence as $\Omega\to\infty$, while at $t=0$ its value grows without bound. Furthermore, $\delta_\Omega(t)=\sin(\Omega t)/\pi t$ is not absolutely integrable, but its Fourier transform exists without time truncation, with parameter (frequency) truncation, as
\[
\Delta_\Omega(\omega) = \int_{-\infty}^{\infty}\frac{\sin(\Omega t)}{\pi t}\,e^{-i\omega t}\,dt = \int_{-\infty}^{\infty}\frac{\sin(\Omega t)}{\pi t}\,\cos(\omega t)\,dt = \begin{cases}1, & |\omega|<\Omega,\\ 0, & |\omega|>\Omega.\end{cases}
\]
Thus, the generalized-limit Fourier transform is
\[
F(\omega) = \lim_{\Omega\to\infty}\Delta_\Omega(\omega) = \lim_{\Omega\to\infty}\int_{-\infty}^{\infty}\frac{\sin(\Omega t)}{\pi t}\,e^{-i\omega t}\,dt = 1 \qquad (-\infty<\omega<\infty).
\]
The ordinary inverse Fourier transform, taken as an improper integral, as well as the generalized-limit inverse Fourier transform, yields the correct result,
\[
\overleftarrow{f}(t) = \lim_{\Omega\to\infty}\frac{1}{2\pi}\int_{-\Omega}^{\Omega}1\cdot e^{i\omega t}\,d\omega = \lim_{\Omega\to\infty}\frac{\sin(\Omega t)}{\pi t}.
\]
The inverse Fourier transform recovers the original function correctly: frequency-domain truncation is still required here, since $F(\omega)=1$ is not absolutely integrable, but it is applied in precisely the form of the ordinary improper Riemann integral $\lim_{\Omega\to\infty}\int_{-\Omega}^{\Omega}$, rather than as a separate device beyond it.

\subsection{A combined obstruction: local singularity together with non-decay}

The examples of Sections~5.1 and~5.2 illustrate two mathematically distinct origins of non-$L^1(\mathbb{R})$ behavior: non-decay at infinity and local singularity. The following example combines both obstructions in a particularly strong form. It also shows that a function containing inverse-power singularities of arbitrarily high order remains directly accessible to the same finite-stage t.g.l.\ construction. Comparison with the exponentially increasing function of Example~5 further illustrates the common constructive treatment of severe behavior at infinity and at an isolated finite point.

\medskip
\noindent \textbf{Example 11. A parameterized stronger isolated singularity: $f_\alpha(t)=e^{\alpha/t}$, $\alpha>0$.}

It is smooth and bounded on every compact interval avoiding the origin and is therefore t.g.l.-admissible with singular set $S=\{0\}$ according to Definition~2.1. Its behavior at the origin is strongly asymmetric:
\[
\lim_{t\to0^-}e^{\alpha/t}=0, \qquad \lim_{t\to0^+}e^{\alpha/t}=+\infty, \qquad \alpha>0,
\]
and $\int_0^\varepsilon e^{\alpha/t}\,dt$ diverges for every $\varepsilon>0$. Hence the Cauchy principal-value mechanism cannot remove the singularity by cancellation between the two sides of the origin. The Fourier factor $e^{-i\omega t}$ supplies no such local cancellation either, since $e^{-i\omega t}\to1$ as $t\to0$.

Nevertheless, the t.g.l.\ construction applies without modification. For every fixed $0<\varepsilon<T$,
\[
F_{T,\varepsilon}(\omega) = \int_{-T}^{-\varepsilon}e^{\alpha/t}e^{-i\omega t}\,dt + \int_{\varepsilon}^{T}e^{\alpha/t}e^{-i\omega t}\,dt
\]
is an ordinary finite integral. Thus, however strongly $e^{\alpha/t}$ diverges as $t\to0^+$, the singular point is treated by the same sharp-exclusion procedure already used for the algebraic singularities of Example~7.

The relation with the preceding examples can be seen directly from
\[
e^{\alpha/t} = \sum_{n=0}^{\infty}\frac{\alpha^n}{n!t^n} = 1+\frac{\alpha}{t}+\frac{\alpha^2}{2!t^2}+\frac{\alpha^3}{3!t^3}+\cdots, \qquad t\ne0.
\]
For every fixed $0<\varepsilon<T$, this series converges uniformly on $D_{T,\varepsilon}=[-T,-\varepsilon]\cup[\varepsilon,T]$. Therefore,
\[
F_{T,\varepsilon}(\omega) = \sum_{n=0}^{\infty}\frac{\alpha^n}{n!}\int_{D_{T,\varepsilon}} t^{-n}e^{-i\omega t}\,dt.
\]

At the finite stage there is therefore no difficulty in passing between the complete function and its series representation. The complete finite-stage expression is formed before the generalized limiting operations are taken. The generalized limit may then be applied outside the entire sum. This is distinct from moving the generalized limit inside the infinite sum and taking it separately for every term.

For comparison, Example~7 has already established
\[
\mathrm{t.g.l.}\,\frac1{t^n} \longleftrightarrow -i\pi\frac{(-i\omega)^{n-1}}{(n-1)!}\operatorname{sgn}(\omega), \qquad n\ge1,
\]
together with $1\longleftrightarrow2\pi\delta(\omega)$. If the generalized limit is taken term by term, these previously obtained transform pairs give
\[
e^{\alpha/t} \longleftrightarrow 2\pi\delta(\omega) -i\pi\operatorname{sgn}(\omega)\sum_{n=1}^{\infty}\frac{\alpha^n(-i\omega)^{n-1}}{n!(n-1)!}.
\]
This termwise expression displays explicitly the relation between $e^{\alpha/t}$ and the inverse-power transforms obtained in Example~7, in the same formal sense already illustrated for $e^{\alpha t}$ in Example~5: the defining t.g.l.\ construction acts on the complete finite-stage transform before the generalized limits are taken, and does not require interchanging the generalized limit with the infinite series.

\medskip
\noindent \textit{Second-stage generalized limit.} The function $e^{\alpha/t}$ combines a severe isolated singularity at the origin with a second and independent obstruction: $e^{\alpha/t}\to1$ as $t\to\pm\infty$, so it does not decay at infinity either. A single conversion function may be chosen to control both obstructions simultaneously in the second-stage pairing of \Cref{thm:pairing}. Take
\[
g_{\beta,\sigma}(t) = e^{-\beta/|t| - \sigma|t|}, \qquad \sigma>0,\ \beta>\alpha>0.
\]
Here $e^{-\beta/|t|}$ suppresses the essential-type singular behavior at $t=0$, $e^{-\sigma|t|}$ supplies decay at infinity, and $\beta>\alpha$ guarantees decay of $e^{\alpha/t}g_{\beta,\sigma}(t)$ on both sides of $t=0$ (the condition $\beta>|\alpha|$ is preferable if $\alpha<0$ is also to be covered). The conversion function is introduced here only for the second-stage generalized spectral pairing; it is not part of the first-stage definition of the t.g.l.\ Fourier transform. The first-stage construction continues to use only sharp truncation and ordinary finite-domain Fourier integration.

The Fourier transform of $g_{\beta,\sigma}$ is
\[
G_{\beta,\sigma}(\omega) = \int_{-\infty}^{\infty} e^{-\beta/|t| - \sigma|t|}e^{-i\omega t}\,dt.
\]
Since $g_{\beta,\sigma}$ is real and even, $G_{\beta,\sigma}^*(\omega)=G_{\beta,\sigma}(\omega)$, and it may be written in the elementary integral form
\[
G_{\beta,\sigma}(\omega) = 2\int_0^{\infty} e^{-\beta/t - \sigma t}\cos(\omega t)\,dt.
\]

Using \Cref{thm:pairing} at every finite truncation stage (here $g_{\beta,\sigma}$ is $C^\infty$ with rapidly decaying derivatives, so $G_{\beta,\sigma}\in L^1(\mathbb{R})$ and the frequency truncation may be removed directly),
\[
\frac{1}{2\pi}\int_{-\infty}^{\infty} F_{T,\varepsilon}(\omega)\,G_{\beta,\sigma}^*(\omega)\,d\omega = \int_{D_{T,\varepsilon}} e^{\alpha/t}e^{-\beta/|t|-\sigma|t|}\,dt.
\]
Displaying the two sides of the origin explicitly, and using $|t|=-t$ for $t<0$,
\[
= \int_{-T}^{-\varepsilon} e^{(\alpha+\beta)/t+\sigma t}\,dt + \int_{\varepsilon}^{T} e^{-(\beta-\alpha)/t-\sigma t}\,dt.
\]

For $\beta>\alpha>0$ and $\sigma>0$, the time-domain paired integrand tends to zero exponentially as $t\to0^+$ and $t\to0^-$, and also decays exponentially as $|t|\to\infty$. Therefore the paired integral has an ordinary absolutely convergent limit:
\[
\lim_{T\to\infty,\,\varepsilon\to0} \frac{1}{2\pi}\int_{-\infty}^{\infty} F_{T,\varepsilon}(\omega)\,G_{\beta,\sigma}^*(\omega)\,d\omega = \int_{-\infty}^{\infty} e^{\alpha/t}e^{-\beta/|t|-\sigma|t|}\,dt < \infty.
\]
This is the concrete second-stage generalized-limit meaning for the present example.

The first-stage object $F_{T,\varepsilon}(\omega)$ is an ordinary Fourier transform for every fixed $T$ and $\varepsilon$; it need not possess an ordinary pointwise limit as the truncations are removed. The second-stage construction does not require forming an independent pointwise Fourier transform $F_\alpha(\omega)$. Instead, $F_{T,\varepsilon}$ is paired with the ordinary Fourier transform $G_{\beta,\sigma}$ of a suitable conversion function, and the generalized truncation limits are then taken outside the complete pairing. This illustrates \Cref{thm:pairing} particularly clearly, since the present example contains two different obstructions simultaneously --- a nonintegrable isolated singularity and non-decay at infinity --- and a single conversion function $g_{\beta,\sigma}$ controls both in the second-stage pairing.

\begin{remarks}
Example~7 shows how the algebraic singularities $1/t^n$ are treated by the same sharp-exclusion construction for every finite $n$. The present example extends this observation to $e^{\alpha/t} = 1+\alpha/t+\alpha^2/(2!t^2)+\cdots$, which contains inverse-power singularities of arbitrarily high order. Conventional principal-value cancellation does not apply to this function, and the finite-order Hadamard procedures used for individual algebraic singularities do not directly provide the same construction for the complete function.

Growth at infinity, as in $e^{\alpha t}$ (Example~5), and the local essential singularity of $e^{\alpha/t}$ at $t=0$ are structurally different obstructions. Since $e^{\alpha/t}$ is not locally integrable at $0^+$, distributions do not apply as they stand, nor does the finite-part method~[11]; Sato's hyperfunctions~[21] can accommodate it, with extensions across $t=0$ that are not unique (Section~6.3). In the second-stage pairing above, the conversion function must satisfy $\beta>\alpha$, so that it is adapted to the strength of the singularity. The first-stage t.g.l.\ construction requires no such adaptation, since it is defined independently of $\alpha$.

\emph{Remark (manner of truncation).} (1) Convergence of the forward family is outside the t.g.l.\ definition, and for $e^{\alpha/t}$ it does not occur. (2) No manner of truncation rescues it: for any exclusion $(-\varepsilon_-,\varepsilon_+)$ with $\varepsilon_+\to0$, the integral of $e^{\alpha/t}$ over $(\varepsilon_+,2\varepsilon_+)$ exceeds $\varepsilon_+e^{\alpha/2\varepsilon_+}$, so the real part of $F_{T,\varepsilon}$ grows without bound. The divergence is exponential, so there is no finite part to subtract. (3) The inverse transform does not depend on the manner of truncation and converges pointwise for $t\ne0$ (\Cref{thm:inversion}). For the auxiliary function $g_{\beta,\sigma}$, $\beta>\alpha$, a scalar pairing exists (\Cref{thm:pairing}); it is one particular pairing and not an extension of $e^{\alpha/t}$ across the origin, and this example establishes no continuous limiting functional on a specified common test space (Section~6.4).

The t.g.l.\ formulation thus returns the pointwise reconstruction for $t\ne0$, a divergent forward family, and one pairing. The ordinary function $e^{\alpha/t}$ is t.g.l.-admissible without first being extended across $t=0$ as a generalized function: the singular point is sharply excluded, the finite-domain Fourier integral is an ordinary integral, and the generalized limits are applied afterwards. This does not assert any inclusion relation between the theories. Together with Example~5, it shows that increasingly rapid growth or increasingly severe singular behavior does not require more elaborate machinery at the finite stage; the difference appears only in the subsequent limits.
\end{remarks}

\subsection{Examples of functions without $L^1(\mathbb{R})$ for which truncation is not required}

The examples discussed so far, such as constants, polynomials, and periodic functions, require time truncation in order to define the forward transform through generalized limits. However, the use of time truncation is not an essential requirement of the t.g.l.\ formulation itself.

There exist oscillatory functions for which the ordinary improper Fourier integral already converges, although absolute convergence in the Lebesgue $L^1$ sense fails. The two examples below are examined against four classical notions of convergence: absolute ($L^1$) convergence, which fails for both; mean-square ($L^2$) convergence, which holds for the sinc function of Example~12 but fails for the chirp signal of Example~13; the sequential-regularization approach, which is not needed in either case since no auxiliary regularizing sequence must be introduced; and the distributional definition, which is not invoked in either case, since the oscillatory integrals themselves are evaluated directly.

If a damping factor such as $e^{-\sigma|t|}$ were introduced instead, both examples could equally be treated by evaluating the resulting absolutely convergent integral and taking the limit $\sigma \to 0^+$; and since both functions are bounded and therefore define tempered distributions, both could equally be treated by the distributional definition $\langle \mathcal{F}T,\varphi \rangle = \langle T,\mathcal{F}\varphi\rangle$ through pairing with a test function $\varphi \in \mathcal{S}(\mathbb{R})$. Neither auxiliary device is required here, however, since the oscillatory integrals themselves already converge directly as ordinary improper Riemann integrals. Typical examples include sinc-type functions and quadratic-phase chirp signals.

\medskip
\noindent \textbf{Example 12. Sinc function $f(t) = \sin(\omega_0 t)/(\pi t)$, $\omega_0>0$.}

This function is bounded and oscillatory but not absolutely integrable, since $|\sin(\omega_0 t)|$ does not decay and $\int_{-\infty}^{\infty} |\sin(\omega_0 t)/t|\,dt = \infty$. Nevertheless, the same computation used for $\Delta_\Omega(\omega)$ in Example~10 shows that its Fourier transform converges directly, without time truncation, as an ordinary oscillatory improper integral:
\[
F(\omega) = \int_{-\infty}^{\infty} \frac{\sin\omega_0 t}{\pi t}\, e^{-i\omega t}\,dt = \int_{-\infty}^{\infty} \frac{\sin\omega_0 t}{\pi t}\, \cos(\omega t)\,dt =
\begin{cases}
1 & |\omega|<\omega_0 \\
1/2 & |\omega|=\omega_0 \\
0 & |\omega|>\omega_0,
\end{cases}
\]
in accordance with the midpoint convention at the jump discontinuities $\omega=\pm\omega_0$.
Since $F(\omega)$ is itself compactly supported on $[-\omega_0,\omega_0]$, the inverse transform is likewise an ordinary integral over a finite interval, requiring no frequency-domain truncation either:
\[
\frac{1}{2\pi}\int_{-\omega_0}^{\omega_0} e^{i\omega t}\,d\omega = \frac{\sin\omega_0 t}{\pi t} = f(t).
\]
Thus the sinc function, like the chirp signal treated next, admits both its forward and inverse transform without truncation in either domain --- but here it is the compact support of $F(\omega)$, rather than oscillatory cancellation of an unbounded integrand, that removes the need for frequency-domain truncation on the inverse side.

Although not absolutely integrable, the sinc function belongs to $L^2(\mathbb{R})$ (Remark following \Cref{prop:l2opt}), and \Cref{prop:l2opt}'s classical $L^2$ convergence applies to it in the sharpest possible form, with the truncation error vanishing exactly for $\Omega \geq \omega_0$ rather than merely in the limit. Its transform is likewise obtained without appeal to a sequential regularizing factor or to the distributional duality pairing $\langle \mathcal{F}T,\varphi\rangle = \langle T,\mathcal{F}\varphi\rangle$: the oscillatory integral above is evaluated directly.

\medskip
\noindent \textbf{Example 13. Quadratic-phase chirp signal $f(t) = \exp(i\pi t^2)$.}

for which
\[
\int_{-\infty}^{\infty} |f(t)|\,dt = \infty,
\]
while the oscillatory Fresnel integral
\[
\lim_{T\to\infty} \int_{-T}^{T} e^{i\pi t^2}\,dt = e^{i\pi/4},
\]
exists in the improper oscillatory sense. The corresponding Fourier transform is given by
\[
F(\omega) = \exp(i\pi/4)\, \exp(-i\omega^2/4\pi),
\]
Nor does the chirp signal belong to $L^2(\mathbb{R})$: since $|f(t)|=1$ for all $t$, $\int_{-\infty}^{\infty}|f(t)|^2\,dt=\infty$, and likewise $\int_{-\infty}^{\infty}|F(\omega)|^2\,d\omega=\infty$ since $|F(\omega)|$ is also constant (Remark following \Cref{prop:l2opt}). The chirp signal therefore lies outside both the $L^1$ and $L^2$ frameworks, and its treatment relies entirely on oscillatory cancellation.

The t.g.l. inverse transform is obtained, as for any admissible function, by truncating in both the time and frequency domains and taking the ordered limits $\Omega\to\infty$ followed by $T\to\infty$: for the truncated function $f_T(t) = e^{i\pi t^2}\mathbf{1}_{[-T,T]}(t)$, the truncated transform $F_T(\omega)$ is a truncated Fresnel integral with no elementary closed form (Appendix~C), and
\[
\overleftarrow{f}(t) = \lim_{T\to\infty}\lim_{\Omega\to\infty} \frac{1}{2\pi}\int_{-\Omega}^{\Omega} F_T(\omega)\, e^{i\omega t}\,d\omega = \exp(i\pi t^2)
\]
recovers $f(t)$ through the orthogonal localization of the Dirichlet kernel, exactly as in the general construction of \Cref{thm:inversion}; the full derivation, for the general chirp $e^{iax^2}$ of which this is the case $a=\pi$, is given in Appendix~C. Unlike the sequential approach, no auxiliary smoothing sequence such as $e^{-t^2/n^2}$ is introduced here; and unlike the distributional definition, the oscillatory integral $\int e^{i\pi t^2}\, e^{-i\omega t}\,dt$ itself is evaluated directly, through truncation and the ordered limit, rather than replaced by a duality pairing with a test function.

This dual-truncation route is not, however, the only way to recover $f(t)$ from $F(\omega)$ in this particular example. Since $F(\omega)$ is itself an oscillatory quadratic-phase function of the same type as $f(t)$, differing only by a sign change and rescaling in the exponent, the untruncated inverse integral
\[
\frac{1}{2\pi}\int_{-\infty}^{\infty} F(\omega)\, e^{i\omega t}\, d\omega
\]
converges directly by the same oscillatory-cancellation mechanism invoked for the forward transform above, and reproduces $f(t)$ exactly, without truncation in either domain (Appendix~C). This second route is available only because of the chirp's self-duality under the Fourier transform; it does not mean that the general dual-truncation construction of \Cref{thm:inversion} can be dispensed with for other admissible functions, most of which lack this self-duality.

For comparison, consider the ordinary Gaussian function
\[
g(t) = \exp(-\pi t^2) \leftrightarrow G(\omega) = \exp(-\omega^2/4\pi),
\]
which belongs to $L^1(\mathbb{R})$ and admits an ordinary Fourier transform.

The Gaussian and the chirp illustrate two structurally opposite mechanisms by which a function can be self-dual under the Fourier transform. The Gaussian decays rapidly, is absolutely integrable, and its transform is obtained through ordinary absolutely convergent integration; Lebesgue and Riemann integration agree trivially in this case, and no oscillatory cancellation is involved. It is precisely this rapid decay and absolute integrability that make the Gaussian, and more generally the rapidly decreasing test functions of the Schwartz class it exemplifies, well suited to the Lebesgue-integral foundations of classical distribution theory.

The chirp, by contrast, has constant modulus and no decay whatsoever; as shown above and in Appendix~C, it is nevertheless self-dual under the Fourier transform up to a sign change and rescaling of the quadratic-phase coefficient, but this self-duality is reached entirely through oscillatory cancellation of the orthogonal kernel $\exp(i\omega t)$ rather than through absolute convergence. The two functions therefore occupy opposite ends of the same classical dichotomy: the Gaussian is reached through the absolute-convergence mechanism common to both the Riemann and Lebesgue integrals, while the chirp is reached through the oscillatory-convergence mechanism specific to the improper Riemann integral, so that the improper Riemann integral --- rather than the Lebesgue integral alone --- is required to reach both within a single elementary, constructive framework.

The chirp signal of Example~13, together with the sinc function of Example~12, shows that absolute convergence in the Lebesgue $L^1$ sense is not a necessary condition for meaningful Fourier transforms when oscillatory cancellation or compact spectral support is present. For such oscillatory functions, the generalized Fourier transform can be defined directly without time truncation. This indicates that the essential feature of the t.g.l.\ formulation is not truncation itself, but the generalized limiting interpretation of classical oscillatory integrals. Time truncation is required only when such improper oscillatory limits do not exist.

Truncated chirp signals denoted as,
\[
f_T(t)=\exp(ik\pi t^2)\, \mathbf{1}_{[0,T]}(t), \qquad k:\ \text{constant}
\]
play an important role in wireless communications or Radar systems, by forming a narrow pulse signal to measure radio channel characteristics or to monitor objects by radio waves. The minimum time-resolution $\Delta t$ is proportional to the minimum space-resolution $\Delta d$; for a simple one-way propagation geometry, this is approximately $\Delta d\approx c\Delta t$, where $c$ represents the light velocity.

Narrow pulse radio wave signals are hard to generate since its peak power becomes high inversely proportional to the time-resolution $\Delta t$. A device to combat this problem is the matched linear filter (e.g.\ [15]) whose impulse response $h(t)$ is the time-reversed, complex-conjugated form of the chirp signal pulse $f_T(t)$, i.e.,
\[
h(t)= f_T^{*}(T-t).
\]
Then the output signal $y(t)$ is the convolution integral as
\[
y(t) = \int_{-\infty}^{\infty} h(t-x)\, f_T(x)\,dx = \int_{-\infty}^{\infty} f_T^{*}(T-t+x)\, f_T(x)\,dx
\]
here $y(t)$ represents also the autocorrelation function of $f_T(t)$. The maximum peak appears at $t=T$.

A calculated example is shown in Figures 1 and 2.

The Fourier transform $Y(\omega)$ of $y(t)$ is given as:
\[
Y(\omega) = H(\omega)\, F_T(\omega) = e^{-i\omega T}\, F_T^*(\omega)\, F_T(\omega) = \big|F_T(\omega)\big|^2\, e^{-i\omega T}.
\]
As shown in Appendix B, the truncated chirp signal of the present model satisfies the reciprocal relation between time resolution $\Delta t$ and bandwidth $\Omega$, specific to this truncated-chirp construction rather than a universal time-bandwidth relation, as
\[
\Omega\, \Delta t \approx 4\pi,
\]
which demonstrates that increasing the finite observation duration increases the effective spectral bandwidth and improves the achievable time-domain resolution after signal synthesis.

\begin{figure}[H]
\centering
\includegraphics[width=0.75\textwidth]{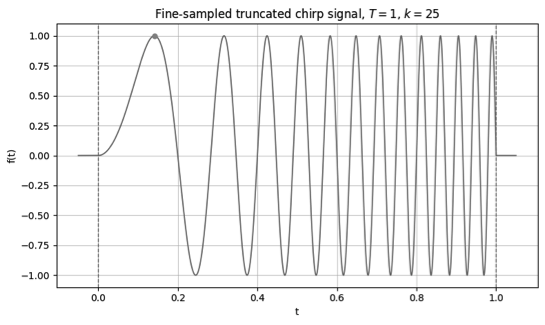}
\caption{Truncated chirp signal.}
\end{figure}

\begin{figure}[H]
\centering
\includegraphics[width=0.75\textwidth]{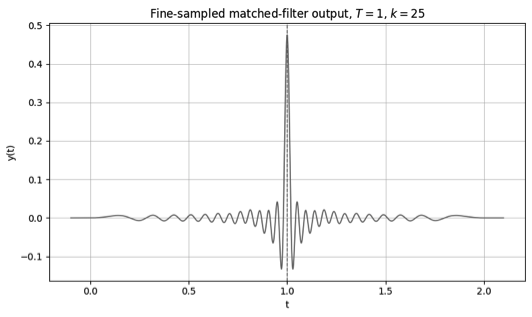}
\caption{Matched filter output signal.}
\end{figure}

\section{Conceptual Discussion}

The preceding sections defined the t.g.l.\ formulation, established its inversion properties, and illustrated it through examples. We now summarize the structural meaning of the construction and its relation to established generalized Fourier theories. A meaningful Fourier-transform result requires some sense of convergence --- pointwise, $L^p$, pairing, or analytic continuation --- since $f\in L^1(\mathbb{R})$ is too restrictive for most functions of interest; how wide a class beyond $L^1$ an approach covers, and in which sense, is the natural criterion for the comparisons below. The essential feature of the t.g.l.\ formulation is that pointwise recovery of $f(t)$ from the inverse transform does not require the untruncated forward transform $F(\omega)=\lim_{T\to\infty}F_T(\omega)$ to exist: the reconstruction works directly with the finite-stage transforms $F_T(\omega)$ and the prescribed ordered limits (Section~6.2). For example, for a BIBO-stable LTI system the output $y(t)$ can be recovered pointwise even when the forward transform of the input diverges (Section~4.1).

\subsection{Reciprocal-domain Structure, Ordered Limits, and Oscillatory Localization}

The Fourier kernel $\exp(i\omega t)$ expresses a reciprocal relation between two Fourier-conjugate domains, $t$ and $\omega$.

This reciprocity is one of four interrelated properties of the exponential kernel $K(t,\omega)=e^{i\omega t}$ that together underlie its central role in Fourier analysis.

First, \textbf{reciprocal duality}: $K(t,\omega)=K(\omega,t)$, so that, apart from the conjugation or sign convention distinguishing the forward and inverse transforms, the reciprocal operation uses the same kernel structure with the roles of $t$ and $\omega$ interchanged.

Second, \textbf{orthogonality and localization}: products of two kernels reduce to an exponential depending only on the frequency difference, $e^{i\omega t}e^{-i\omega' t}=e^{i(\omega-\omega')t}$, whose integration over an expanding time interval produces the delta localization underlying the mean-orthogonality relation of Section~3.3; conversely, integration over an expanding frequency interval produces the Dirichlet-kernel localization examined below.

Third, the \textbf{eigenfunction property}: $d^n e^{i\omega t}/dt^n=(i\omega)^n e^{i\omega t}$, so that constant-coefficient differential operators, and more generally linear time-invariant systems, act on $e^{i\omega t}$ by the algebraic rule $e^{i\omega t}\mapsto H(\omega)e^{i\omega t}$, as discussed in Section~4.3.

Fourth, the \textbf{character property}, $e^{i\omega(t_1+t_2)}=e^{i\omega t_1}e^{i\omega t_2}$, which underlies the translation and convolution theorems, $f(t-t_0)\leftrightarrow e^{-i\omega t_0}F(\omega)$ and $f*g\leftrightarrow FG$, used throughout Section~4.

\[
\boxed{
\begin{array}{c}
\text{reciprocal duality}\\
\text{orthogonality/localization}\\
\text{eigenfunction property}\\
\text{translation/convolution character property}
\end{array}}
\]

These four properties of the Fourier kernel, combined with the mechanism of truncation and limit --- the ``t.g.l.'' of truncate-and-generalized-limit --- constitute the essential philosophy of the t.g.l.\ Fourier transform. The kernel's structural properties determine what the Fourier transform accomplishes; truncation and limit determine how the t.g.l.\ construction accomplishes it, using only elementary, classical means.

The truncation-and-limit mechanism rests on the concept of limit combined with the concept of infinity, two of the most fundamental concepts in calculus: a function is first restricted to a finite domain, where every integral involved is an elementary Riemann integral, and the effect of that restriction is subsequently removed through an explicitly prescribed limiting process, rather than through membership in a predefined generalized-function space. The same principle governs the local singularity exclusion used for functions such as $1/t^n$ (Section~2.4): the limit $\varepsilon\to0$ is taken on the parameter of the excluded neighborhood, not evaluated at the singular point. The outer truncation limit and the local exclusion limit are thus instances of the same t.g.l.\ mechanism.

From the constructive viewpoint adopted here, these four properties of the Fourier kernel are more fundamental than any particular generalized-function representation of the transform; their relation to the distributional formulation is discussed in Section~6.3.

The forward transform $f_T\to F_T$ is a single oscillatory integral, whereas the inverse involves a second, reciprocal integration over frequency, and it is this second integration that generates the localization. Substituting $F_T$ into $f_{T,\Omega}(t)=\frac{1}{2\pi}\int_{-\Omega}^{\Omega}F_T(\omega)e^{i\omega t}\,d\omega$ and interchanging the finite-stage integrations gives
\[
f_{T,\Omega}(t) = \int_{-\infty}^{\infty} f_T(x)\,\frac{\sin(\Omega(t-x))}{\pi(t-x)}\,dx.
\]
For every finite $\Omega$ the kernel is an ordinary oscillatory function; no Dirac delta is assumed. Under \Cref{thm:inversion}, $\lim_{\Omega\to\infty}f_{T,\Omega}(t)=\tfrac12[f_T(t^-)+f_T(t^+)]$, equal to $f_T(t)$ at continuity points; only afterward is the time truncation removed, giving the finite-stage cycle
\[
f \longrightarrow f_T \longrightarrow F_T \longrightarrow f_T \longrightarrow f.
\]

The order of limits reflects the order in which the truncations enter the construction: the time truncation first, so that $F_T(\omega)$ is an ordinary Fourier integral, then the frequency truncation, to obtain a finite double integral. The frequency truncation is therefore removed first, at fixed $T$, and $T\to\infty$ afterward. For $f\in L^1(\mathbb{R})$ standard interchange theorems may make the order immaterial; outside that setting the reversed order may fail (see the examples in Section~3). The distributional shorthand $D_\Omega\to\delta$ plays no role in the construction. Sharp time truncation spreads $F_T(\omega)$ over the whole frequency axis (Observation~3.2), but \Cref{thm:inversion} does not require spectral concentration.

\subsection{Forward Spectral Meaning and Inverse Reconstruction}

Three distinct objects must be kept apart throughout this formulation:
\[
\text{finite-stage transform} \;\ne\; \text{pointwise untruncated transform} \;\ne\; \text{generalized spectral pairing}.
\]
First, $F_T(\omega)$ is an ordinary finite-stage Fourier transform, well defined for every finite $T$. Second, the pointwise limit $F(\omega) = \lim_{T\to\infty}F_T(\omega)$ may fail to exist. Third, when an independent generalized spectral meaning is required despite this failure, it is supplied --- not automatically, but when a suitable auxiliary function $G$ can be found --- through the pairing-based generalized limit of \Cref{thm:pairing},
\[
\lim_{T\to\infty} \frac{1}{2\pi}\int_{-\infty}^{\infty} F_T(\omega)\,G^{*}(\omega)\,d\omega.
\]
A successful pairing against one auxiliary function is not distributional convergence: a genuine tempered distribution requires the pairing to exist for every Schwartz test function, together with the linearity and continuity of the resulting functional on $\mathcal{S}(\mathbb{R})$ (Section~3.9).

Pointwise inverse reconstruction, established through the cycle
\[
f \longrightarrow f_T \longrightarrow F_T \longrightarrow f_T \longrightarrow f,
\]
is structurally separate from both of these: it requires neither a pointwise untruncated spectrum $F(\omega)$ nor an independent generalized spectral pairing. The separation between generalized forward spectral interpretation (\Cref{thm:pairing}) and pointwise reconstruction (\Cref{thm:inversion}) is a basic structural feature of the t.g.l.\ formulation, not an incidental consequence of particular examples.

In LTI applications, the auxiliary frequency-domain factor of the second-stage pairing may be the system transfer function $H(\omega)$, so that the pairing corresponds to multiplication by the transfer function of a physical system. The finite-stage output $y_T=h*f_T$, its limit, and the pointwise recovery of $y(t)$ are treated in Section~4.1 and \Cref{prop:lti}.

A common thread in the approaches surveyed in Section~6.3 is the requirement that
the forward transform exist in some sense, as a canonical object. The t.g.l.\ definition gives
up this requirement, so that $f$ is recovered pointwise within elementary
calculus whatever the growth or local singularity of $f$; the role of canonical
convergence is discussed in Section~6.6.

\subsection{Relation to Established Generalized Fourier Theories}

Distribution theory enlarges the class of mathematical objects and defines generalized functions through a suitable functional-analytic structure. The t.g.l.\ formulation takes a different route: it retains ordinary finite-stage functions and ordinary Fourier integrals, extending instead the limiting procedure applied to them. Thus, at every finite truncation stage,
\[
F_T(\omega)=\int_{-\infty}^{\infty} f_T(t)e^{-i\omega t}\,dt
\]
is an ordinary Fourier integral.

The conventional approaches (distribution theory, Lighthill's good-function sequences~[9] and Feichtinger's mild distributions~[10,17,19,20]) differ from the t.g.l.\ definition in one respect: for the inverse transform they presuppose that the forward transform exists in some sense, as a distribution, a regular sequence of good functions, or an element of $\mathcal S_0'$. The t.g.l.\ definition does not; the inverse is formed from the finite-stage family $\{F_{T,\varepsilon}\}$ by the ordered limits of \Cref{prop:recon}. The difference shows in how far the treatable functions extend beyond tempered distributions. As far as we understand, all three frameworks operate within $\mathcal S'$ ($\mathcal S_0'\subset\mathcal S'$), so by \Cref{prop:tempered} none gives a forward transform for the non-tempered $e^{\alpha t^2}$. The function $e^{\alpha/t}$ is not locally integrable at $0^+$ and has no extension to a distribution on $\mathbb{R}$. The t.g.l.\ formulation gives the family $\{F_{T,\varepsilon}\}$ and the reconstruction at the non-excluded points for these functions, and the same argument extends without modification to $e^{t^n}$ and $e^{1/t^n}$ $(n>1)$. It claims no value at the singular point and no generalized transform. We do not claim that the frameworks above cannot give a pointwise reconstruction; the difference concerns the existence of the forward transform.

Beyond tempered distributions, ultradistribution theories replace the Schwartz
space by test-function spaces with other smoothness and growth conditions;
Gel'fand--Shilov theory includes analytic-function test spaces, for which
analytic functionals arise; and hyperfunctions are boundary values
$u=F_+(t+i0)-F_-(t-i0)$ of holomorphic functions, Fourier hyperfunction
theory~[22] adapting this to growth at infinity. These are systematic
frameworks that also contain generalized objects that are not functions: for
example, $e^{1/z}$ defines a hyperfunction supported at the origin, although
it is not a distribution~[21], and Smirnov~[26] showed that the Fourier
transform of Sato's hyperfunctions leads to a space of ultrafunctionals
containing the Gel'fand--Shilov analytic functionals. As far as we understand,
such theories can provide a transform for functions such as $e^{\alpha t^2}$ as
a functional on a Fourier-compatible space, and would call for an appropriately
chosen space for $e^{1/t^n}$ $(n>1)$. For $e^{\alpha/t}$ an extension across
$t=0$ must be specified, and it is determined only up to a hyperfunction
supported at the origin; none is constructed here.

\begin{sloppypar}
The advanced theories thus locate the generalization in the test-function
spaces and the duality, whereas t.g.l.\ locates it in finite-stage ordinary
Fourier operations and ordered limits; the second-stage pairing of
\Cref{thm:pairing} is a device for assigning spectral meaning, not a new
test-function space. The scopes therefore cannot be compared by inclusion: the
advanced theories are more comprehensive, since they include generalized
objects that are not ordinary functions, such as the Dirac delta, which is not
itself t.g.l.-admissible although delta-convergent ordinary families can be
treated at finite stages. The t.g.l.\ formulation extends the operational reach
of Fourier analysis while retaining its elementary finite-stage structure,
agrees with established Fourier theory where both apply (Section~3.8), and is a
complementary constructive route to results some of which are obtainable there
by different mechanisms (see Section~6.4 on scope).
\end{sloppypar}

\subsection{Scope, Limitations, and Rigor of the t.g.l.\ Formulation}

The scope of the t.g.l.\ formulation is determined operationally: a function falls within its reach if the required finite-stage ordinary integrals and ordered limits can be formed, rather than by membership in a single predefined global function space. Direct t.g.l.-admissibility (Definition~2.1) requires only that the function be Riemann integrable on the compact intervals retained after excluding a locally finite set of isolated singularities; no absolute integrability over $\mathbb{R}$, and no growth condition at infinity, is imposed merely to form the finite-stage transform.

This direct, finite-stage admissibility must be kept distinct from pointwise reconstructibility at a specified point, which additionally requires the local regularity hypothesis of \Cref{thm:inversion}.

Several limitations follow directly from this operational scope. The direct t.g.l.\ formulation takes ordinary functions as its primary objects. More general functional-analytic theories additionally admit generalized objects --- such as distributions, infinite-order analytic functionals, and hyperfunctions --- that need not be representable by ordinary functions, although some such objects may arise within the t.g.l.\ framework as limits of ordinary-function families. In particular, the primitive Dirac delta distribution is not itself a directly t.g.l.-admissible ordinary target function, though delta-convergent ordinary families may be treated at each finite stage (Examples~9, 10). Differentiability of the original function is not a general requirement of the t.g.l.\ construction; discontinuities and nondifferentiable functions can be treated when the relevant finite-stage integrals and prescribed limits exist (Section~2.6).

The t.g.l.\ formulation should therefore not be interpreted as defining a function space that simply contains or is contained in the standard $L^1$, $L^2$, or tempered-distribution frameworks. Its scope is determined operationally by the existence of the finite-stage transforms and the prescribed generalized and reconstruction limits. This permits treatment of some nondifferentiable, non-absolutely-integrable, rapidly growing, or locally singular ordinary functions, but does not imply that every generalized function covered by distributional, sequential, or other functional-analytic theories is t.g.l.-admissible.

Isolated singularities more severe than removable ones may be handled by local exclusion, as throughout this paper, but the existence and uniqueness of any resulting generalized limit must be established case by case rather than asserted in general. Similarly, the existence and adequacy of the pairing underlying \Cref{thm:pairing} must be verified case by case, as discussed in Section~6.2.

The generalized spectral pairing of \Cref{thm:pairing} lies outside the t.g.l.\ definition, which consists of the finite-stage family $\{F_{T,\varepsilon}\}$ and the ordered reconstruction and requires no pairing limit or auxiliary function; the pairing is an optional construction for assigning spectral meaning. Canonicality is relative to a specified probe class: on a suitable common test space, convergence of all pairings and continuity of the limiting action may identify a unique generalized Fourier transform, which the core t.g.l.\ definition does not assert. For tempered distributions, for example a polynomially bounded $f$ paired with all rapidly decreasing functions, the limit exists for every such function and reproduces the classical distributional Fourier transform (Section~3.8); likewise for Gel'fand--Shilov generalized functions with appropriate test spaces.

For $f=e^{\alpha/t}$, $\alpha>0$, the pairing diverges for smooth probes positive near zero. The specially chosen probes of Example~11 suppress this, but this does not establish a continuous functional on a Fourier-compatible test space, nor that none exists. Pointwise reconstruction for $t\ne0$ (\Cref{thm:inversion}) is independent of this; no value at, or extension across, the singular point is claimed.

The t.g.l.-admissible class should not be described by a simple inclusion relation with $\mathcal{S}'(\mathbb{R})$: as discussed in Section~6.3, it reaches some functions outside $\mathcal{S}'(\mathbb{R})$ case by case, while distribution theory and the broader generalized-function theories remain the more comprehensive, systematic frameworks, particularly for generalized objects that are not themselves ordinary functions.

Accordingly, the present theory makes no claim of maximal generality. Its intended contribution is a constructive extension of Fourier transformation based on ordinary finite-stage integrals and explicitly ordered limits, applicable whenever those operations can be carried through.

A natural objection is that truncation discards the behavior of the function outside the truncation interval. At every finite stage the truncated function is an ordinary compactly supported function, so its Fourier transform and the frequency-truncated inverse are classical operations; the truncation is not asserted to equal the original function globally. The original function is recovered through the prescribed ordered limits,
\[
\overleftarrow{f}(t) = \lim_{T_1,T_2\to\infty}\left[\lim_{\Omega\to\infty} \frac{1}{2\pi}\int_{-\Omega}^{\Omega}\left(\int_{-T_1}^{T_2}f(x)e^{-i\omega x}\,dx\right)e^{i\omega t}\,d\omega\right],
\]
under the hypotheses of \Cref{thm:inversion}. The rigor of the construction thus rests on three requirements: existence of the finite-stage ordinary Fourier transforms, validity of finite-stage inversion under the local condition of \Cref{thm:inversion}, and existence of the prescribed outer limits. These were established in Section~2 and Section~3.2; the construction does not rest on assigning values to the function at infinity.

Two related objections concern the singular points and infinity. No finite-stage integral crosses a singular point: each is an ordinary integral of a bounded function over a compact set that excludes a neighborhood of the singular set, so no principal-value or finite-part prescription enters, and, since the value of $f$ at a single point does not affect any such integral, no value at a singular point is needed or claimed. This does not mean that the contribution of the excluded neighborhood is negligible: for $f(t)=1/t^{2}$ it equals $2/\varepsilon$ up to a bounded term and diverges as $\varepsilon\to0$. Such a divergence is not discarded or regularized; it is localized at the singular point, does not enter the reconstruction at $t\neq0$, and is the reason no value at the singular point is claimed. Likewise, no value at infinity is used. For fixed $t\notin S$, $f_{T,\varepsilon}(t)=f(t)$ once $T$ is sufficiently large and $\varepsilon$ sufficiently small, so the outer limits are eventually stationary; the content of the reconstruction lies in the inner frequency limit, which is governed by the classical Riemann localization principle and depends only on $f$ near $t$ (\Cref{thm:inversion}), independently of the manner of truncation (Example~7, Remark).

From a physical viewpoint, excluding the values at singular points and at infinity is natural. Infinity cannot be reached by any measurement, and a quantity that diverges at a point is not observed as such; what is observed is its effect on a finite region through a finite-resolution instrument. Any method that makes such a function integrable therefore rests on an interpretation of the divergent part: damping factors, test-function pairings, regularization, and analytic continuation are all of this kind. The t.g.l.\ formulation is also an interpretation, but a minimal one: it uses only ordinary integrals over finite regions, and it lets the truncation approach the divergent part as closely as desired while the frequency-domain computation remains exact at every stage. Where the singular behavior is itself the object of interest, for example a point source, the established generalized-function theories are the more appropriate framework.

\subsection{Possible Complex-Domain Extension}

Functions of exponential or infra-exponential growth are already accommodated by established theories beyond the tempered-distribution setting: Titchmarsh [3] developed generalized Fourier integrals by contour methods (Section~5.1), and Chung and Kim [22] related distributions of exponential growth to Fourier hyperfunctions. Example~6, $e^{\alpha t^n}$ with $n\ge2$, grows faster than $e^{a|t|}$ for every $a>0$; broader theories address such growth through specialized test-function spaces (Gel'fand--Shilov [14]) or complex-analytic frameworks (Kawai's Fourier hyperfunctions [23]). No claim is made that these frameworks share the same scope, or that every such function belongs to each of them.

A complementary complex-time interpretation, in which an ordered limiting structure analogous to the t.g.l.\ construction arises from contour deformation, is given in Appendix~D; it is not required for the definition or application of the t.g.l.\ formulation.

\paragraph{Limits of the canonical route for $e^{\alpha t^n}$, $n>1$.}
For $f(t)=e^{\alpha t^n}$ with $n>1$, a canonical forward transform need not be available in $\mathcal S'$ (it exists in $Z'$ by \Cref{prop:weaklimit}, with the limitations noted there): for $n=2$ the truncated transforms do not converge in $\mathcal{S}'$ (\Cref{prop:tempered}). If time and frequency are complexified and an analytic functional is assumed, a formal spectrum can be written down; for $e^{\alpha t^2}$ it is, formally, $i\sqrt{\pi/\alpha}\,e^{\omega^2/(4\alpha)}$ (cf.\ Example~6 and Appendix~D). Whether such an object is useful in the operate-and-return setting is not clear. Consider the Gaussian filter $h(t)=e^{-at^2}$ of Section~4.2 applied to $x(t)=e^{\alpha t^2}$. The time-domain convolution converges if and only if $a>\alpha$. The formal product $H(\omega)F(\omega)$ behaves like $\exp\!\big(\omega^2(\tfrac{1}{4\alpha}-\tfrac{1}{4a})\big)$: for $a>\alpha$ it grows and admits no classical or tempered inverse, and for $a<\alpha$ it decays while the convolution diverges. The canonical route thus works in the opposite regime to the time-domain operation, whereas the t.g.l.\ finite-stage output $y_T=h*x_T$, with transform $H\,X_T$, returns the value of the convergent case without a limit of $X_T$. This is a formal comparison, not a claim that analytic-functional theories are inconsistent. For $n<0$ (e.g.\ $e^{\alpha/t}$) the obstruction is local at $t=0$ and $e^{\alpha/t}$ has no distributional extension (Section~6.3), so the comparison does not carry over; see Example~11 and Appendix~D.

Against the criterion stated at the outset of this section, the t.g.l.\ formulation handles examples with growth beyond the tempered-distribution setting, including $e^{t^n}$, as well as severe isolated singular behavior such as $e^{t^{-n}}$ for $n>1$, in the sense of the second-stage generalized limit for the forward transform and of pointwise convergence for the inverse transform, respectively. A clarification of how widely the existing functional-analytic approaches discussed above cover such functions, and of the t.g.l.\ formulation's precise mathematical structural relation to those approaches, is left for future study.

\subsection{On the Role of Canonical Convergence in Generalized Fourier Analysis}

When the ordinary Fourier integral fails, meaning may be given to the transform in two broad ways, and the t.g.l.\ formulation differs from both. In a direct approach, the transform is built from the function and the Fourier kernel by integration; the values of $f$ near a singular point and at infinity enter the integral, and when it diverges a damping factor, a principal value, or another regularization is introduced and, as a rule, removed by a limit after the integration, so that the result may depend on the prescription. In a pairing approach, the transform is defined by its action on test functions, $\langle\widehat f,\phi\rangle=\langle f,\widehat\phi\rangle$; the treatment of a singular point is then built into the definition of the functional on the test space (for example, the Cauchy principal value for $1/t$), and the pairing functions belong to the definition. Schwartz theory pairs tempered distributions with rapidly decreasing functions, and Gel'fand--Shilov spaces, ultradistributions, and Fourier hyperfunctions use other test spaces (Section~6.3); each yields a canonical transform within its own space, although no single space accommodates every growth and singularity. In applications the pairing function may be dictated by the problem itself, for example by the transfer function of a linear system. The t.g.l.\ formulation instead uses ordinary finite-stage integrals over a compact set that excludes the singular points, uses no value of $f$ at infinity, introduces no regularization, and does not require the forward family to converge. Pointwise inversion is recovered by the ordered limits away from the singular points, so that functions such as $e^{\alpha t^{2}}$ and $e^{\alpha/t}$, which define no tempered distribution, are treated within elementary calculus, at the price that the forward transform is not a canonical object. Pairings enter only as an optional, derived layer (\Cref{thm:pairing}).

A canonical generalized forward transform is nevertheless not indispensable for every purpose of Fourier analysis, particularly when the objective is the reconstruction of an original function. Increasingly rapid growth or strong local singularities can often be handled by suitably chosen damping factors, such as $e^{-\beta t^{n}}$ for $e^{\alpha t^{n}}$ or $e^{-\beta/t^{n}}$ for $e^{\alpha/t^{n}}$ ($\beta>\alpha$; Examples~6 and~11), but such a choice depends on the class of functions and does not by itself determine a unique generalized forward spectrum. Securing a canonical object for increasingly severe divergence is, in our view, a main reason why the established theories become more elaborate.

The t.g.l.\ formulation adopts a different viewpoint, motivated by the
transition from Fourier series to Fourier integrals (Appendix~A). Instead of
requiring convergence of the forward Fourier transform, it retains the family
of ordinary Fourier transforms obtained at finite truncation stages. The
inverse frequency integration is then completed before the outer time-truncation
limit is taken, with additional prescribed limits where local singularities
occur. Under the stated inversion conditions, this ordered procedure
reconstructs $f$ pointwise at the non-excluded points, without requiring the
finite-stage forward transforms to converge pointwise, in norm, or in any
generalized-function topology.

The t.g.l.\ definition thus involves a trade-off. The price is that the forward transform is not a canonical object: it is the finite-stage family $\{F_{T,\varepsilon}\}$, and uniqueness of a spectrum, and statements about the spectrum itself, are not provided in general. The benefit is that the original function is recovered pointwise at the non-excluded points within elementary calculus, whatever the growth or local singularity of $f$, without any convergence of the forward family. Spectral values are not excluded: for a suitable pairing function the finite-stage family does converge in the generalized-limit sense (\Cref{thm:pairing}), but this holds case by case, is derived and conditional, and is not part of the definition. This does not diminish the importance of generalized-function spaces when a canonical transform, continuity, differentiation, convolution, or other operational properties are required. Rather, it separates two distinct mathematical objectives: the construction of a generalized forward-transform object and the reconstruction of a function from finite-stage Fourier integrals. For the uses of Fourier analysis treated in Section~4, in which a function is mapped to the frequency domain, operated on, and returned, the present work is primarily concerned with the latter, and the canonicality of the forward transform is not needed.

\subsection{Novelty and Usefulness of the t.g.l.\ Formulation}

The principal novelty of the present formulation, to the author's knowledge, lies in its definition of inverse Fourier transformation as an ordered limiting operation on a family of ordinary finite-stage Fourier transforms, without requiring convergence of the forward family; the original function is reconstructed pointwise at the non-excluded points under the local conditions of \Cref{prop:recon}.

Its usefulness lies in providing a constructive approach to Fourier analysis of functions with rapid growth or severe local singularities, using only ordinary integrals and limiting operations of classical calculus. Pointwise reconstruction is obtained without introducing a generalized-function space or a distributional extension at the singularity (for $e^{\alpha/t}$ none exists; Section~6.3), whereas generalized spectral representations, and operations that require the forward family to converge, need additional conditions. For a system with impulse response $h$ driven by $e^{\alpha t^2}$ or $e^{\alpha/t}$, the convolution integral itself may diverge, because the input grows too fast or is not locally integrable; this is a property of the input, not of the t.g.l.\ formulation. The finite-stage output $y_T=h*x_T$ (Section~4) is finite in every case, and in applications the limit need not be taken to completion: one may stop at a stage that is practically sufficient.

\section{Conclusion}

This paper has developed a truncate-and-generalized-limit (t.g.l.) formulation of generalized Fourier analysis based on ordinary finite-stage Fourier integrals and explicitly ordered limiting operations. For each finite truncation the construction remains within classical Fourier analysis, while the removal of the time- and frequency-domain truncations provides the generalized extension. \Cref{prop:recon} collects the constructive statement: a finite-stage family, a prescribed order of inverse limits, and pointwise reconstruction under local conditions, without requiring convergence of the forward family and without asserting a canonical generalized spectrum.

A central feature is the distinction between forward spectral interpretation and inverse reconstruction. The pointwise limit of the forward-transform family need not exist for the original function to be reconstructed. At fixed time truncation, frequency-domain inversion first recovers the truncated function through Dirichlet-kernel localization; the time truncation is removed only afterward. Generalized spectral meaning, when required, is introduced separately through a pairing-based second-stage limit. The t.g.l.\ definition thus gives up the requirement that the forward transform exist as a canonical object, which is, in our view, a main source of complexity in the established theories; in the uses treated in Section~4 it is not needed. In exchange, the original function is recovered pointwise within elementary calculus.

The t.g.l.\ inverse transform may also be viewed, in light of Appendix~A, as a continuous-frequency extension of the classical Fourier series expansion: only the target function is truncated by a finite-interval window, while the orthogonal basis functions $e^{i\omega t}$ and the Fourier-conjugate domains themselves remain infinite, and the generalized limits are taken outside the integral operations.

The examples show that the same construction applies to several classes of functions outside $L^1(\mathbb{R})$: non-decaying functions, isolated singularities, rapidly increasing functions, and oscillatory functions. This includes growth beyond the tempered-distribution setting, such as $e^{\alpha t^2}$, and severe isolated singular behavior, such as the essential singularity of $e^{\alpha/t}$; on common domains the results are compatible with established generalized Fourier interpretations.

In particular, the differentiated-function transform pair (\Cref{thm:deriv}) follows constructively from the delta-compensated boundary cancellation at truncation discontinuities (\Cref{thm:s2one}), rather than from the duality-pairing definition of differentiation in distribution theory.

The main contribution is therefore not the individual transform identities, many of which are classical, but their organization within a common constructive framework based on truncation and ordered reciprocal-domain limits. The scope and mathematical structure of this framework, its precise relation to distributional and complex-analytic theories, and its behavior beyond the tempered-distribution setting remain subjects for further investigation.

\section*{Acknowledgement}

The author owes the motivation for this paper to the late Prof.~Yoichiro Takahashi, former Director of the Research Institute for Mathematical Sciences (RIMS), Kyoto University, and is deeply grateful to him for his inspiring comments.

The author also thanks Emeritus Prof.~A. Yoshikawa and Prof.~K. Kajiwara of Kyushu University for their valuable suggestions and for recommending helpful reference books.

The author is particularly grateful to Prof.~H. Chiba of Tohoku University for his valuable suggestion to include frequency-domain truncation explicitly in the definition of the t.g.l.\ Fourier transform.

The author is deeply grateful to Prof.~H. Feichtinger of the University of Vienna for providing useful comments and introducing important articles on modern Fourier analysis, which improved this manuscript and the author's understanding.

The author acknowledges the use of an AI-based interactive discussion tool during the preparation of this manuscript. The discussions were helpful in improving the English writing, the clarity of exposition of mathematical ideas, and the presentation of related background concepts. All mathematical formulations, interpretations, and conclusions presented herein remain solely those of the author.

This study is partially supported by Microwave Lab Inc.\ in Munakata-city.

\section*{Appendix A: Fourier Integrals as a Transition from Fourier Series}

The Fourier series of a function $f(t)$ can be viewed as an orthogonal-function expansion on the finite domain $t \in [-T,T]$ (private communication from the late Prof.\ Y. Takahashi), as
\[
f(t) = \sum_{n=-\infty}^{\infty} C(n\omega_0)\, e^{in\omega_0 t} \qquad \left(\omega_0 = \frac{2\pi}{2T},\ -T \le t \le T\right),
\]
where $C(n\omega_0)$ denotes the Fourier coefficients, defined by
\[
C(n\omega_0) = \frac{1}{2T} \int_{-T}^{T} f(t)\, e^{-in\omega_0 t}\,dt \qquad (n=0,\pm1,\pm2,\dots).
\]

When $T \gg 1$, the frequency step $\omega_0$ becomes small; we denote it by $\Delta\omega$. Then we obtain
\[
f(t) = \sum_{n=-\infty}^{\infty} C(n\Delta\omega)\, e^{in\Delta\omega t}, \qquad
C(n\Delta\omega) = \frac{1}{2T} \int_{-T}^{T} f(t)\, e^{-in\Delta\omega t}\,dt.
\]
As $T \to \infty$, the variable $n\Delta\omega$ becomes continuous; we denote it by $\omega$. Thus we obtain
\[
C(\omega) = \lim_{T\to\infty} C(n\Delta\omega) = \lim_{T\to\infty} \frac{1}{2T}\, F_T(\omega),
\]
where
\[
F_T(\omega) = \int_{-T}^{T} f(t)\, e^{-i\omega t}\,dt.
\]
The Fourier transform is defined by $F(\omega) = \lim_{T\to\infty} F_T(\omega)$. We call $C(\omega)$ the \emph{spectrum function}, and $F(\omega)$ the \emph{spectral density function}. Let $C_T(\omega) = C(n\Delta\omega) = F_T(\omega)/2T$, $\omega = 2\pi f$, and $df = 1/2T$; then $F_T(2\pi f) = C_T(2\pi f)/df$. If $t$ denotes time and $f(t)$ a physical variable such as voltage, then $C_T(2\pi f)$ has the unit of volts, while $F_T(2\pi f)$ has the unit of volts per hertz.

The Fourier series approaches the Fourier integral as
\begin{align*}
f(t) &= \lim_{T\to\infty} \sum_{n=-\infty}^{\infty} C(n\Delta\omega)\, e^{in\Delta\omega t}
= \lim_{T\to\infty} \frac{1}{\Delta\omega} \sum_{n=-\infty}^{\infty} C(n\Delta\omega)\, e^{in\Delta\omega t}\, \Delta\omega \\
&\approx \lim_{T\to\infty} \frac{1}{2\pi} \int_{-\infty}^{\infty} F_T(\omega)\, e^{i\omega t}\,d\omega.
\end{align*}
The approximation symbol above reflects two distinct convergence requirements that are silently coupled through the single parameter $T$: the discrete sum must converge to the continuous integral as the frequency spacing $\Delta\omega = \pi/T$ shrinks (a Riemann-sum-to-integral limit), while at the same time the truncated transform $F_T(\omega)$ must converge to $F(\omega)$, as discussed below. Classical treatments do not distinguish these two requirements, since both are driven by the same limit $T\to\infty$. The t.g.l.\ formulation of the present paper instead separates such roles into explicitly named and independently ordered parameters ($T_1,T_2\to\infty$, $\Omega\to\infty$, $\varepsilon_1,\varepsilon_2\to0$), as developed in Section~2.6 and \Cref{thm:noninterchange}.

For non-absolutely integrable functions, the limit $T\to\infty$ cannot be moved inside the integrand, because $\lim_{T\to\infty} F_T(\omega)$ does not converge. However, for absolutely integrable functions this interchange is valid and leads to the ordinary definition of the inverse Fourier transform:
\[
f(t) = \lim_{T\to\infty} \frac{1}{2\pi} \int_{-\infty}^{\infty} F_T(\omega)\, e^{i\omega t}\,d\omega
= \frac{1}{2\pi} \int_{-\infty}^{\infty} \lim_{T\to\infty} F_T(\omega)\, e^{i\omega t}\,d\omega
= \frac{1}{2\pi} \int_{-\infty}^{\infty} F(\omega)\, e^{i\omega t}\,d\omega.
\]

The traditional treatment yields the same result by regarding $f(t)$ as a periodic function of period $2T$,
\[
f(t) = \sum_{n=-\infty}^{\infty} f_T(t-2nT) \qquad (-\infty<t<\infty),
\]
where $f_T(t)=0$ for $|t|>T$, and then taking $T\to\infty$. The extension of the domain of $f(t)$ as a periodic function from $t\in[-T,T]$ to $t\in(-\infty,\infty)$ is natural, since trigonometric functions are themselves periodic. Nevertheless, one may feel that starting from $t\in[-T,T]$, as above, is the more direct route.

If the variable $f$ is used in place of $\omega$, the factor $1/2\pi$ disappears from the inverse Fourier transform. In that case, the alternative convention using $1/\sqrt{2\pi}$ for both the forward and inverse transforms loses its symmetrical appearance. We therefore prefer the convention adopted in this paper.

The Fourier series expansion rests on the orthogonality of trigonometric functions on a finite interval,
\[
\frac{1}{2T}\int_{-T}^{T} e^{im\omega_0 t}\, e^{-in\omega_0 t}\,dt = \delta[m,n] \qquad (m,n=0,\pm1,\pm2,\dots),
\]
where the Kronecker delta $\delta[m,n]=1$ for $m=n$ and $0$ otherwise. As $T\to\infty$, writing $m\omega_0=\omega_m$ and $n\omega_0=\omega_n$, we obtain
\[
\lim_{T\to\infty} \frac{1}{2T}\int_{-T}^{T} e^{i\omega_m t}\, e^{-i\omega_n t}\,dt
= \lim_{T\to\infty} \frac{\pi}{T}\, \delta_T(\omega_m-\omega_n) = \delta[m,n], \qquad \delta_T(\omega) = \frac{\sin\omega T}{\pi\omega}.
\]
Thus $\lim_{T\to\infty} \pi\,\delta_T(\omega_m-\omega_n)/T$, which involves the delta-convergent sequence of functions $\delta_T$, expresses the orthogonality of $\exp(i\omega t)$ on the infinite interval.

It is worth noting that the t.g.l.\ integral truncates the \emph{functions} themselves rather than the time domain; the transforms are accordingly defined through integrals rather than through the summations characteristic of Fourier series.

\section*{Appendix B: Derivation of the Truncated Chirp Signal Reciprocal Relationship}

The Fourier transform of the truncated linear chirp $s_T(t)=\exp(i\pi kt^2)\,\mathbf{1}_{[0,T]}(t)$ is
\[
F_T(\omega)=\int_0^T e^{i(\pi k t^2-\omega t)}\,dt,
\]
a truncated Fresnel-type integral with no elementary closed form for finite $T$ (Appendix~C). Its effective bandwidth is therefore characterized by stationary phase [16]. The phase $\theta(t)=\pi kt^2-\omega t$ has $\theta'(t)=2\pi kt-\omega$, so the stationary point $t_s=\omega/(2\pi k)$ contributes only when $0\le t_s\le T$, that is, $0\le\omega\le2\pi kT$. We define the effective angular-frequency bandwidth $\Omega\approx2\pi kT$, or $B=\Omega/(2\pi)\approx kT$.

The matched-filter output is the autocorrelation $R_T(\tau)=\int_0^T s_T(t)\,s_T^{*}(t-\tau)\,dt$. Since $t^2-(t-\tau)^2=2t\tau-\tau^2$, we have $s_T(t)s_T^{*}(t-\tau)=e^{-i\pi k\tau^2}e^{i2\pi k\tau t}$, and
\[
R_T(\tau)=T\,e^{i\pi k\tau(T-\tau)}\,\frac{\sin(\pi kT\tau)}{\pi kT\tau},\qquad |R_T(\tau)|\approx T\,|\mathrm{sinc}(kT\tau)|,\quad \mathrm{sinc}(x)=\frac{\sin\pi x}{\pi x}.
\]
The first zero is at $\tau=1/(kT)$. With the effective time resolution $\Delta t=2\tau=2/(kT)$ and $B\approx kT$,
\[
B\,\Delta t\approx2,\qquad\text{or equivalently}\qquad \Omega\,\Delta t\approx4\pi.
\]
This relation is specific to the present truncated-chirp model rather than a universal time-bandwidth relation; it shows that increasing the pulse duration $T$ increases the effective signal bandwidth, improving the time resolution after signal synthesis.

\section*{Appendix C: Derivation of the Fourier Transform of a Chirp Signal}

Section~5.4 states the generalized Fourier transform of $f(t)=\exp(i\pi t^2)$ without derivation. This appendix derives it for the general chirp $f(x)=e^{iax^2}$ $(a>0)$; $a=\pi$ gives Section~5.4. Unlike the elementary examples of the main text, the truncated Fourier integral of a chirp has no elementary closed form at any finite truncation stage; the t.g.l.\ formulation is applied directly to this truncated Fresnel integral, not to an antiderivative.

\emph{Forward transform.} Since $f\notin L^1(\mathbb{R})$, define $F_T(\omega)=\int_{-T}^{T}e^{iax^2}e^{-i\omega x}\,dx$. Completing the square, $ax^2-\omega x=a(x-\omega/2a)^2-\omega^2/4a$, and substituting $u=x-\omega/2a$,
\[
F_T(\omega)=e^{-i\omega^2/4a}\int_{-T-\omega/2a}^{\,T-\omega/2a}e^{iau^2}\,du .
\]
As $T\to\infty$ the integral tends to the classical Fresnel integral $\int_{-\infty}^{\infty}e^{iau^2}du=\sqrt{\pi/a}\,e^{i\pi/4}$, so
\[
F(\omega)=\sqrt{\frac{\pi}{a}}\,e^{i\pi/4}\,e^{-i\omega^2/4a},
\]
which agrees with the classical generalized Fourier transform; $a=\pi$ recovers $F(\omega)=e^{i\pi/4}e^{-i\omega^2/4\pi}$, the result stated in Section~5.4.

\emph{Inverse transform.} For finite $T$ and $\Omega$ both integrations are over bounded intervals and may be interchanged:
\[
\frac{1}{2\pi}\int_{-\Omega}^{\Omega}F_T(\omega)\,e^{i\omega t}\,d\omega=\int_{-T}^{T}e^{iax^2}\,\frac{\sin\Omega(t-x)}{\pi(t-x)}\,dx .
\]
The prescribed order of limits, $\Omega\to\infty$ followed by $T\to\infty$, yields $f(t)=e^{iat^2}$.

\emph{Independent check.} Since $F(\omega)$ is itself a quadratic-phase function of the same type as $f$, the untruncated inverse integral converges directly by oscillatory cancellation. Completing the square, $-\omega^2/4a+\omega t=-(\omega-2at)^2/4a+at^2$, and using the conjugate Fresnel integral $\int_{-\infty}^{\infty}e^{-iv^2/4a}dv=\sqrt{4\pi a}\,e^{-i\pi/4}$,
\[
\frac{1}{2\pi}\int_{-\infty}^{\infty}F(\omega)\,e^{i\omega t}\,d\omega=\frac{\sqrt{\pi/a}\sqrt{4\pi a}}{2\pi}\,e^{i\pi/4}e^{-i\pi/4}e^{iat^2}=f(t).
\]
This second route exists because the chirp is self-similar under the Fourier transform; it does not show that the ordered-limit structure is dispensable, since most admissible functions do not share this self-similarity.

The derivation illustrates three points. The finite-truncation integral $F_T(\omega)$ is not evaluated in closed form. The forward transform converges as an ordinary oscillatory limit, so no second-stage pairing (\Cref{thm:pairing}) is needed. The inverse requires truncation in both domains, with the ordered limits $\Omega\to\infty$ followed by $T\to\infty$ taken through the localization of the Dirichlet kernel; the generalized character of the reconstruction arises from these ordered limits themselves rather than from generalized functionals. The same result follows from the distributional route, since $e^{iax^2}\in L^\infty(\mathbb{R})\subset\mathcal{S}'(\mathbb{R})$.

\section*{Appendix D: Complex-Time and Complex-Frequency Interpretation of the Ordered-Limit Construction}

This appendix gives a complementary complex-domain interpretation of the ordered limiting structure used in the t.g.l.\ Fourier formulation; it is not required for the formulation's definition or application. Through the examples below, it illustrates how related truncation and limiting structures arise when divergent real-axis Fourier integrals are interpreted by analytic continuation and contour deformation. The comparison is structural, not an assertion of general equivalence between the t.g.l.\ formulation and complex-domain generalized Fourier theories. Complex-time contour integration and analytic-functional formulations are related through complex analyticity and contour representations, but they are not identical constructions, and no general equivalence is asserted here.

The function $e^{\alpha t^2}$ ($\alpha>0$), introduced in Example~6, is used as the principal example.

The choice to complexify time rather than frequency is not arbitrary. The Laplace transform complexifies frequency, $\omega\to s=\sigma+i\omega$, which gives the kernel $e^{-st}=e^{-\sigma t}e^{-i\omega t}$ a damping factor linear in $t$; this tames functions of at most exponential order, but not $e^{\alpha t^2}$, since $\alpha t^2-\sigma t\to\infty$ for every fixed $\sigma$. Complexifying time instead reaches this case, because the damping is not fixed at linear order: with $t=\rho e^{i\phi}$, the rotation $\phi=\pi/n$ sends $t^n$ to $-\rho^n$, so $e^{\alpha t^n}=e^{-\alpha\rho^n}$ decays for any $n$. The substitution $t=is$ used below is the $n=2$ case.

For $f(t)=e^{\alpha t^2}$, the ordinary Fourier integral on the real time axis diverges, but rotating time to the imaginary axis, $t=is$ ($s$ real), gives $e^{\alpha t^2}=e^{-\alpha s^2}$, turning the rapidly increasing function into a decreasing Gaussian; frequency may be rotated analogously, $\omega=-is$, as in the companion computation of Example~6's Remark. Since $e^{-i\omega t}$ is entire in both variables, this suggests considering suitable contours in the complexified domains. Complex-domain formulations are not new --- Kawai [23] gives an established and more sophisticated precedent --- and the present observation is only an elementary interpretation suggested by Example~6, not a claim of novelty or an attempt to reproduce Kawai's theory.

The obstruction for $e^{\alpha t^2}$ occurs at infinity, not on any finite interval: for every finite $T$, $f_T(t)=e^{\alpha t^2}\mathbf{1}_{[-T,T]}(t)$ is ordinary and integrable, and $F_T(\omega)=\int_{-T}^{T}e^{\alpha t^2}e^{-i\omega t}\,dt$ is an ordinary finite-domain Fourier integral, as for any t.g.l.-admissible function. The t.g.l.\ construction completes the inverse frequency-domain limit at this finite stage, recovering $f_T$, before the outer truncation $T\to\infty$ is removed (Section~6.1); it does not require $\lim_{T\to\infty}F_T(\omega)$ to exist first (Section~6.2). This independence is what lets the ordered-limit construction reach $e^{\alpha t^2}$ without placing it in an enlarged generalized-function space --- an operational distinction, not a claim that the t.g.l.\ formulation is more general than the theories discussed above.

A systematic theory in which both Fourier-dual variables are complex and their integration contours form part of the transform definition would require admissible contour classes, growth conditions, treatment of singularities, and a corresponding inversion theorem; the divergent real-axis integral for $e^{\alpha t^2}$ cannot simply be identified with an imaginary-axis contour integral by an ordinary deformation, since contributions from the connecting arcs cannot automatically be discarded.

The contour-shift technique used below --- displacing a Gaussian integral off the real axis and relating it back via Cauchy's theorem --- is a classical device in complex analysis, standard in evaluating integrals such as $\int_{-\infty}^{\infty}e^{-x^2}\cos(2bx)\,dx$ [25]; here it is adapted, through the order $y_T$ then $x_T$, to the divergent case $e^{\alpha t^2}$. Consider the quarter-rectangle contour with vertices $0,\,x_T,\,x_T+iy_T,\,iy_T$ (Fig.~3, contour $C_1$), and its mirror image $C_2$ in the third quadrant. Taking $y_T\to\infty$ first, at fixed $x_T$, the top and bottom edges ($C_{13}$, $C_{23}$) vanish by uniform decay over the fixed finite range of the real part, while the side edges ($C_{12}$, $C_{24}$) do not, each growing unboundedly as $x_T\to\infty$. Completing the square and shifting the resulting Gaussian contour back to the real axis --- legitimate since the integrand is entire and decays along the shift --- gives, for every finite $x_T$,
\[
\int_0^{x_T} e^{\alpha t^2-i\omega t}\,dt \;+\; \lim_{y_T\to\infty}\int_{C_{12}} e^{\alpha t^2-i\omega t}\,dt \;=\; i\int_0^{\infty} e^{-\alpha y^2+\omega y}\,dy,
\]
an exact identity, not an asymptotic one: the right-hand side is fixed and finite, while the two terms on the left individually diverge, at exactly compensating rates, as $x_T\to\infty$. Separating real and imaginary parts makes this transparent: since the right-hand side is purely imaginary,
\[
\mathrm{Re}\!\left[\int_0^{x_T} e^{\alpha t^2-i\omega t}\,dt\right] \;+\; \mathrm{Re}\!\left[\lim_{y_T\to\infty}\int_{C_{12}} e^{\alpha t^2-i\omega t}\,dt\right] \;=\; 0
\]
exactly, for every finite $x_T$, while the imaginary parts together sum to the finite value on the right. The mirror identity holds on the third-quadrant contour, and the two constants together reproduce the value obtained formally in Example~6's Remark.

This shows that the divergence of the side edges is exactly the negative of the divergence of the real-axis integral, leaving no remainder. Writing $I(x_T,y_T)$ for the total integral of $e^{\alpha t^2-i\omega t}$ around $C_1$, the prescribed order is
\[
\lim_{x_T\to\infty}\Big[\lim_{y_T\to\infty} I(x_T,y_T)\Big]:
\]
the inner limit, at each fixed finite $x_T$, completes the cancellation between the real-axis integral and the side edge, and only then is $x_T\to\infty$ taken. This is specific to this contour and this order of limits: for the square contour ($x_T=y_T\to\infty$ jointly) the cancellation is incomplete, leaving a further divergent contribution, and neither reversing the order nor taking the simultaneous limit is part of the present construction or need give the same result.

\begin{figure}[!htbp]
\centering
\includegraphics[width=0.6\textwidth]{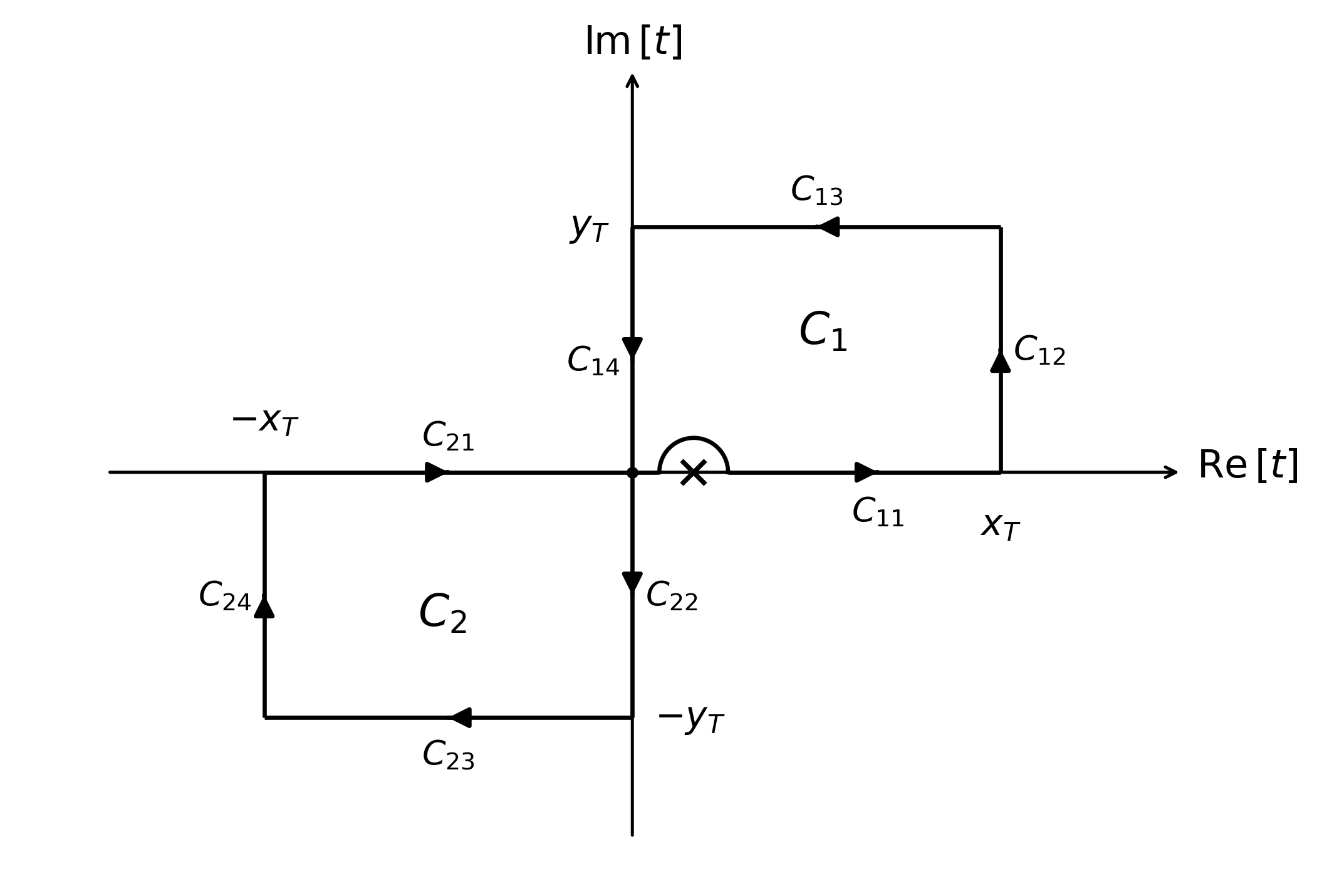}
\caption{The quarter-rectangle contour $C_1$ in the first quadrant and its mirror image $C_2$ in the third quadrant, used in the complex-domain analysis of Example~6.}
\end{figure}

Written in Cartesian coordinates, with the two real-axis pieces merged into one symmetric integral, the identity becomes
\[
\int_{-x_T}^{x_T} e^{\alpha x^2-i\omega x}\,dx \;-\; F_{\mathrm{div}}(x_T,\omega) \;=\; i\int_{-\infty}^{\infty} e^{-\alpha y^2+\omega y}\,dy \;=\; i\sqrt{\pi/\alpha}\,e^{\omega^2/(4\alpha)},
\]
where
\[
F_{\mathrm{div}}(x_T,\omega) \;=\; -i\int_0^{\infty} e^{\alpha(x_T+iy)^2-i\omega(x_T+iy)}\,dy \;-\; i\int_{-\infty}^{0} e^{\alpha(-x_T+iy)^2-i\omega(-x_T+iy)}\,dy,
\]
Here $y_T\to\infty$ is written directly as $\infty$: for each fixed finite $x_T$, both integrands decay like $e^{-\alpha y^2}$, so each converges absolutely and the outer limit is no longer needed. The last equality is the standard closed-form Gaussian evaluation, reducing to $i\sqrt{\pi}\,e^{\omega^2/4}$ when $\alpha=1$, as in Example~6's Remark. A divergent real-axis integral, minus a divergent compensating contour term retained until the limit is taken, thus equals this explicit finite value.

The right-hand side $i\sqrt{\pi/\alpha}\,e^{\omega^2/(4\alpha)}$ comes entirely from the inner edges $C_{14},C_{22}$, which lie on the imaginary axis and never involve $x_T$; the identity thus holds with this same constant for every finite $x_T$, not merely in the limit. Sending $x_T\to\infty$ leaves the right-hand side untouched, while neither term on the left has a limit on its own: both the real-axis integral and $F_{\mathrm{div}}(x_T,\omega)$ grow like $e^{\alpha x_T^2}$ at exactly matching rate and phase, so $\lim_{x_T\to\infty}$ applies only to their difference, never termwise. The finite value is thus an exact invariant already present at every finite truncation, not an asymptotic outcome; the outer limit merely confirms this stability rather than producing it.

The subtraction made explicit above resembles Hadamard finite-part integration, in that divergent contributions are retained together so that they compensate before the final limit. It should not, however, be identified with the Hadamard finite part of the original real-axis integral. For real $\omega$ the symmetrically truncated integral $\int_{-x_T}^{x_T}e^{\alpha t^2-i\omega t}\,dt$ is real for every finite $x_T$, whereas the convergent imaginary-axis integral is purely imaginary; and the divergence here is exponential ($e^{\alpha x_T^2}$) rather than algebraic, so ordinary finite-part subtraction does not apply. The resemblance concerns the mechanism of cancellation, not equality of the resulting values.

This ordered-limit structure parallels the t.g.l.\ formulation itself, where the prescribed order of the two truncation limits is similarly essential (Section~3, \Cref{thm:noninterchange}): the frequency-domain limit $\Omega\to\infty$ is completed at each fixed time-domain truncation $T$, and only then is $T\to\infty$ taken,
\[
f(t) = \lim_{T\to\infty}\Big[\lim_{\Omega\to\infty} I_{T,\Omega}(t)\Big],
\]
with the reversed order failing in general. The complex-contour construction above follows the same pattern, with $y_T$ playing the role of $\Omega$: one truncation is retained while the other variable's operation is completed, and only then is the retained truncation removed. The two constructions are not claimed to be mathematically identical --- one real-variable, the other complex-analytic --- but both illustrate the same principle: the prescribed order of removing truncations may be essential to the generalized interpretation, not incidental. Its precise relation to Kawai's Fourier hyperfunctions and to Gel'fand--Shilov-type theories is left for future study.

For $e^{\alpha t}$, the obstruction is global exponential growth rather than a finite-plane singularity, since $e^{\alpha z}$ is entire. Complex-frequency formulations accommodate this by analytic continuation or contour displacement in the frequency variable, as discussed in Example~5 with Gel'fand--Shilov and Titchmarsh. The t.g.l.\ formulation instead retains ordinary finite real-time Fourier integrals and reconstructs through the prescribed limiting construction given there; further details are in Example~5.

In these cases the complex-domain constructions illustrate the role of an
ordered limit, but they remain interpretations: the definition and the
reconstruction of the t.g.l.\ formulation use only ordinary finite-stage
integrals on the real axis and do not depend on them.

\end{document}